\documentclass[10pt,reqno]{amsart}
\usepackage{hyperref}
\usepackage{latexsym,amsmath}
\usepackage{enumerate}
\usepackage{amsfonts}
\usepackage{amssymb}
\usepackage{latexsym}
\usepackage{fixmath}
 \usepackage[usenames,dvipsnames]{color}
\usepackage{graphicx}

\usepackage[T1]{fontenc}
\usepackage{fourier}
\usepackage{bbm}
\usepackage{color}
\usepackage{fullpage}

\numberwithin{equation}{section}

\renewcommand{\vec}[1]{\boldsymbol{#1}}

\renewcommand{\subset}{\subseteq}

\newcommand\disteq{\,\sim\,}

\newcommand\vC{\vec C}

\newcommand\vX{\vec X}
\newcommand\vw{\vec w}

\newcommand\vk{\vec k}
\newcommand\vd{\vec d}
\newcommand\vy{\vec y}

\newcommand\vx{\vec x}

\newcommand\va{\vec a}

\newcommand\THETA{\vec\theta}

\newcommand\CPC{Combinatorics, Probability and Computing}

\newcommand{\GG}{\G}

\newcommand\ALPHA{\vec\alpha}
\newcommand\MU{\vec\mu}

\newcommand\vm{{\vec m}}

\newcommand\CHI{{\vec\chi}}
\newcommand\DELTA{{\vec\Delta}}
\newcommand\GAMMA{{\vec\gamma}}

\newcommand\vA{\vec A}

\newcommand\G{\vec G}

\newcommand\SIGMA{\vec\sigma}

\newcommand\aco[1]{#1}

\newcommand\fF{\mathfrak{F}}
\newcommand\cA{\mathcal{A}}
\newcommand\cB{\mathcal{B}}
\newcommand\cC{\mathcal{C}}
\newcommand\cD{\mathcal{D}}
\newcommand\cF{\mathcal{F}}

\newcommand\cE{\mathcal{E}}
\newcommand\cU{\mathcal{U}}
\newcommand\cN{\mathcal{N}}
\newcommand\cQ{\mathcal{Q}}
\newcommand\cH{\mathcal{H}}
\newcommand\cS{\mathcal{S}}

\newcommand\cK{\mathcal{K}}

\newcommand\cL{\mathcal{L}}
\newcommand\cM{\mathcal{M}}

\newcommand\cP{\mathcal{P}}
\newcommand\cX{\mathcal{X}}
\newcommand\cY{\mathcal{Y}}
\newcommand\cV{\mathcal{V}}
\newcommand\cW{\mathcal{W}}
\newcommand\cZ{\mathcal{Z}}

\def\cC{{\mathcal C}}
\def\cE{{\mathcal E}}

\newcommand\eps{\varepsilon}

\newcommand\ZZ{\mathbb{Z}}
\newcommand\FF{\mathbb{F}}
\newcommand\NN{\mathbb{N}}

\newcommand\Var{\mathrm{Var}}
\newcommand\Erw{\mathbb{E}}
\newcommand{\vecone}{\vec{1}}

\newcommand{\Po}{{\rm Po}}
\newcommand{\Bin}{{\rm Bin}}

\newcommand\TV[1]{\left\|{#1}\right\|_{\mathrm{TV}}}

\newcommand\dTV{d_{\mathrm{TV}}}
\newcommand{\bink}[2] {{\binom{#1}{#2}}}
\newcommand\bc[1]{\left({#1}\right)}
\newcommand\cbc[1]{\left\{{#1}\right\}}
\newcommand\bcfr[2]{\bc{\frac{#1}{#2}}}

\newcommand\brk[1]{\left\lbrack{#1}\right\rbrack}

\newcommand\abs[1]{\left|{#1}\right|}

\newcommand\RR{\mathbb{R}}

\newcommand{\Whp}{W.h.p.}
\newcommand{\whp}{w.h.p.}

\newcommand{\tensor}{\otimes}

\newcommand{\Komlos}{Koml\'os}

\newcommand\pr{\mathbb{P}} 
 
\newcommand\Lem{Lemma}
\newcommand\Prop{Proposition}
\newcommand\Thm{Theorem}

\newcommand\Cor{Corollary}
\newcommand\Sec{Section}
\newcommand\Chap{Chapter}

\newtheorem{definition}{Definition}[section]
\newtheorem{claim}[definition]{Claim}
\newtheorem{example}[definition]{Example}
\newtheorem{remark}[definition]{Remark}
\newtheorem{theorem}[definition]{Theorem}
\newtheorem{lemma}[definition]{Lemma}
\newtheorem{proposition}[definition]{Proposition}
\newtheorem{corollary}[definition]{Corollary}

\newtheorem{fact}[definition]{Fact}

\DeclareMathOperator{\nul}{nul}
\DeclareMathOperator{\rank}{rk}
\DeclareMathOperator{\rate}{rate}

\newcommand{\rk}[1]{\rank(#1)}

\newcommand{\supp}[1]{{\text{supp}\left(#1\right)}}
\newcommand\A{\vA}

\def\D{{\mathcal D}}
\def\E{{\mathcal E}}

\def\N{{\mathcal N}}

\def\P{{\mathcal P}}

\def\po{{\bf Po}}
\def\ex{{\mathbb E}}
\def\pr{{\mathbb P}}

\def\bfd{{\vec d}}

\def\bfk{\vk}

\def\bfm{{\vec m}}
\def\bfn{{\vec n}}

\def\bfG{{\bf G}}

\def\bftheta{{\pmb{\theta}}}

\def\cH{{\mathcal H}}

\newcommand{\remove}[1]{}
\newcommand\eqn[1]{(\ref{#1})}

\newcommand{\be}{\begin{equation}}
\newcommand{\bel}[1]{\begin{equation}\lab{#1}\ }
\newcommand{\ee}{\end{equation}}
\newcommand{\bea}{\begin{eqnarray}}
\newcommand{\eea}{\end{eqnarray}}
\newcommand{\bean}{\begin{eqnarray*}}
\newcommand{\eean}{\end{eqnarray*}}

\newcommand{\vn}{{\vec n}}
\newcommand{\vM}{\vec M}
\newcommand{\vN}{\vec N}
\newcommand{\fD}{\mathfrak D}

\begin{document}

\title{The rank of random matrices over finite fields}

\author{Amin Coja-Oghlan, Pu Gao}
\thanks{Gao's research is supported by ARC DE170100716 and ARC DP160100835.}

\address{Amin Coja-Oghlan, {\tt acoghlan@math.uni-frankfurt.de}, Goethe University, Mathematics Institute, 10 Robert Mayer St, Frankfurt 60325, Germany.}

\address{Pu Gao, {\tt jane.gao@monash.edu},   School of Mathematical Sciences,    Monash University,  Australia}

\begin{abstract}
We determine the rank of a random matrix $\vA$ over a finite field with prescribed numbers of non-zero entries in each row and column.
As an application we obtain a formula for the rate of low-density parity check codes.
This formula verifies a conjecture of Lelarge [Proc.\ IEEE Information Theory Workshop 2013].
The proofs are based on coupling arguments and the interpolation method from mathematical physics.
\hfill
%\medskip\noindent
{\em MSC:} 05C80, 	60B20, 	94B05 
\end{abstract}

\maketitle

\section{Introduction}\label{Sec_intro}

\subsection{Background and motivation}
Random matrices over finite fields count among the most basic objects of probabilistic combinatorics.
Their study goes back to the early days of the discipline~\cite{ERmatrices}.
More recently they have been at the centre of an exciting development in coding theory.
The subject of a tremendous amount of research over the past 20 years, {\em low density parity check} (`ldpc') codes have become a mainstay of modern communications standards; you probably carry some around in your pocket~\cite{RichardsonUrbanke}.
The codebook of an ldpc code comprises the kernel of a sparse random matrix over a finite field.
Celebrated recent results establish that ldpc codes meet the Shannon bound, i.e., that they are information-theoretically optimal~\cite{GMU}.
The practical relevance of these results derives from the fact that ldpc codes based on matrices drawn from carefully tailored distributions even admit efficient decoding algorithms~\cite{KYMP,RichardsonUrbanke}.
In addition, sparse random matrices over finite fields have been studied extensively in the theory of random constraint satisfaction problems, e.g.,~\cite{AchlioptasMolloy,Dietzfelbinger,DuboisMandler,Ibrahimi,PittelSorkin}.

Despite the great interest in the subject certain fundamental questions remained open.
The most obvious one concerns the rank.
Although this parameter was already studied in early contributions~\cite{Balakin1,Balakin2,Kovalenko}, there has been no general formula for the rank of sparse random matrices, where the number of non-zero entries grows linearly with the number of rows.
The present paper delivers such a formula.
To be precise, we will determine the rank of a sparse random matrix with prescribed row and column degrees (viz.\ number of non-zero entries).
Ldpc codes are based on precisely such random matrices as a diligent choice of the degrees greatly boosts the code's performance~\cite{RichardsonUrbanke}.
The rank of the random matrix is directly related to the {\em rate} of the ldpc code, defined as the nullity of the matrix divided by the number of columns, arguably the most basic parameter of any linear code.
%The rate formula that we establish implies a conjecture of  
Lelarge~\cite{Lelarge} noticed that an upper bound on the rank of the random matrix, and thus a lower bound on the rate of ldpc codes, 
follows from the result on the matching number of random bipartite graphs from~\cite{BLS2}.
He conjectured the bound to be tight.
We prove this conjecture.

In fact, there is an interesting twist.
Lelarge observed that a sophisticated but mathematically non-rigorous approach from statistical physics called the cavity method renders a wrong `prediction' as to the rank for certain degree distributions.%
	\footnote{The derivation of this erroneous prediction was posed as an exercise in~\cite[\Chap~19]{MM}.}
This discrepancy merits attention because the cavity method has by now been brought to bear on a very wide range of practically relevant problems, ranging from signal processing to machine learning~\cite{LF}.
The proof of the rank formula that we develop sheds light on the issue.
Specifically, the `replica symmetric' version of the cavity method predicts that the rank can be expressed analytically as the optimal solution to a variational problem.
A priori, this variational problem asks to optimise a functional called the Bethe free entropy over an infinite-dimensional space of probability measures.
Such problems have been tackled in the physics literature numerically by means of a heuristic called population dynamics.
For the rank problem this was carried out by Alamino and Saad~\cite{AlaminoSaad}.
But thanks to the algebraic nature of the problem we will be able to dramatically simplify the variational problem, arriving at a humble one-dimensional optimisation task.
The main result of the paper shows that the optimal solution to this one-dimensional problem does indeed yield the rank.
The formula matches Lelarge's conjecture.
Furthermore, the solution can be lifted to a solution to the original infinite-dimensional problem.
For certain degree distributions the solution thus obtained is of a new type, different from the solutions that surfaced in the experiments from~\cite{AlaminoSaad} or the heuristic derivations from~\cite{MM}.
The fact that the heuristic approaches missed the actual optimiser explains the discrepancy between the physics predictions and mathematical reality.

In the following paragraphs we will introduce the model and state the main results.
A discussion of related work and a detailed comparison with the physics predictions follow in \Sec~\ref{Sec_related}.

\subsection{The rank formula}
Let $q\geq2$ be a prime power and let $\CHI$ be an $\FF_q^*=\FF_q\setminus\cbc 0$-valued random variable.
Moreover, let $\vd\geq1,\vk\geq3$ be integer-valued random variables such that $\Erw[\vd^r]+\Erw[\vk^r]<\infty$ for a real $r>2$ and set $d=\Erw[\vd]$, $k=\Erw[\vk]$.
Let $n>0$ be an integer divisible by the greatest common divisor of the support of $\vk$ and let $\vm\disteq\Po(dn/k)$.
Further, let $(\vd_i,\vk_i,\CHI_{i,j})_{i,j\geq1}$ be copies of $\vd,\vk$, $\CHI$, respectively, mutually independent and independent of $\vm$.
Given
\begin{equation}\label{eqWellDef1}
\sum_{i=1}^n\vd_i=\sum_{i=1}^{\vm}\vk_i,
\end{equation}
draw a simple bipartite graph $\G$ comprising a set $\{a_1,\ldots,a_{\vm}\}$ of {\em check nodes} and a set $\{x_1,\ldots,x_n\}$ of 
{\em variable nodes} such that the degree of $a_i$ equals $\vk_i$ and the degree of $x_j$ equals $\vd_j$ for all $i,j$ uniformly at random.
Then let $\A$ be the $\vm\times n$-matrix with entries
\begin{align*}
\A_{ij}&=\vecone\{a_ix_j\in E(\G)\}\cdot\CHI_{i,j}. %&(i\in[\vm],\,j\in[n]).
\end{align*}	
Thus, the $i$'th row of $\A$ contains precisely $\vk_i$ non-zero entries and the $j$'th column contains precisely $\vd_j$ non-zero entries.
Standard arguments show that $\A$ is well-defined for large enough $n$, i.e., \eqref{eqWellDef1} is satisfied and there exists a simple $\G$ with the desired degrees with positive probability; see \Prop~\ref{Lemma_welldef}.
We call $\G$ the {\em Tanner graph} of $\A$.

Since $\vd,\vk$ have finite means the matrix $\A$ is sparse, i.e., the expected number of non-zero entries is $O(n)$.
Yet because the degree distributions are subject  only to the modest condition $\Erw[\vd^r]+\Erw[\vk^r]<\infty$, the typical maximum number of non-zero entries per row or column may be as large as $n^{1/2-\Omega(1)}$.
Natural choices of $\vd$ and $\vk$ include one-point distributions, truncated Poisson distributions as well as power laws.
Additionally, clever choices of $\vd$ and $\vk$ that facilitate the construction of error-correcting codes have been proposed~\cite{RichardsonUrbanke}.

The following theorem, the main result of the paper, provides a formula for the asymptotic rank of $\A$.
Let $D(x)$ and $K(x)$ denote the probability generating functions of $\vd$ and $\vk$, respectively.
Since $\Erw[\vd^2]+\Erw[\vk^2]<\infty$, the functions $D(x),K(x)$ are continuously differentiable on $[0,1]$.

\begin{theorem}\label{thm:rank}
Let
\begin{align*}
\Phi(\alpha)&=D\left(1-K'(\alpha)/k\right)+\frac{d}{k}\bc{K(\alpha)+(1-\alpha)K'(\alpha)-1}.
\end{align*}
Then
	$$\lim_{n\to\infty}\frac{\rank(\A)}n=1-\max_{\alpha\in[0,1]}\Phi(\alpha)\qquad\mbox{in probability.}$$
\end{theorem}

\noindent
The upper bound $\rank(\A)/n\leq1-\max_{\alpha\in[0,1]}\Phi(\alpha)+o(1)$ \whp\ was  previously derived by Lelarge~\cite{Lelarge} from the Leibniz determinant formula and the formula for the matching number of a random bipartite graphs from~\cite{BLS2}.
Thus, the lower bound on the rank constitutes the main contribution of this paper.
Nonetheless, we will also give an independent proof of the upper bound, which is more direct and
signifincantly shorter than~\cite{BLS2,Lelarge}.%
\footnote{Lelarge deals with degrees s.t.\ $\Erw[\vd^2],\Erw[\vk^2]<\infty$, whereas
we assume $\Erw[\vd^r],\Erw[\vk^r]<\infty$ for an $r$ arbitrarily close to but greater than 2.}

\subsection{The rate of ldpc codes}
\Thm~\ref{thm:rank} implies a universal formula for the rate of ldpc codes.
To be precise, the standard construction of ldpc codes assumes that $\vd,\vk$ are bounded random variables~\cite{RichardsonUrbanke}.
Then it makes sense to prescribe the variable and check degrees exactly rather than just in distribution.
Indeed, assuming that 
$n$ such that $n\pr\brk{\vd=j}$ is an integer for every $j\in\supp\vd$ and that
$m=dn/k$ is an integer such that $m\pr\brk{\vk=j}$ is integral for each $j\in\supp\vk$,
let $\fD$ be the event that $\vm=m$ and that
\begin{align*}
\sum_{i=1}^n\vecone\{\vd_i=\ell\}&=n\pr\brk{\vd=\ell}\qquad\mbox{and}\qquad
\sum_{i=1}^{m}\vecone\{\vk_i=\ell\}=m\pr\brk{\vk=\ell}\qquad\mbox{for all integers $\ell$}.
\end{align*}
The {\em rate of the $(\vd,\vk)$-ldpc code of block length $n$} is defined as 
$$
\rate_n(\vd,\vk)=n^{-1}\Erw\brk{\nul(\A)\mid\fD}.
$$

\begin{theorem}\label{Thm_LDPC}
We have $\lim_{n\to\infty}\rate_n(\vd,\vk)= \max_{\alpha\in[0,1]}\Phi(\alpha).$
\end{theorem}

\subsection{The 2-core}\label{Sec_intro_2core}
In several examples the solution to $\max_{\alpha\in[0,1]}\Phi(\alpha)$ has a natural combinatorial interpretation
in terms of the Tanner graph $\G$.
Indeed, define the {\em 2-core} of $\G$ as the subgraph $\GG_*$ obtained by repeating the following operation.
\begin{quote}
While there is a variable node $x_i$ of degree one or less, remove that variable node along with the adjacent check node (if any).
\end{quote}
Of course, the 2-core might be empty.
We will see momentarily how the 2-core is related to the rank.

Extending prior results that dealt with the case that the degrees of all check nodes coincide~\cite{Cooper04},
we compute the likely number of variable and check nodes in the 2-core.
Since $\vd,\vk$ have finite second moments,
\begin{equation}\label{def:rho}
\rho=\max\{x\in[0,1]: \Phi'(x)=0\}
\end{equation}
is well-defined.
So are $D''(x),K''(x)$ for $x\in[0,1]$.
Let
\begin{equation}\label{def:f}
\phi(\alpha)=1-\alpha-\frac{1}{d}D'\left(1-K'(\alpha)/{k}\right)
\end{equation}
so that $\Phi'(\alpha)=\frac{d}{k}K''(\alpha)\phi(\alpha)$.

\begin{theorem} \label{thm:2core}
Assume that $\vd,\vk$ are such that $\phi'(\rho)<0$.
Let
 $\bfn^*$ and $\bfm^*$ be the number of  variable and constraint nodes in the 2-core, respectively.
Then
 \begin{align}\label{eqthm:2core}
\lim_{n\to\infty}\frac{\bfn^*}{n}&=1-D\left(1-\frac{K'(\rho)}{k}\right)-\frac{K'(\rho)}{k} D'\left(1-\frac{K'(\rho)}{k}\right),&
\lim_{n\to\infty}\frac{\bfm^*}{n}&=\frac dk K( \rho)\quad\mbox{in probability.}
 \end{align} 
 \end{theorem}

\noindent
We observe that the expressions on the r.h.s.\ of~\eqref{eqthm:2core} evaluate to zero if $\rho=0$.

\Thm~\ref{thm:2core} yields an upper bound on the rank of $\A$.
Indeed, upper bounding the rank is equivalent to lower bounding the nullity.
To this end, we count solutions to $\A x=0$ where $x_i=0$ for all variables in the 2-core.
Since the number of check nodes that have a neighbour outside the 2-core is $m-\vm^*$ and the number of variable nodes outside  is $n-\vn^*$, we obtain $\nul(\A)\geq n-\vn^*-(\vm-\vm^*)$.
Invoking \Thm~\ref{thm:2core} and using $\Phi'(\rho)=0$, we find that 
$\rank(\A)/n\leq 1-\Phi(\rho)+o(1)$ \whp\ 
Since the rank is also trivially upper-bounded by $\vm$ and $1-\Phi(0)=d/k\sim\vm/n$ \whp, these purely combinatorial deliberations show that \whp\
\begin{align}\label{eq2corebound}
\rank(\A)/n\leq 1-\max\{\Phi(0),\Phi(\rho)\}+o(1).
\end{align}
This graph-theoretic upper bound is tight if all vectors in $\ker(\A)$ are constant zero on the 2-core.
The following theorem shows that \eqref{eq2corebound}  is indeed tight in several interesting cases.

\begin{theorem}\label{Thm_tight}
Assume that
\begin{enumerate}[(i)]
\item either $\Var(\vd)=0$ or $\vd\disteq\Po_{\geq\ell}(\lambda)$ for an integer $\ell\geq1$ and $\lambda>0$, and
\item either $\Var(\vk)=0$ or $\vk\disteq\Po_{\geq\ell'}(\lambda')$ for an integer $\ell'\geq3$ and $\lambda'>0$.
\end{enumerate}
Then 
\begin{align*}
\lim_{n\to\infty}\rank(\A)/n&=1-\max\{\Phi(0),\Phi(\rho)\}&\mbox{in probability}.
\end{align*}
Moreover, assuming (i) and (ii) we have  $\phi'(\rho)<0$ unless $\pr(\bfd=1)=0$ and $2(k-1)\pr(\bfd=2)>d$.
\end{theorem}

\begin{remark}
If $\pr(\bfd=1)=0$ and $2(k-1)\pr(\bfd=2)>d$, then the 2-core comprises the entire graph $\bfG$.
But the 2-core may be instable, i.e., a subgraph obtained by deleting just a few nodes may have a much smaller or empty 2-core.
\end{remark}

\begin{figure}
\includegraphics[height=5cm]{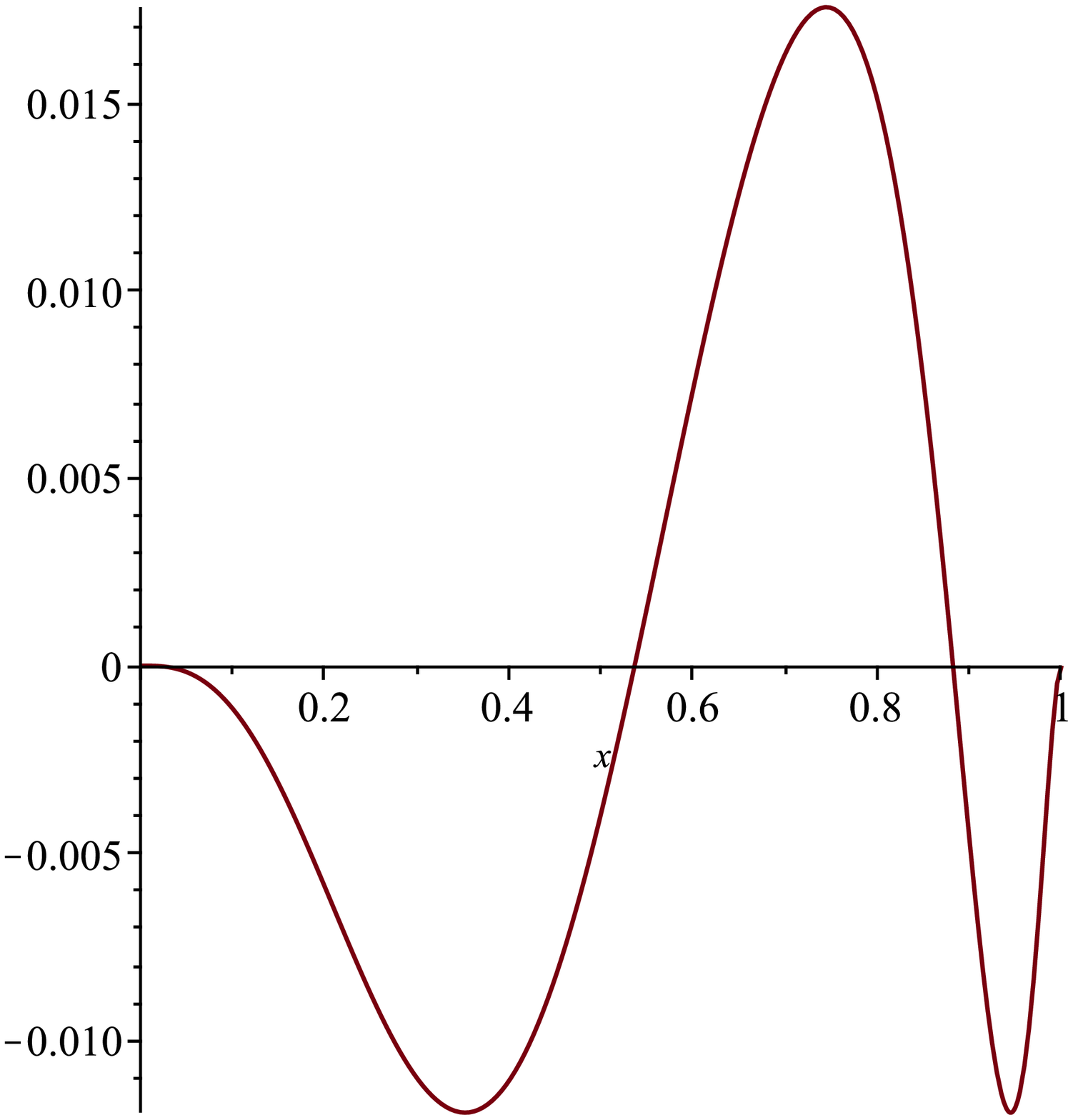}$\qquad$
\includegraphics[height=5cm]{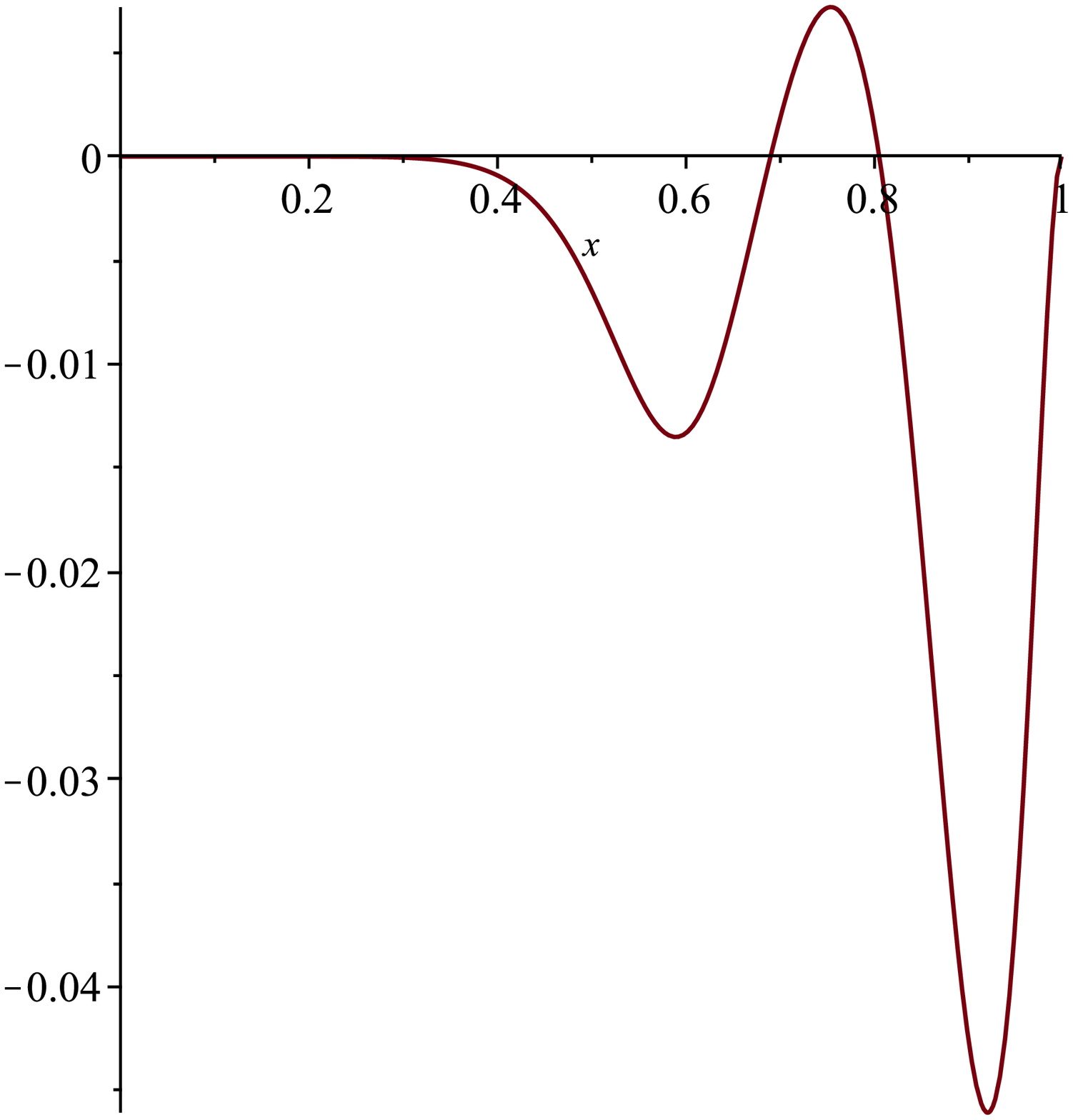}
\caption{The function $\Phi$ for Examples~\ref{Ex_Lelarge} (left) and~\ref{Ex_Amin} (right)}\label{Fig_plot}
\end{figure}

\noindent
\Thm~\ref{Thm_tight} verifies experimental findings of Alamino and Saad~\cite{AlaminoSaad} rigorously. 
Yet the following example of Lelarge~\cite{Lelarge} shows that \eqref{eq2corebound} is not universally tight.

\begin{example}\label{Ex_Lelarge}
Let $\vd,\vk$ be the variables whose probability generating functions read
$D(x)=K(x)=4x^3/5+x^{15}/5$.
Then $\rho=1$ and $\Phi(0)=\Phi(\rho)=0$ while $\max_{x\in[0,1]}\Phi(x)>0$.
Hence, the kernel of $\A$ contains vectors that are not constant zero on the 2-core \whp
\end{example}

\noindent
Here is an example where \eqref{eq2corebound} fails to hold even though $\Var(\vk)=0$.

\begin{example}\label{Ex_Amin}
Let $\vk=10$ deterministically and let $\vd$ be the variable with
$D(x)=(190x^3+7x^{200})/197$.
Then $\rho=1$, $\Phi(0)=\Phi(\rho)=0$ but $\max_{x\in[0,1]}\Phi(x)>0$.
\end{example}

\noindent
Figure~\ref{Fig_plot} shows the function $\Phi$ for the two examples.

\subsection{Related work}\label{Sec_related}
The rank problem has been studied both via rigorous methods in the combinatorics and coding theory communities and via physics-inspired non-rigorous approaches.

\subsubsection{Rigorous work}
Kovalenko~\cite{Kovalenko} studied the rank of dense random matrices over $\FF_2$, where each entry takes the value $0$ or $1$ with equal probability.
The result was subsequently extended to dense matrices over arbitrary finite fields $\FF_q$ with entries drawn from non-uniform distributions~\cite{Kovalenko2}.
The rank problem for sparser random matrices with an average number of $\Omega(\log n)$ non-zero entries per row 
was tackled by Balakin~\cite{Balakin1,Balakin2}, Bl\"omer, Karp and Welzl~\cite{BKW}  and Cooper~\cite{Cooper}.
Random matrices with precisely two non-zero entries per row have also been studied extensively~\cite{Kolchin1,Kolchin2}.
In this very particular situation the value of one variable implies that of the other, 
and thus the rank problem is intimately related to the component and cycle structure of the Tanner graph.
Kolchin's monograph~\cite{Kolchin} presents much of this classical work.

More recently sparse matrices where the number of non-zero entries scales linearly with the number of columns have received a great deal of attention.
Miller and Cohen~\cite{MillerCohen} investigated the biregular case, i.e., $\Var(\vd)=\Var(\vk)=0$ and $q=2$.
Furthermore, M\'easson,  Montanari and Urbanke derived a sufficient condition for $\vd,\vk$ to yield a random matrix $\A$ of full rank \whp~\cite{MMU}.
In the case $q=2$ and $\Var(\vk)=0$ this problem is equivalent to the $k$-XORSAT problem, a well-studied constraint satisfaction problem.
The most prominent case is that $\vd\disteq\Po(d)$.
In this case the precise threshold $d_k$ for the mean $d$ up to which the random matrix typically has full rank was determined by
Dubois and Mandler~\cite{DuboisMandler} for $q=2$, $k=3$ and by 
Dietzfelbinger, Goerdt, Mitzenmacher, Montanari, Pagh and Rink~\cite{Dietzfelbinger} and
Pittel and Sorkin~\cite{PittelSorkin} for $q=2$ and general  $k$.
Falke and Goerdt~\cite{GoerdtFalke} extended the result to $q\leq4$ and
Ayre, Coja-Oghlan, Gao and M\"uller~\cite{Ayre} to general $k$ and $q$.
They also provided a formula for the rank for $d$ beyond $d_k$.
For $q=2$ the same result was obtained independently by Cooper, Frieze and Pegden~\cite{CFP}.

The real rank of discrete random matrices has, of course, been studied as well.
For instance, \Komlos~\cite{Komlos} showed that a random $\pm1$ matrix is likely regular, a result that sparked a series of important papers on the asymptotic order of the singularity probability~\cite{BVM,KKS,TaoVu}.
Moreover, Bordenave, Lelarge and Salez~\cite{BLS} studied the real rank of the adjacency matrix of random graphs with given degrees.

A wealth of references on low-density parity check codes can be found in Richardson and Urbanke's monograph~\cite{RichardsonUrbanke}.
Chapter~3 contains a detailed discussion of the impact of the degree sequence.
The objective in coding theory is to recover the original codeword sent from the information received through a noisy channel.
This problem can be studied from both an algorithmic and an information-theoretic viewpoint.
Of course, the results depend on the channel model.
Ldpc codes are optimal in either respect on the binary erasure channel, where bits may be deleted but not altered~\cite{RichardsonUrbanke}.
Moreover, they are information-theoretically optimal on the binary symmetric channel, where bits may be flipped~\cite{GMU}.
In addition, a variant called spatially coupled ldpc codes even admit efficient algorithms for optimal decoding on binary memoryless symmetric channels~\cite{KYMP}.

The $k$-core of a random graph (the maximum subgraph with minimum degree at least $k$) was first investigated by Bollob\'{a}s~\cite{Bollobas} in the study of  $k$-connected subgraphs of random graphs. Since then it has received great attention. There had been continuing progress~\cite{Chvatal,Luczak91,Luczak92}  in estimating the threshold of the appearance of a non-empty $k$-core. A breakthrough was made by Pittel, Spencer and Wormald~\cite{Pittel}, who determined the sharp threshold of the $k$-core emergence. Their work inspired extensive research on the $k$-core of other random graph models such as random graphs with specified degrees~\cite{FR,FR2} and random hypergraphs~\cite{Cooper04,molloy2005cores}, as well as the development of new proof techniques for analysing the $k$-core~\cite{Cain,Janson,Kim,Riordan}. Particularly relevant to this paper are work by Molloy~\cite{molloy2005cores} and Cooper~\cite{Cooper04}. Cooper studied the $k$-core of random uniform hypergraphs with specified degrees. Our result in Theorem~\ref{thm:2core} deals with a more general random hypergraph model where the sizes of hyperedges do not need to be uniform, whereas Cooper's work deals with the general $k$-core (rather than 2-core). Molloy studied the $k$-core of a random uniform hypergraph without degree constraints. While Molloy's work is a special case of Cooper's, his proof is simple and easy to be adapted for more complicated random graph models. Our proof of Theorem~\ref{thm:2core} is a modification of Molloy's approach.

\subsubsection{The cavity method (and its caveats)}\label{Sec_cavity}
The cavity method, an analytic but non-rigorous technique inspired by the statistical mechanics of disordered systems, comes in two instalments, the simpler {\em replica symmetric ansatz} and the more elaborate {\em one-step replica symmetry breaking ansatz} (`1RSB').
The replica symmetric ansatz predicts that the rank of $\A$ converges in probability to the solution of an optimisation problem on an  infinite-dimensional space of probability measures.
To be precise, let $\cP(\FF_q)$ be the space of probability measures on $\FF_q$.
Identify this space with the standard simplex in $\RR^q$.
Further, let $\cP^2(\FF_q)$ be the (infinite-dimensional) space of probability measures on $\cP(\FF_q)$.
Given $\pi\in\cP^2(\FF_q)$ let $(\MU_{i,j})_{i,j\geq1}$ be a sequence of samples from $\pi$.
In addition, define $\hat\vk$ by
\begin{align}\label{eqhatvk}
\pr\brk{\hat\vk=\ell}&=\ell\cdot\pr\brk{\vk=\ell}/k&
(\ell\geq0),
\end{align}
and let $(\hat\vk_i)_{i\geq1}$ be copies of  $\hat\vk$ .
The random variables $(\MU_{h,i})_{h,i\geq1}$, $(\hat\vk_i)_{i\geq1}$ are mutually independent
and independent of $\vd,\vk$ and the $(\CHI_{h,i})_{h,i\geq1}$.
The {\em Bethe free entropy} is the functional $\cB:\cP^2(\FF_q)\to\RR$ defined by
\begin{align*}
\cB(\pi)=\Erw&\brk{\log_q\sum_{\sigma_1\in\FF_q}\prod_{i=1}^{\vd}\sum_{\sigma_2,\ldots,\sigma_{\hat\vk_i}\in\FF_q}
		\vecone\cbc{\sum_{j=1}^{\hat\vk_i}\sigma_j\CHI_{i,j}=0}\prod_{j=2}^{\hat\vk_i}\MU_{i,j}(\sigma_j)}\\
	&-\frac{d}k\Erw\brk{(\vk-1)\log_q\sum_{\sigma_1,\ldots,\sigma_{\vk}\in\FF_q}
			\vecone\cbc{\sum_{i=1}^{\vk}\sigma_i\CHI_{1,i}=0}
		\prod_{i=1}^{\vk}\MU_i(\sigma_i)}.
\end{align*}
The replica symmetric ansatz predicts that
\begin{align}\label{eqRS}
\lim_{n\to\infty}\frac1n\nul\A&=\sup_{\pi\in\cP^2(\FF_q)}\cB(\pi)&\mbox{in probability.}
\end{align}

For a detailed (heuristic) derivation of the Bethe free entropy and the prediction~\eqref{eqRS} we refer to~\cite[\Chap~14]{MM} and~\cite{AlaminoSaad}.
But let us briefly comment on the intended semantics of $\pi$.
Consider the Tanner graph $\GG$ representing the random matrix $\A$.
Suppose that variable node $x_i$ and check node $a_j$ are adjacent.
Then for $\sigma\in\FF_q$ we define the {\em message} $\mu_{\A,x_j\to a_i}(\sigma)$ from $x_j$ to $a_i$ as follows.
Obtain $\A_{x_j\to a_i}$ from $\A$ by changing the $ij$'th matrix entry to zero; this corresponds to deleting the $x_j$-$a_i$-edge from the Tanner graph.
Then $\mu_{\A,x_j\to a_i}(\sigma)$ is the probability that in a random vector $\SIGMA\in\ker\A_{x_j\to a_i}$ we have $\SIGMA_j=\sigma$.
In other words, $\mu_{\A,x_j\to a_i}\in\cP(\FF_q)$ is the marginal distribution of the value assigned to $x_j$ in a random vector from the kernel of $\A_{x_j\to a_i}$.
Further, define $\pi_{\A}$ as the empirical distribution of the messages $\mu_{\A,x_j\to a_i}$ over the edges of the Tanner graph;
in symbols,
\begin{align*}
\pi_{\A}&=\frac1{\sum_{i=1}^n\vd_i}\sum_{j=1}^n\sum_{i=1}^{\vm}\vecone\{\A_{ij}\neq0\}\delta_{\mu_{\A,x_j\to a_i}}\in\cP^2(\FF_q).
\end{align*}
Then the replica symmetric ansatz predicts that $\pi_{\A}$ is asymptotically a maximiser of the Bethe free energy, i.e., that
$\sup_{\pi\in\cP^2(\FF_q)}\cB(\pi)=\cB(\pi_{\A})+o(1)$ \whp\
Thus, the maximiser $\pi$ in \eqref{eqRS} is deemed to encode the messages whizzing along the edges of the Tanner graph in the limit $n\to\infty$.

A bit of linear algebra (that seems to have gone unnoticed in the physics literature) reveals that the messages actually have a very special form.
Namely, for any adjacent $x_j$ and $a_i$ either $\mu_{\A,x_j\to a_i}(\sigma)=1/q$ for all $\sigma\in\FF_q$ or $\mu_{\A,x_j\to a_i}(0)=1$.
In other words, $\mu_{\A,x_j\to a_i}$ is either the uniform distribution $q^{-1}\vecone\in\cP(\FF_q)$ or the atom $\delta_0$ on $0$.
In effect, the rank should come out as $\cB(\pi_\alpha)$ for a convex combination
\begin{align}\label{eqpialpha}
\pi_\alpha&=\alpha\delta_{\delta_0}+(1-\alpha)\delta_{q^{-1}\vecone}
\end{align}
of the atom on $\delta_0$ and the atom on the uniform distribution.
In fact, a simple calculation yields
%\begin{align*}
$\Phi(\alpha)=\cB(\pi_\alpha).$
%\end{align*}
Thus, \Thm~\ref{thm:rank} shows that  $\rk\A/n$ converges to $1-\sup_{\alpha\in[0,1]}\cB(\pi_\alpha)$ in probability.
\Thm~\ref{thm:rank} vindicates the cavity method to that extent.
But a question that remains open is whether the Bethe free entropy admits other `spurious' maximisers
$\pi\in\cP^2(\FF_q)$ with $\cB(\pi)>\sup_{\alpha\in[0,1]}\cB(\pi_\alpha)$.

Alamino and Saad~\cite{AlaminoSaad} tackled the optimisation problem~\eqref{eqRS} directly (without noticing the restriction to $(\pi_\alpha)_{\alpha})$ by means of a numerical heuristic called population dynamics.
In all the examples that they studied they found that $\pi\in\{\pi_0,\pi_\rho\}$, with $\rho$ from \eqref{def:rho};
	in fact, it so happens that all their examples fall within the purview of \Thm~\ref{Thm_tight}.%
	\footnote{Strictly speaking, Alamino and Saad, who worked numerically with $n$ in the hundreds, reported $\pi\in\{\pi_0,\pi_1\}$. Indeed, $\rho\in\{0,1\}$ in the first class of examples that they studied, but not in the other two. For instance, in their example (3) the actual value of $\rho$ is either $0$ or a number strictly smaller than one, although $\rho>0.97$ whenever $\Phi(\rho)>\Phi(0)$.}
This led Alamino and Saad to conjecture that the maximiser $\pi$ is generally of this form, although they cautioned that further evidence seems necessary.
Example~\ref{Ex_Lelarge} provides a counterexample.

The 1RSB variant of the cavity method, which is conceptually more intricate than the replica symmetric version, is presented in \cite[\Chap~19]{MM}.
An exercise in that chapter asks the reader to verify that the rate of an ldpc code is generally equal to $1-\max\{\Phi(0),\Phi(\rho)\}+o(1)$.
This `prediction' rests on the hypothesis that either the random matrix has full rank, or all vectors in the kernel are constant zero on the 2-core.
\Thm~\ref{Thm_tight} gives sufficient conditions for this to be correct, while Example~\ref{Ex_Lelarge} provides a counterexample.

\subsection{Preliminaries and notation}\label{Sec_pre}
Throughout the paper we will be dealing with a double limit where $\eps\to0$ after $n\to\infty$ (`$\lim_{\eps\to0}\lim_{n\to\infty}$').
The standard $O$-notation refers to the limit $n\to\infty$ with $\eps$ fixed.
Hence, $O(1)$ hides a term that remains bounded as $n\to\infty$ but may be unbounded as $\eps\to 0$.
For instance, $1/\eps=O(1)$.
In addition, we will use the symbols $O_\eps$, $o_\eps$, etc.\ to refer to the joint limit $\eps\to0$, $n\to\infty$.
Thus, $O_\eps(1)$ denotes a term that remains bounded both in terms of $n$ and $\eps$ and $o_\eps(1)$ denotes a term that gets arbitrarily small in absolute value as $\eps$ gets small and $n$ large.
We will always assume tacitly that $n$ is sufficiently large for our various estimates to hold.

Frequently we will define random matrices indirectly by way of constructing their Tanner graphs.
Generally, suppose that $G=(V\cup F,E)$ is a bipartite multi-graph on a set $V$ of variable nodes and a set $F$ of check nodes. Thus, $E=(e_1,\ldots,e_\ell)$ is an ordered tuple of edges, each joining a node $v$ from $V$ with an $f$ from $F$, and $e_i=e_j$ is allowed for $i\neq j$.
With $\CHI_1,\CHI_2,\ldots$ mutually independent copies of $\CHI$,
we define a random matrix $\vA(G)$ with columns indexed by $V$ and rows indexed by $F$ by letting
\begin{align}\label{eqTanner}
\vA_{ax}(G)&=\sum_{i=1}^\ell \CHI_i\vecone\{e_i=ax\}&(x\in V,\ a\in F).
\end{align}	
Hence, each $a$-$x$-edge in $G$ contributes one summand to $\vA_{ax}(G)$.
In particular, the matrix $\vA(G)$ has {\em at most} $\ell$ non-zero entries, as cancellations may occur in~\eqref{eqTanner}.

We use standard notation for graphs and multi-graphs.
For instance, for a vertex $v$ of a multi-graph $G$ we denote by $\partial_G v$ the set of neighbours.
More generally, for an integer $\ell\geq1$ we let $\partial_G^\ell v$ be the set of vertices at distance precisely $\ell$ from $v$.
We omit the reference to $G$ where possible.
Further, if $G$ is the Tanner graph of a matrix $A$, $a$ is a check node and $\sigma\in\FF_q^{\partial a}$ assigns a value from $\FF_q$ to each variable node adjacent to $a$, then we write $\sigma\models_Aa$ if $\sum_{x\in\partial a}A_{ax}\sigma_x=0$.
We omit the reference to $A$ where it emerges from the context.

Let us denote the set of probability measures on a finite set $\cX$ by $\cP(\cX)$.
We will be working a fair bit with probability distributions on discrete cubes $\Omega^{I}$, with $\Omega,I$ finite sets.
For a subset $J\subset I$ and $\mu\in\cP(\Omega^I)$ we denote by $\mu_{J}\in\cP(\Omega^J)$ the distribution that $\mu$ induces on the coordinates $J$:
\begin{align*}
\mu_{J}(\sigma)&=\mu\bc{\cbc{\tau\in\Omega^{I}:\forall j\in J:\tau_j=\sigma_j}}&(\sigma\in\Omega^{J}).
\end{align*}
If $J=\{j_1,\ldots,j_\ell\}$ is given explicitly, we use the shorthand $\mu_{J}=\mu_{j_1,\ldots,j_\ell}$.

Asymptotic properties of discrete distributions $\mu\in\cP(\Omega^{I})$ with $\Omega$ fixed and $|I|$ large will play an important role.
Following~\cite{Victor}  we say that $\mu$ is {\em $(\eps,\ell)$-symmetric} if
\begin{align*}
\sum_{i_1,\ldots,i_\ell\in I}\dTV\bc{\mu_{i_1,\ldots,i_\ell},\mu_{i_1}\tensor\cdots\tensor\mu_{i_\ell}}&<\eps |I|^\ell.
\end{align*}
If $\ell=2$ we just say that $\mu$ is $\eps$-symmetric.
Thus, loosely speaking, a measure $\mu$ is $\eps$-symmetric if for `most' pairs $i_1,i_2$ the joint distribution $\mu_{i_1,i_2}$ is `close' to the product  $\mu_{i_1}\tensor\mu_{i_2}$ of the marginals.
The following lemma shows that $\delta$-symmetry implies $(\eps,\ell)$-symmetry for large enough $|I|$ and small enough $\delta$.

\begin{lemma}[\cite{Victor}]\label{lem:l-wise}
For any $\eps>0$, $\ell\ge 3$ there exists $\delta=\delta(\eps,\ell)>0$ such that for all sets $I$ of size $|I|>1/\delta$ and all $\mu\in \P(\Omega^I)$ the following is true: 
if $\mu$ is $\delta$-symmetric, then $\mu$ is $(\eps,\ell)$-symmetric.
\end{lemma}

Throughout we keep the assumptions on the distributions $\vd,\vk$ listed in \Sec~\ref{Sec_intro}.
In particular, $\Erw[\vd^r]+\Erw[\vk^r]<\infty$ for some $r>2$.
We write $\gcd(\vk)$ and $\gcd(\vd)$ for the greatest common divisor of the support of $\vd$ and $\vk$, respectively.
When working with $\A$ we tacitly assume that $\gcd(\vk)$ divides $n$.
But let us make sure that $\A$ is well-defined in the first place.

\begin{proposition}\label{Lemma_welldef}
With probability $\Omega(n^{-1/2})$ over the choice of $\vm$, $(\vd_i)_{i\geq1}$, $(\vk_i)_{i\geq1}$ the condition \eqref{eqWellDef1} is satisfied and there exists a simple Tanner graph $\G$ with variable degrees $\vd_1,\ldots,\vd_n$ and check degrees $\vk_1,\ldots,\vk_{\vm}$.
\end{proposition}

\noindent
The proof of \Prop~\ref{Lemma_welldef} is based on technical but standard arguments; we defer it to \Sec~\ref{Sec_welldef}.

In addition to the size-biased random variable $\hat\vk$ from \eqref{eqhatvk} we also define $\hat\vd$ by
\begin{align*}
\pr\brk{\hat\vd=\ell}&=\ell\pr\brk{\vd=\ell}/{d}&(\ell\geq0).
\end{align*}
Throughout the paper we let $(\vk_i,\vd_i,\hat\vk_i,\hat\vd_i)_{i\geq1}$ denote mutually independent copies of $\vk,\vd,\hat\vk,\hat\vd$.
Unless specified otherwise, all these random variables are assumed to be independent of any other sources of randomness.
Finally, we need the following basic lemma on sums of independent random variables.

\begin{lemma}\label{Lemma_sums}
Let $r>2$, $\delta>0$ and suppose that $(\vec\lambda_i)_{i\geq1}$ are independent copies of a random variable
$\vec\lambda\geq0$ with $\Erw[\vec\lambda^r]=O(1)$.
Further, let $s=\Theta(n)$.
Then
$\pr\brk{\abs{\sum_{i=1}^{s}(\vec\lambda_i-\Erw[\vec\lambda])}>\delta n}=o(1/n).$
\end{lemma}

\noindent
For the sake of completeness the proof of \Lem~\ref{Lemma_sums} is included in the appendix.

\section{Proof strategy}

\noindent
In this section we survey the techniques upon which the proofs of the theorems stated in \Sec~\ref{Sec_intro} are based.
We will also compare these techniques with those employed in prior work.
Indeed, in contrast to much of the previous work on the rank of random matrices, based on combinatorial and graph-theoretic considerations 
(e.g.,~\cite{Cooper,CFP}), we will take a more probabilistic viewpoint that harnesses ideas from mathematical physics.
The protagonist of this approach is the Boltzmann distribution.

\subsection{The Boltzmann distribution}\label{Sec_Boltzmann}
Suppose that $A$ is an $m\times n$-matrix over $\FF_q$.
Borrowing a term from statistical physics, we refer to the probability distribution $\mu_A$ on $\FF_q^n$ defined by
\begin{align*}
\mu_A(\sigma)&=\vecone\{\sigma\in\ker(A)\} q^{-\nul(A)}&(\sigma\in\FF_q^n)
\end{align*}
as the {\em Boltzmann distribution} of $A$.
Thus, $\mu_A$ induces the uniform distribution on the kernel of $A$.
We denote a random vector drawn from $\mu_A$  by $\SIGMA_A=(\SIGMA_{A,1},\ldots,\SIGMA_{A,n})\in\FF_q^n$, or just by $\SIGMA=(\SIGMA_1,\ldots,\SIGMA_n)$ where $A$ is apparent.

The Boltzmann distribution is important to us because the proofs of the main results hinge on coupling arguments.
Roughly speaking, we will be dealing with matrices obtained from $\A$ by adding a few rows and columns.
We will need to study the ensuing change in nullity.
Crucially, the total number of new rows and columns will typically be bounded (i.e., independent of $n$), and each  will merely contain a bounded number of non-zero entries.
Let us investigate how the Boltzmann distribution can then be harnessed to trace the nullity.

Suppose that the $m'\times n'$-matrix $A'$ is obtained from the $m\times n$-matrix $A$ by adding  rows and columns:
\begin{align*}
A'&=\begin{pmatrix}A&0\\A''&A'''	\end{pmatrix},
\end{align*}
with $A''$ of size $(m'-m)\times n$ and $A'''$ of size $(m'-m)\times(n'-n)$.

\begin{fact}\label{Fact_nullity}
Let $I\subset[n]$ be the set of column indices where $A''$ has a non-zero entry and let $J=\{n+1,\ldots,n'\}$. Then
\begin{align}\label{eqFact_nullity}
\nul(A')-\nul(A)&=\log_q\sum_{\sigma\in\FF_q^J,\tau\in\FF_q^I}\mu_{A,I}\bc{\tau}
	\prod_{h=1}^{m'-m}\vecone\cbc{\sum_{i\in I}A''_{hi}\tau_j+\sum_{j\in J}A'''_{h,j-n}\sigma_{j}=0}.
\end{align}
\end{fact}
\begin{proof}
Inside the logarithm we sum for every $\tau$ the number of possible extensions $\sigma$ to a vector in the kernel of $A'$ divided by $q^{\nul(A)}$.
\end{proof}

To calculate the r.h.s.\ of~\eqref{eqFact_nullity} we need to get a handle on the joint distribution $\mu_{A,I}$ of a bounded number of coordinates $I$.
The following lemma, whose proof consists of a few lines of linear algebra, marks a first step.

\begin{lemma}[{\cite[\Lem~2.3]{Ayre}}]\label{Lemma_Spartition}
For any matrix $A\in\FF_q^{m\times n}$ there exists a decomposition $S_0,\ldots,S_\ell$ of $[n]$ into pairwise disjoint sets with the following properties.
\begin{enumerate}[(i)]
\item if $i\in S_0$, then $\sigma_i=0$ for all $\sigma\in\ker(A)$.
\item if $i,j\in S_h$ for  $h\in[\ell]$, then $\mu_{A,i}(\tau)=1/q$ for all $\tau\in\FF_q$ and there is $s\in\FF_q^*$ such that $\sigma_j=s\sigma_i\mbox{ for all }\sigma\in\ker(A).$
\item if $i\in S_g$ and $j\in S_h$ for $1\leq g<h\leq\ell$, then for all $\tau,\tau'\in\FF_q$ we have 
	$\mu_{A,i,j}(\tau,\tau')=1/q^2.$
\end{enumerate}
\end{lemma}

\noindent
\Lem~\ref{Lemma_Spartition} implies that each Boltzmann marginal $\mu_{A,i}$ is either the uniform distribution on $\FF_q$ or the atom $\delta_0$.%
\footnote{This readily implies that the Belief Propagation messages from \Sec~\ref{Sec_cavity} are either uniform or atoms.}
The latter occurs iff $i\in S_0$.
Let us therefore call the coordinates $i\in S_0$ {\em frozen}.
Moreover, \Lem~\ref{Lemma_Spartition} provides information about pairwise correlations.
Specifically, if $i,j$ belong to the same block $S_h$, then $\SIGMA_i$ and $\SIGMA_j$ are linearly related, but if $i,j$ belong to different blocks, then $\SIGMA_i$ and $\SIGMA_j$ are stochastically independent.

To seize upon this observation we introduce a perturbation of the matrix $A$.
Namely, for $i_1,\ldots,i_\ell\in[n]$ let $A[i_1,\ldots,i_\ell]$ be the matrix obtained from $A$ by
adding for each $j\in[\ell]$ a row that has an entry one in column $i_j$ and zeros in all other columns.
Thus, the kernel of $A[i_1,\ldots,i_\ell]\sigma$ comprises all $\sigma\in\ker A$ with $\sigma_{i_j}=0$ for all $j\in[\ell]$, i.e.,
we explicitly freeze $i_1,\ldots,i_\ell$.
The following lemma, which is a generalisation of \cite[\Cor~3.2]{Ayre}, shows that freezing a few random coordinates nearly eliminates pairwise correlations.

\begin{lemma}\label{Lemma_0pinning}
For any $0<\eps<1$ there exists $\Theta=\Theta(\eps)>0$ dependent on $\eps$ only such that for all $A\in\FF_q^{m\times n}$, $\emptyset\neq U\subset[n]$ the following is true.
Draw $\THETA\in[\Theta]$ and $(\vec i_1,\ldots,\vec i_{\THETA})\in U^{\THETA}$ uniformly at random and
let $\hat A=A[\vec i_1,\ldots,\vec i_{\THETA}]$.
Then
\begin{align}\label{eqLemma_0pinning}
\pr\brk{\mu_{\hat A,U}\mbox{ is $\eps$-symmetric}}>1-\eps.
\end{align}
\end{lemma}

\noindent
The proof, which is an easy extension of the argument from~\cite{Ayre}, is included in Appendix~\ref{Apx_0pinning}.
Crucially, the number $\THETA$ of rows that we add to $A$ is bounded by a number $\Theta$ that depends only on $\eps$ but not on $n$ or $m$.
Therefore, in our coupling arguments below the perturbation will shift the rank by a negligible amount.

The coupling arguments, which we are going to explain next, are based on \Lem s~\ref{lem:l-wise}, \ref{Lemma_Spartition} and~\ref{Lemma_0pinning} together with Fact~\ref{Fact_nullity}, which make for a powerful combo.
Indeed, Fact~\ref{Fact_nullity} reduces the problem of calculating the nullity to studying the joint distribution of a bounded number of coordinates.
Thanks to \Lem~\ref{Lemma_0pinning} we can make this distribution $\delta$-symmetric, which \Lem s~\ref{lem:l-wise} boosts to $(\eps,\ell)$-symmetry for any bounded $\ell$.
In effect, we will be able to replace the measure $\mu_{A,I}$ in \eqref{eqFact_nullity} by a product measure.
Further, \Lem~\ref{Lemma_Spartition} shows that the marginals of the product measure are either atom $\delta_0$ or uniform on $\FF_q$.

\subsection{The Aizenman--Sims--Starr scheme}\label{Sec_outline1}
We are going to tackle the rank problem by way of calculating the nullity of $\A$.
We will prove matching upper and lower bounds on the nullity via two separate arguments.
The proof of the upper bound on the nullity (which yields the lower bound on the rank) is based on a type of coupling argument that is colloquially referred to as the `Aizenman--Sims--Starr scheme'.
Originally developed to cope with models of a rather different look in mathematical physics, applied to the nullity the argument rests upon the observation that
\begin{align*}
\limsup_{n\to\infty}\frac1n\Erw[\nul(\vA)]&\leq\limsup_{n\to\infty}\Erw[\nul(\vA_{n+1})]-\Erw[\nul(\vA_{n})].
\end{align*}
(The inequality is easily verified by writing a telescoping sum.)
In order to estimate the right hand side, it seems natural to couple $\vA_{n+1}$ and $\vA_{n}$ so as to write
\begin{align}\label{eqASS1}
\Erw[\nul(\vA_{n+1})]-\Erw[\nul(\vA_{n})]=\Erw\brk{\nul(\vA_{n+1})-\nul(\vA_{n})}.
\end{align}
Indeed, if we were to find a coupling under which $\vA_{n+1}$ results from $\vA_n$ by (randomly) adding, say, one column along with a few rows, then we could bring the machinery from \Sec~\ref{Sec_Boltzmann} to bear.

The immediate issue is that the random matrices $\A_n$ and $\A_{n+1}$ do not lend themselves to an easy coupling.
In fact, $\A_{n+1}$ may not even be defined (due to divisibility issues).
But even if it is, the structure of $\A_n$ appears too rigid to allow for $\A_{n+1}$ to be easily described as `$\A_n$ plus one column and a few rows'.
Hence, to facilitate couplings we introduce an auxiliary model that resembles the configuration model from the theory of random graphs and that allows for a bit of wiggling room.

Specifically, we fix a parameter $\eps>0$ along with a large enough $\Theta=\Theta(\eps)>0$ dependent on $\eps$ only.
Then for any integer $n>0$ consider the random matrix  $\vA_\eps=\vA_{\eps,n}$ constructed as follows.
Let $\vm_\eps\disteq\Po((1-\eps)dn/k)$ and independently choose $\THETA\in[\Theta]$ uniformly at random.
Moreover, let 
	$$(\vd_i)_{i\geq1},\qquad(\vk_i)_{i\geq1},\qquad(\CHI_{i,j,s,t})_{i,j,s,t\geq1}$$
be copies of $\vd$, $\vk$ and $\CHI$, mutually independent and independent of $\vm_\eps$ and $\THETA$.
Further, let $\vec\Gamma_\eps=\vec\Gamma_{\eps,n}$ be a random maximal matching of the complete bipartite graph with vertex classes
\begin{align*}
\bigcup_{i=1}^{\vm_\eps}\cbc{a_i}\times[\vk_i],\quad\bigcup_{j=1}^{n}\cbc{x_j}\times[\vd_j].
\end{align*}
Think of $\cbc{a_i}\times[\vk_i]$ as a set of clones of $a_i$ and of $\{x_j\}\times[\vd_j]$ as a set of clones of $x_j$.
We obtain a random Tanner graph $\G_\eps=\G_{\eps,n}$ with variable nodes $x_1,\ldots,x_n$ and check nodes $a_1,\ldots,a_{\vm_\eps},p_1,\ldots,p_{\THETA}$ 
by inserting an edge between $a_i$ and $x_j$ for each matching edge that joins the sets $\cbc{a_i}\times[\vk_i]$ and $\cbc{x_j}\times[\vd_j]$.
Additionally, check node $p_i$ is adjacent to $x_i$ for each $i\in[\THETA]$.
Since there may be several edges joining clones of the same variable and check node, $\G_\eps$ may be a multigraph.
Let $\vA_{\eps,n}=\A(\G_{\eps,n})$ be the random matrix induced by $\G_\eps$.
We observe that working with $\vA_{\eps,n}$ does not shift the rank significantly.

\begin{proposition}\label{Cor_lower}
For any function $\Theta=\Theta(\eps)\geq0$ we have
	$\limsup_{\eps\to0}\limsup_{n\to\infty}n^{-1}\abs{\nul(\vA_\eps)-\nul(\A)}=0$ in probability.
\end{proposition}

By construction, the degrees of the checks $a_i$ and the variables $x_j$ in $\G_\eps$ are upper-bounded by $\vk_i$ and $\vd_j$, respectively.
We thus refer to $\vk_i$ and $\vd_j$ as the {\em target degrees} of $a_i$ and $x_j$.
Indeed, since $\G_\eps$ will turn out to feature only few multi-edges \whp\ and $\vm_\eps$ is significantly smaller than $dn/k$ and thus $\sum_{i=1}^{\vm_\eps}\vk_i\leq\sum_{i=1}^{n}\vd_i$ \whp, most 
check nodes $a_i$ have degree precisely $\vk_i$ \whp\
But we expect that about $\eps dn$ variable nodes $x_i$ will have degree less than $\vd_i$.
In fact, \whp\ $\vec\Gamma_\eps$ fails to cover about $\eps dn$ clones from the set $\bigcup_{i=1}^n\{x_i\}\times[\vd_i]$.
Let us call such unmatched clones {\em cavities}.
The proof of the following upper bound, which constitutes the main technical achievment of the paper, rests on a subtle coupling of  $\vA_{\eps,n+1}$ and $\vA_{\eps,n}$.

\begin{proposition}\label{Prop_coupling}
For any $\eps>0$ there exists $\Theta=\Theta(\eps)>0$ such that for all large enough $n$ we have
$$\Erw[\nul(\vA_{\eps,n+1})]-\Erw[\nul(\vA_{\eps,n})]\leq\max_{\alpha\in[0,1]}\Phi(\alpha)+o_\eps(1).$$
\end{proposition}

\noindent
The coupling upon which the proof of \Prop~\ref{Prop_coupling} is based exploits the flexibility afforded by the 
likely presence of a linear number of cavities.
Roughly speaking, under the coupling $\vA_{\eps,n+1}$ is obtained from $\vA_{\eps,n}$ by adding one column and a (typically bounded) random number of rows.
The check nodes corresponding to the new rows will have non-zero entries at random cavities of $\vA_{\eps,n}$.
We will estimate the difference of the nullities via Fact~\ref{Fact_nullity}.
Indeed, the purpose of the check nodes $p_1,\ldots,p_{\THETA}$ is to ensure that the Boltzmann distribution
$\mu_{\vA_{\eps,n}}$ is sufficiently symmetric \whp\ 
More specifically, while \Lem~\ref{Lemma_0pinning} requires that a {\em random} set of $\THETA$ variables be frozen, the checks $p_1,\ldots,p_{\THETA}$ just freeze the first $\THETA$ variables.
But since the distribution of $\G_{\eps}-\{p_1,\ldots,p_{\THETA}\}$ is invariant under permutations of the variable nodes, 
both constructions are equivalent.
Therefore, we will be able to combine Fact~\ref{Fact_nullity}, \Lem~\ref{Lemma_Spartition} and \Lem~\ref{Lemma_0pinning} to prove
\Prop~\ref{Prop_coupling}.

As an immediate consequence of \Prop~\ref{Prop_coupling} we obtain the desired upper bound on the nullity.

\begin{corollary}\label{Prop_lower}
We have
	$\limsup_{\eps\to0}\limsup_{n\to\infty}\frac1n\Erw[\nul(\vA_\eps)]\leq\max_{\alpha\in[0,1]}\Phi(\alpha).$
\end{corollary}
\begin{proof}
 \Prop~\ref{Prop_coupling} yields
\begin{align*}
\frac1n\Erw[\nul(\vA_{\eps,n})]&=\frac1n\brk{\Erw[\nul(\vA_{\eps,1})]+\sum_{N=1}^{n-1}\Erw[\nul(\vA_{\eps,N+1})]-\Erw[\nul(\vA_{\eps,N})]}\leq\max_{\alpha\in[0,1]}\Phi(\alpha)+o_\eps(1).
\end{align*}
Taking the double limit $n\to\infty$ followed by $\eps\to0$ yields the assertion.
\end{proof}

\subsection{The interpolation method}
A lower bound on the nullity of $\A$  that matches the upper bound from \Prop~\ref{Cor_lower} and \Cor~\ref{Prop_lower} was deduced in~\cite{Lelarge} from the formula for the matching number of random bipartite graphs from~\cite{BLS2}.
But the proof of that formula, reliant on the contraction method in combination with local weak convergence,  is far from elementary.
Here we present a new direct proof of the lower bound.
We adapt another technique from mathematical physics, the interpolation method, to the rank problem.
The basic idea is to construct a family of random matrices $\vA_{\eps}(t)$ parametrised by $t\in[0,1]$.
At $t=1$ we obtain precisely the matrix $\vA_\eps$.
At the other extreme, $\vA_\eps(0)$ is a block diagonal matrix whose nullity can be read off easily.
To establish the lower bound we will control the derivative of the nullity with respect to $t$.
By comparison to applications of the interpolation method to other problems, the construction here is relatively elegant.
In particular, throughout the interpolation we will be dealing with an actual random matrix, rather than some other, more contrived object.

\begin{figure}
\includegraphics[height=4.5cm]{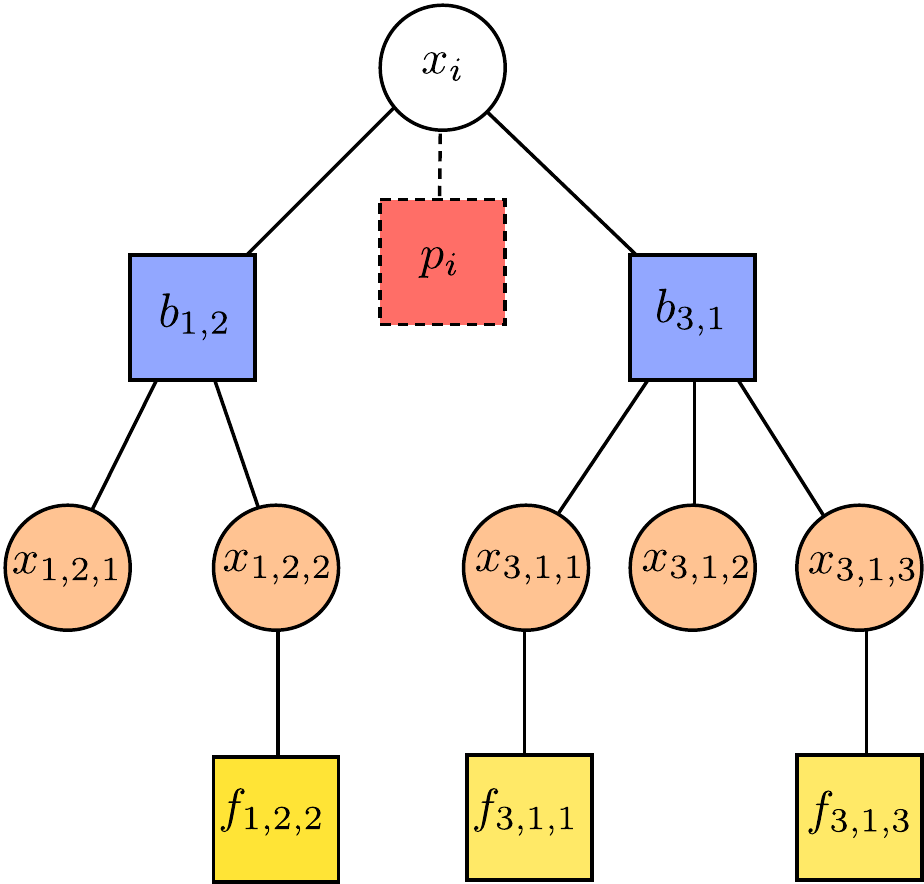}
\hfill\includegraphics[height=4.5cm]{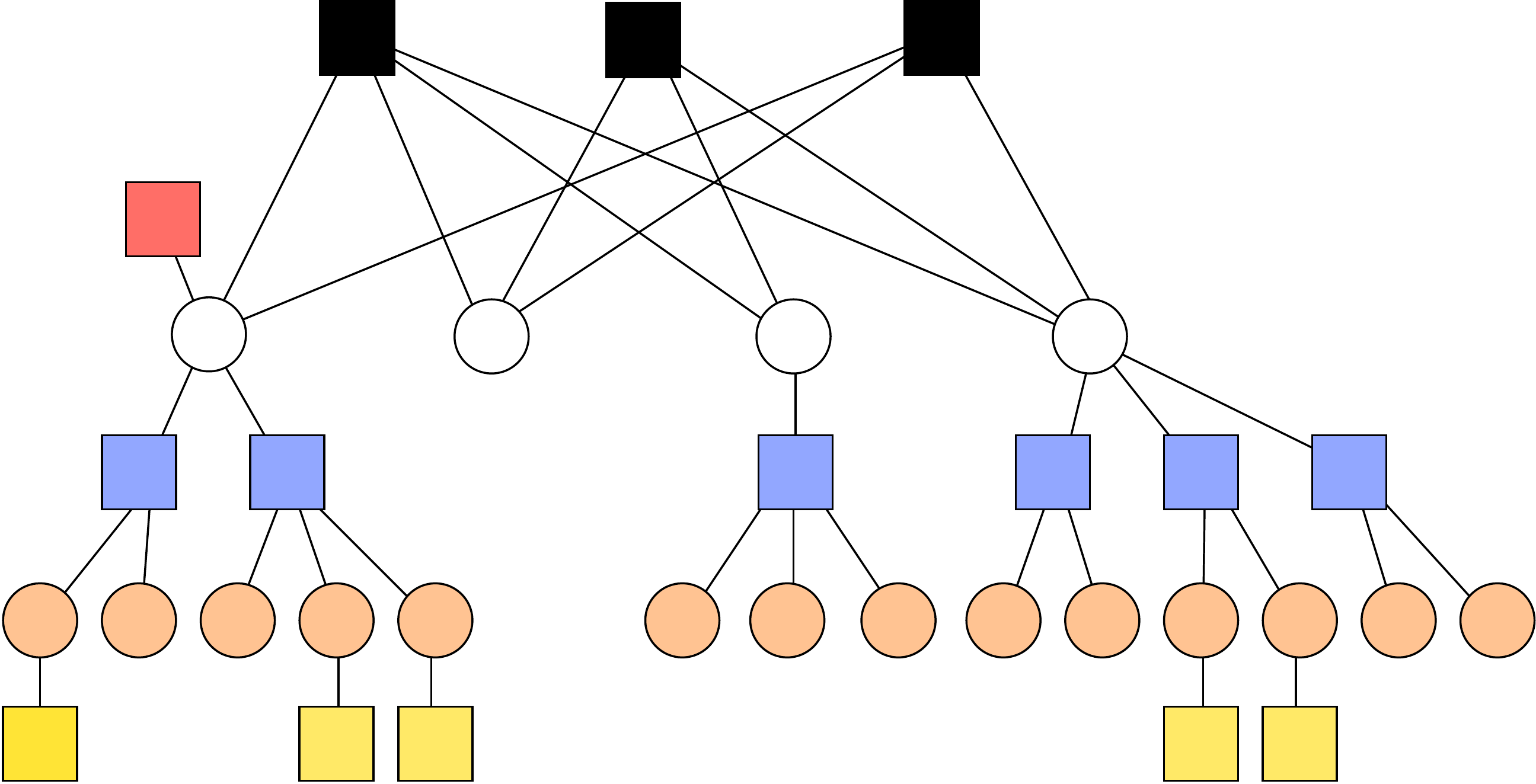}
\caption{Left: sketch of the component of $x_i$ at $t=0$; the check $p_i$ is present iff $i\leq\THETA$.
Right: sketch of the factor graph $\G_\eps(t)$ for $0<t<1$, with the $a_{i,j}$ coloured black and the other colours as in the left figure.}\label{Fig_interpolation}
\end{figure}

Let us inspect the construction in detail.
Apart from $t$ and $\eps$ we need two further parameters: an integer $\Theta=\Theta(\eps)\geq0$ and a real $\beta\in[0,1]$, chosen such that
\begin{equation}\label{eqbeta}
\Phi(\beta)=\max_{\alpha\in[0,1]}\Phi(\alpha).
\end{equation}	
Further, let
\begin{align}\label{eqinterpol}
\vm_\eps(t)&\disteq \Po((1-\eps)tdn/k),&
\vm_\eps'(t)&\disteq\Po((1-\eps)(1-t)dn/k)
\end{align}
be independent Poisson variables.
Also let $(\vk_i,\vk_i',\vk_i'')_{i\geq1}$ and $(\vd_i)_{i\geq1}$ be copies of $\vk$ and $\vd$, respectively, mutually independent and independent of $\vm_\eps(t)$, $\vm_\eps'(t)$.
Additionally, choose $\THETA\in[\Theta]$ uniformly and independently of everything else.

The Tanner graph $\G_\eps(t)$ has variable nodes 
\begin{align*}
x_1,\ldots,x_n&&\mbox{and}&&(x_{i,j,h})_{i\in[\vm_\eps'(t)],\,j\in[\vk_i],\,h\in[\vk_i'-1]}.
\end{align*}
Moreover, let $\cF_t$ be a random set that contains each of the variable nodes $x_{i,j,h}$ with probability $\beta$ independently.
Then the check nodes are 
\begin{align*}
a_1,\ldots,a_{\vm_\eps(t)},&&(b_{i,j})_{i\in[\vm_\eps'],\, j\in[\vk_i']},&&p_1,\ldots,p_{\THETA},&&
f_{i,j,h}\quad\mbox{ for each }x_{i,j,h}\in\cF_t.
\end{align*}
To define the edges of the Tanner graph let $\vec\Gamma_\eps(t)$ be a random maximal matching of the complete bipartite graph with vertex sets
\begin{align*}
\bigcup_{i=1}^n\cbc{x_i}\times[\vd_i],\qquad\bc{\bigcup_{i=1}^{\vm_\eps(t)}\cbc{a_i}\times[\vk_i]}
		\cup\cbc{b_{i,j}':i\in[\vm_\eps'(t)],\ j\in[\vk_i']}.
\end{align*}
For each matching edge $(x_i,s,a_j,t)\in\vec\Gamma_\eps(t)$ insert an edge between $x_i$ and $a_j$ into the Tanner graph and
for each $(x_i,s,b_{j,h})\in\vec\Gamma_\eps(t)$ insert an edge between $x_i$ and $b_{j,h}$.
Thus, $\G_\eps(t)$ may contain multi-edges.
Further, add an edge between $x_i$ and $p_i$ for $i=1,\ldots,\THETA$ and
add an edge between $x_{i,j,h}$ and $a_{i,j}$ as well as an edge between $x_{i,j,h}\in\cF_t$ and $f_{i,j,h}$.
Finally, let $\vA_{\eps}(t)$ be the random matrix induced by $\G_\eps(t)$.

The semantics is as follows.
The checks $a_i$ will play exactly the same role as before, i.e., each is adjacent to $\vk_i$ of the variable nodes $x_1,\ldots,x_n$ \whp\
By contrast, each $b_{i,j}$ is adjacent to precisely one of the variables $x_1,\ldots,x_n$.
In addition, $b_{i,j}$ is adjacent to the $\vk_i'-1$ variable nodes $x_{i,j,h}$, $h\in[\vk_i'-1]$.
These variable nodes, in turn, are adjacent only to $b_{i,j}$ and to $f_{i,j,h}$ if $x_{i,j,h}\in\cF$.
The checks $f_{i,j,h}$ are unary, i.e., $f_{i,j,h}$ simply forces $x_{i,j,h}$ to take the value zero.
Finally, each of the checks $p_i$ is adjacent to $x_i$ only, i.e., $p_1,\ldots,p_{\THETA}$ just freeze $x_1,\ldots,x_{\THETA}$.

For $t=1$ the Tanner graph contains $\vm_\eps(1)\disteq\Po((1-\eps)dn/k)$ `real' checks $a_i$ and none of the checks $b_{i,j}$ or $f_{i,j,h}$.
In effect, $\vA_\eps(1)$ is distributed precisely as $\vA_\eps$ from \Sec~\ref{Sec_outline1}.
By contrast, at $t=0$ we have $\vm_\eps(0)=0$, i.e., there are no checks $a_i$ involving several of the variables $x_1,\ldots,x_n$.
As a consequence, the Tanner graph decomposes into $n$ connected components, one for each of the $x_i$.
In fact, each component is a tree comprising $x_i$, some of the checks $b_{j,h}$ and their proprietary variables $x_{j,h,s}$ along with possibly a check $f_{j,h,s}$ that freezes $x_{j,h,s}$ to zero.
For $i\in[\THETA]$ there is a check $p_i$ freezing $x_i$ to zero as well.
Thus, $\vA_\eps(0)$ is a block diagonal matrix consisting of $n$ blocks, one for each component.
In effect, the rank of $\vA_\eps(0)$ will be easy to compute.
Finally, for $0<t<1$ we have a blend of the two extremal cases.
There will be some checks $a_i$ and some $b_{i,j}$ with their retainer variables and checks;  see Figure~\ref{Fig_interpolation}.

We are going to study the nullity of $\vA_\eps(t)$ for $t\in[0,1]$.
But since the newly introduced variables $x_{i,j,h}$ inflate the nullity, we subtract a correction term to retain the same scale throughout the process.
In addition, we need a correction term to make up for the greater total number of check nodes in $\vA_\eps(0)$ by comparison to $\vA_\eps(1)$.
Thus, let
\begin{align*}
\cN_t&=\nul\vA_{\eps}(t)+|\cF_t|-\sum_{i=1}^{\vm_\eps'(t)}\vk_i'(\vk_i'-1),&
\cY_t&=\sum_{i=1}^{\vm_\eps}(\vk_i-1)(\beta^{\vk_i}-1).
\end{align*}
The following two statements summarise the interpolation argument.
First, we compute $\Erw[\cN_0]$.

\begin{lemma}\label{Lemma_interpol2}
For any fixed $\theta\geq0$ we have
	$n^{-1}\Erw[\cN_0]=D(1-K'(\beta)/k)+dK'(\beta)/k-d+o_\eps(1).$
\end{lemma}

\noindent
The next proposition provides monotonicity.

\begin{proposition}\label{Lemma_interpolation}
For any $\eps>0$ there exists $\Theta=\Theta(\eps)>0$ such that 
$\frac1n\frac\partial{\partial t}\Erw[\cN_t+\cY_t]\geq o_\eps(1)$
uniformly $t\in(0,1)$.
\end{proposition}

\noindent
As an immediate consequence of \Lem~\ref{Lemma_interpol2} and \Prop~\ref{Lemma_interpolation} we obtain a lower bound on the nullity that matches the upper bound from \Prop~\ref{Prop_lower}.

\begin{corollary}\label{Prop_upper}
We have
	$\limsup_{\eps\to0}\limsup_{n\to\infty}\frac1n\Erw[\nul(\vA_\eps)]\geq\max_{\alpha\in[0,1]}\Phi(\alpha).$
\end{corollary}
\begin{proof}
\Lem~\ref{Lemma_interpolation} implies that
\begin{align}\label{eqProp_upper_1}
\Erw[\nul \vA_\eps]&=\Erw[\nul \vA_\eps(1)]=\Erw[\cN_1]=\Erw[\cN_1+\cY_1]-\Erw[\cY_1]
		\geq \Erw[\cN_0+\cY_0]-\Erw[\cY_1]-o_\eps(n)=\Erw[\cN_0]-\Erw[\cY_1]-o_\eps(n).
\end{align}
Further,
\begin{align*}%\label{eqProp_upper_2}
\frac1n\Erw[\cY_1]&=\frac dk\bc{\beta K'(\beta)-k+1-K(\beta)}+o_\eps(1),&
\frac1n\Erw[\cN_0]&=-d+dK'(\beta)/k+D(1-K'(\beta)/k)+o_\eps(1)&&\mbox{[by \Lem~\ref{Lemma_interpol2}]}.%\label{eqProp_upper_3}
\end{align*}
Hence, 
$n^{-1}(\Erw[\cN_0]-\Erw[\cY_1])=\Phi(\beta)+o_\eps(1)=\max_{\alpha\in[0,1]}\Phi(\alpha)+o_\eps(1)$
by~\eqref{eqbeta}.
Thus, the assertion follows from \eqref{eqProp_upper_1}.
\end{proof}

\noindent
Combining \Prop~\ref{Cor_lower}, \Prop~\ref{Prop_lower} and \Cor~\ref{Prop_upper}, and a standard concentration argument for $\nul{\vA_{\eps}}$ (see Lemma~\ref{Lemma_conc}), we complete the proof of \Thm~\ref{thm:rank}.

\subsection{Discussion}
Much of the prior work on the rank of random matrices over $\FF_q$ relies on relatively elementary techniques such as the second moment method~\cite{Dietzfelbinger,DuboisMandler,GoerdtFalke}, or the idea of bounding the number of linearly dependent row sets via the first moment method~\cite{Kolchin,PittelSorkin}.
Other proof strategies depend on graph-theoretic arguments such as close control of the 2-core and the `mantle'~\cite{CFP}.
By contrast, the present paper harnesses two ideas from mathematical physics, the Aizenman-Sims-Starr scheme and the interpolation method.
Both were originally invented to investigate the Sherrington-Kirkpatrick spin glass model~\cite{Aizenman,Guerra}.
Yet over the recent years these techniques have found several uses in `diluted' models defined on sparse random structures, e.g.,~\cite{bayati,FranzLeone,LelargeInterpolation,MontanariBounds}.
In some of these papers the idea of carving out `cavities' to facilitate coupling arguments is used as well.
Also the combintion of the Aizenman-Sims-Starr scheme and the interpolation method to prove matching upper and lower bounds has been applied  successfully to other problems such as the Viana--Bray spin glass model or the stochastic block model (e.g., \cite{CKPZ,BetheLattices,Panchenko}).

Yet the only prior application of the Aizenman-Sims-Starr scheme to the rank problem that we are aware of 
is our previous paper on random matrices with $\vk=k$ constant and $\vd\disteq\Po(d)$~\cite{Ayre}.
The Aizenman-Sims-Starr scheme was used there to derive an upper bound on the nullity, like in the present paper.
But in the present setting matters are complicated very significantly by the fact that we work with general degree sequences.
Indeed, while the Poisson distribution lends itself easily to coupling arguments due to its memorylessness, in the present paper the couplings require delicate manoeuvres, as we will see in \Sec~\ref{Sec_lower}.
The possible presence of nodes of very high degrees adds to the intricacy.
While coupling arguments have previously been developed for graphs with given degrees (e.g.,~\cite{Dembo,LelargeInterpolation,MontanariBounds}),
a new subtle construction is needed to carry out the very accurate calculations required for the Aizenman-Sims-Starr scheme.

The use of the interpolation method to lower bound the nullity is a further technical novelty
as, to our knowledge, this technique has not been applied to the rank problem previously.
Indeed, in most prior work there was no need for a sophisticated lower bound argument
because the graph-theoretic bound~\eqref{eq2corebound}  was tight~\cite{Ayre,CFP}.
The interpolation scheme for the rank problem is conceptually more elegant than prior applications of the method to other problems.
The reason is that normally the interpolation is set up in terms of the Bethe free energy functional that `lives' on an infinite dimensional space of probability measures; cf.~\Sec~\ref{Sec_cavity}.
Then the construction of the interpolation scheme has to incorporate a distributional parameter $\pi$.
In effect, the combinatorial interpretation of the structures at `times' $t\in(0,1)$ is not exactly straightforward.
By contrast, because in the rank problem the infinite dimensional variational problem collapses to a one-dimensional optimisation, the interpolation argument merely requires a real parameter $\beta\in[0,1]$ and the intermediate structures $\vA_\eps(t)$ are actual matrices.

Finally, it would be interesting to see if the present methods can be extended to other problems of an algebraic flavour.
Natural candidates would be systems of linear equations over rings rather than fields or systems of equations over non-Abelian groups.

\subsection{Overview}
We proceed to prove \Prop~\ref{Prop_lower} in \Sec~\ref{Sec_lower}.
Subsequently in \Sec~\ref{Sec_interpolation} we deal with the proofs of \Lem~\ref{Lemma_interpol2} and \Prop~\ref{Lemma_interpolation}.
Further, \Sec~\ref{Sec_conc} contains the proofs of \Prop s~\ref{Lemma_welldef}, \ref{Cor_lower} and~\ref{Lemma_nulconc} and \Thm~\ref{Thm_LDPC}.
Moreover,  the proof of \Thm~\ref{Thm_tight} can be found in \Sec~\ref{Sec_tight}.
Finally, \Sec~\ref{sec:2core} contains the proof of \Thm~\ref{thm:2core}.

\section{The Aizenman-Sims-Starr scheme}\label{Sec_lower}

\noindent
In this section we prove \Prop~\ref{Prop_lower}.
As set out in \Sec~\ref{Sec_outline1}, we are going to bound the difference of the nullities of $\vA_{\eps,n+1}$ and $\vA_{\eps,n}$ via Fact~\ref{Fact_nullity}.
We begin by coupling the random variables $\nul(\vA_{\eps,n+1})$ and $\nul(\vA_{\eps,n})$.

\subsection{The coupling}
Let $\vM=(\vM_j)_{j\geq3}$ and $\DELTA=(\DELTA_j)_{j\geq3}$ be sequences of Poisson variables with means
\begin{align}\label{eqPoissons}
\Erw[\vM_j]&=(1-\eps)\pr\brk{\vk=j}dn/k,&\Erw[\DELTA_j]&=(1-\eps)\pr\brk{\vk=j}d/k.
\end{align}
All of these random variables are mutually independent and independent of $\THETA$ and the $(\vd_i)_{i\geq1}$.
Further, let
\begin{align}\label{eqm}
\vM_j^+&=\vM_j+\DELTA_j,
%&\vM_j^-&=(\vM_j-\GAMMA_j)\vee 0,
&\vm_\eps&=\sum_{j\geq3}\vM_j,&\vm_\eps^+&=\sum_{j\geq3}\vM_j^+.
\end{align}
Since  $\sum_{j\geq3}\vM_j\disteq\Po((1-\eps)dn/k)$, (\ref{eqm}) is consistent with the earlier convention that $\vm_\eps\disteq\Po((1-\eps)dn/k)$.

The random vectors $\vd,\vM$ naturally define a random Tanner (multi-)graph $\G_{n,\vM}$ with variable nodes $x_1,\ldots,x_n$ 
and check nodes $p_1,\ldots,p_{\THETA}$ and $a_{i,j}$, $i\geq3$, $j\in[\vM_i]$.
Its edges are induced by a random maximal matching  $\vec\Gamma_{n,\vM}$ of the complete bipartite graph with vertex classes
\begin{align*}
\bigcup_{h=1}^n\{x_h\}\times[\vd_h]\quad\mbox{and}\quad\bigcup_{i\geq3}\bigcup_{j=1}^{\vM_i}\{a_{i,j}\}\times[i].
\end{align*}
Each matching edge $(x_h,s,a_{i,j},t)\in \vec\Gamma_{n,\vM}$ induces an edge between $x_h$ and $a_{i,j}$ in the Tanner graph.
In addition, there is an edge between $p_i$ and $x_i$ for every $i\in[\THETA]$.
Let $\vA_{n,\vM}=\A(\G_{n,\vM})$ be the corresponding random matrix.
The random matrix $\vA_{n+1,\vM^+}$ and its associated Tanner graph $\G_{n+1,\vM^+}$ are defined analogously.

\begin{lemma}\label{Lemma_ComplicatedModel}
For any $\theta>0$ we have
$\Erw[\nul(\vA_{\eps,n})]=\Erw[\nul(\vA_{n,\vM})]$, 
$\Erw[\nul(\vA_{\eps,n+1})]=\Erw[\nul(\vA_{n+1,\vM^+})]$.
\end{lemma}
\begin{proof}
We defined $\vA_{\eps,n}$ as the $n\times\vm_\eps$-matrix with target column and row degrees drawn from $\vd$ and $\vk$ independently with a $\THETA\times\THETA$ identity matrix attached at the bottom.
In effect, because $\vm_\eps$ is a Poisson variable, the number of rows of with target degree $j$ is distributed as $\vM_j$, and these numbers are mutually independent.
Hence, $\vA_{\eps,n}$ and $\vA_{n,\vM}$ are identically distributed, and so are their nullities.
The same argument applies to $\vA_{\eps,n+1}$.
\end{proof}

Up to this point we merely introduced a new description of $\vA_{\eps,n}$ and $\vA_{\eps,n+1}$.
To actually construct a coupling we introduce a third random matrix whose nullity we can easily compare to 
$\nul(\vA_{n,\vM})$ and $\nul(\vA_{n+1,\vM^+})$.
Specifically, let $\GAMMA_i\ge0$ be the number of checks $a_{i,j}$, $j\in[\vM_i^+]$, adjacent to the last variable node $x_{n+1}$ in $\G_{n+1,\vM^+}$.
Also let $\GAMMA=(\GAMMA_i)_{i\geq3}$ and set
\begin{equation}\label{eqminus}
\vM_i^-=(\vM_i-\GAMMA_i)\vee 0.
\end{equation}
Consider the random Tanner graph $\G'=\G_{n,\vM^-}$ and the corresponding random matrix  $\vA'=\vA_{n,\vM^-}$ induced by a random matching $\vec\Gamma_{n,\vM^-}$ as above.
For each variable $x_i$, $i=1,\ldots,n$, let $\cC$ be the set of clones from $\bigcup_{i\in[n]}\{x_i\}\times[\vd_i]$ that $\vec\Gamma_{n,\vM^-}$ leaves unmatched.
We call the elements of $\cC$ {\em cavities}.

Now, obtain the Tanner graph $\G''$ from $\G'$ by adding 
new check nodes $a''_{i,j}$ with target degree $i$ for each $i\geq3$, $j\in[\vM_i-\vM_i^-]$.
The new checks are joined by a random maximal matching of the complete bipartite graph with vertex classes $\cC$ and
 $$\bigcup_{i\geq3}\bigcup_{j\in[\vM_i-\vM_i^-]}\{a_{i,j}''\}\times[i],$$
i.e., for each matching edge we insert a corresponding variable-check edge.
Then $\vA''$ is obtained from $\vA'$ by adding rows corresponding to the new checks and representing each new edge of $\G''$ by a matrix entry chosen independently according to $\CHI$.
Thus, the matrix entries of $\vA'$, $\vA''$ corresponding to the edges of $\G'$ coincide.

Analogously, obtain $\G'''$ by adding one variable node $x_{n+1}$ as well as check nodes $a_{i,j}'''$, $i\geq3$, $j\in[\GAMMA_i]$
and $b_{i,j}'''$, $i\geq3$, $j\in[\vM_i^+-\vM_i^--\GAMMA_i]$ to $\G'$.
The new checks are connected to $\G'$ via a random maximal matching of the complete bipartite graph with vertex classes $\cC$ and
 $$\bigcup_{i\geq3}\bc{\bigcup_{j\in[\vM_i-\vM_i^-]}\{a_{i,j}'''\}\times[i-1]
			\cup\bigcup_{j\in[\vM_i^+-\vM_i^--\GAMMA_i]}\{b_{i,j}'''\}\times[i]}.$$
For each matching edge we insert the corresponding variable-check edge and
in addition each of the check nodes $a_{i,j}'''$ gets connected to $x_{n+1}$ by exactly one edge.
Finally, $\vA'''$ is obtained by adding one row for each of the new checks as well as one column representing $x_{n+1}$.
As always, the entries representing the new edges of the Tanner graph are drawn independently according to $\CHI$.

\begin{lemma}\label{Lemma_valid}
We have  $\Erw[\nul(\vA'')]=\Erw[\nul(\vA_{n,\vM})]+o(1)$ and  $\Erw[\nul(\vA''')]=\Erw[\nul(\vA_{n+1,\vM^+})]+o(1).$
\end{lemma}

\noindent
The proof of \Lem~\ref{Lemma_valid} is tedious but straightforward.
We defer it to \Sec~\ref{Sec_Lemma_valid}.

As a next step we are going to calculate the  differences $\nul(\vA''')-\nul(\vA')$ and $\nul(\vA'')-\nul(\vA')$.
We obtain expressions of one parameter of $\vA'$, namely the fraction of cavities `frozen' to zero. 
Naturally, each vector $\sigma\in\FF_q^{\{x_1,\ldots,x_{n}\}}$ lifts to $\sigma\in\FF_q^{\cC}$ via $(x_i,h)\in\cC\mapsto\sigma(x_i)$; in words, the value of the clone $(x_i,h)$ is nothing but the value of the underlying variable $x_i$.
Accordingly, the Boltzmann distribution $\mu_{\vA'}$ induces a probability distribution $\mu_{\vA',\cC}$ on $\FF_q^{\cC}$:
\begin{align*}
\mu_{\vA',\cC}(\tau)&=
q^{-\nul(\vA')}\abs{\cbc{\sigma\in\ker(\vA'):\forall (x_i,h)\in\cC:\sigma_i=\tau_{x_i,h}}}
&(\tau\in\FF_q^{\cC}).
\end{align*}
Further, \Lem~\ref{Lemma_Spartition} shows that for each cavity $(x_i,h)$ there are two possibilities: either the Boltzmann marginal $\mu_{\vA',x_i}$ is the uniform distribution on $\FF_q$, or $\mu_{\vA',x_i}$ is the point mass on zero.
In the latter case we call the cavity $(x_i,h)$ {\em frozen} in $\vA'$.
Let $\cF\subset\cC$ be the set of all frozen cavities.
Finally, let $\vec\alpha=|\cF|/|\cC|$;
in the unlikely event that $\cC=\emptyset$, we agree that $\vec\alpha=0$.
In \Sec s~\ref{Sec_A'''} and~\ref{Sec_A''} we are going to establish the following two estimates.

\begin{lemma}\label{Lemma_A'''}
We have  $\Erw[\nul(\vA''')-\nul(\vA')]=\Erw[D(1- K'(\vec\alpha)/k)+d(K'(\vec\alpha)+K(\vec\alpha)-1)/k]-d+o_\eps(1).$
\end{lemma}

\begin{lemma}\label{Lemma_A''}
We have $\Erw[\nul(\vA'')-\nul(\vA')]=d\Erw[\vec\alpha K'(\vec\alpha)]/k-d+o_\eps(1).$
\end{lemma}

\noindent
\Prop~\ref{Prop_coupling} is an immediate consequence of \Lem s~\ref{Lemma_ComplicatedModel}--\ref{Lemma_A''}.

While proving \Lem s~\ref{Lemma_A'''} and~\ref{Lemma_A''} in full detail requires a fair bit of work because we are dealing with very general degree distributions $\vd,\vk$, it is not at all difficult to fathom where the right hand side expressions in \Lem s~\ref{Lemma_A'''}--\ref{Lemma_A''} come from.
Regarding \Lem~\ref{Lemma_A''}, we notice that  with probability $1-o_\eps(1)$ the degree of $x_{n+1}$ in $\G_{n,\vM^+}$ is $\vd_{n+1}$.
Moreover, the degrees of the neighbours of $x_{n+1}$ are distributed approximately as the size-biased version $\hat\vk$ of $\vk$ because the probability that a given clone of $x_{n+1}$ is matched to a clone of a specific check node is proportional to the degree of that check node.
Hence, the degrees of the new checks added to $\vA''$ should approximately be distributed as $\hat\vk_1,\ldots,\hat\vk_{\vd}$.
These new checks are attached to random cavities of $\vA'$, each of which is frozen with probability approximately $\ALPHA$.
Hence, for each $i\in[\vd]$ the probability of picking $\hat\vk_i$ frozen cavities should be about $\ALPHA^{\hat\vk_i}$, and in this case the new check will be satisfied by all vectors in the kernel of $\vA'$.
By contrast, if at least one adjacent cavity is unfrozen, and if we assume as per \Lem s~\ref{Lemma_Spartition} and~\ref{Lemma_0pinning} that the values that a random $\SIGMA\in\ker(\vA')$ assigns to the cavities are uniform and independent, then the probability that $\SIGMA$ satisfies the new check equals $1/q$.
Thus, by Fact~\ref{Fact_nullity} the nullity drops by one for each such check.
In summary, this heuristic calculation leads us to expect that $\Erw[\nul(\vA'')-\nul(\vA')]=\Erw[\sum_{i=1}^{\vd}(\ALPHA^{\hat\vk_i}-1)]+o_\eps(1)$.
Rewriting this expression in terms of the generating function $K$ yields the expression displayed in \Lem~\ref{Lemma_A''}.

Similar reasoning explains the expression in \Lem~\ref{Lemma_A'''}.
Indeed, recalling~\eqref{eqPoissons} and following along the lines of the previous paragraph, we expect that the addition of the checks $b_{i,j}'''$ will change the nullity by 
\begin{align}\label{eqInt1}
d\Erw[\ALPHA^{\vk}-1]/k+o_\eps(1).
\end{align}
Moreover, concerning the addition of $x_{n+1}$ and its adjacent checks $a_{i,j}'''$, there are two possible scenarios.
First, that all the cavities adjacent to some $a_{i,j}'''$ are frozen in $\A'$.
Then this check can only be satisfied by setting $x_{n+1}$ to zero.
In effect, $a_{i,j}'''$ freezes $x_{n+1}$.
Then any other check $a_{i,h}'''$ with all-frozen neighbours will be satisfied automatically.
Otherwise, if $a_{i,h}'''$ has at least one unfrozen neighbour, and if we assume as per \Lem s~\ref{Lemma_Spartition} and~\ref{Lemma_0pinning} that in $\SIGMA\in\ker(\vA')$ the unfrozen neighbours take mutually independent uniform values in $\FF_q$, the probability that  $\SIGMA\in\ker(\vA')$ satisfies $a_{i,h}'''$ equals $1/q$.
Thus, Fact~\ref{Fact_nullity} shows that the nullity drops by one for each such check.
Hence, because as in the previous paragraph the degrees of the checks adjacent to $x_{n+1}$ are approximately distributed as $\hat\vk_1,\ldots,\hat\vk_{\vd}$, by inclusion/exclusion the contribution of the `$x_{n+1}$ frozen' scenario comes to
\begin{align}\label{eqInt2}
\Erw\brk{\vd\prod_{i=1}^{\vd}(1-\ALPHA^{\hat\vk_i-1}) +\sum_{i=1}^{\vd}(\ALPHA^{\hat\vk_i-1}-1)}+o_\eps(1).
\end{align}
Second, there is the scenario that $x_{n+1}$ remains unfrozen.
Then each of the adjacent checks contains at least one unfrozen variable among $x_1,\ldots,x_n$.
Assuming that the values that $\SIGMA\in\ker(\vA')$ assigns to these variables are uniform and independent
(again by \Lem s~\ref{Lemma_Spartition} and~\ref{Lemma_0pinning}),
we deduce from Fact~\ref{Fact_nullity} that every check reduces the nullity by one, while the presence of $x_{n+1}$ adds one to the nullity.
Hence, the unfrozen case contributes
\begin{align}\label{eqInt3}
\Erw\brk{(1-\vd)\prod_{i=1}^{\vd}(1-\ALPHA^{\hat\vk_i-1})}+o_\eps(1).
\end{align}
Summing \eqref{eqInt1}, \eqref{eqInt2} and \eqref{eqInt3} and rewriting in terms of generating functions renders the term shown in \Lem~\ref{Lemma_A'''}.

We proceed to prove \Lem s~\ref{Lemma_A'''}--\ref{Lemma_A''} formally.
This requires a bit of groundwork.

\subsection{Preparations}
We establish two statements that pave the way for the proofs of \Lem s~\ref{Lemma_A'''} and~\ref{Lemma_A''}.
For a cavity $c=(x_i,h)\in\cC$ let $\mu_{\vA',c}=\mu_{\vA',x_i}$ denote the marginal of the cavity $c$ in the distribution $\mu_{\vA',\cC}$.
Similarly, for $c_1,\ldots,c_\ell\in\cC$ let $\mu_{\vA',c_1,\ldots,c_\ell}\in\cP(\FF_q^\ell)$ be the joint distribution of the underlying variables.
The following lemma shows that the joint distribution of a bounded number of random cavities (drawn with replacement) likely factorises, providing that the parameter $\Theta$ is chosen sufficiently large.

\begin{lemma}\label{Lemma_theta}
For any $\delta,\ell>0$ there is $\Theta=\Theta(\delta,\ell)>0$ such that with probability at least $1-\delta$ we have
\begin{align}\label{eqLemma_theta1}
\sum_{c_1,\ldots,c_\ell\in\cC}\dTV\bc{\mu_{\vA',c_1,\ldots,c_\ell},\bigotimes_{i=1}^\ell\mu_{\vA',c_i}}&\leq\delta|\cC|^\ell.
\end{align}
\end{lemma}
\begin{proof}
The construction \eqref{eqm} of $\vM,\vM^+$ and \Lem~\ref{Lemma_sums} ensure that $|\cC|\geq\eps n/2$ \whp\
Moreover, since $\Erw[\vd]=O_\eps(1)$ we find $L=L(\eps,\delta,\ell)>0$ such that
the event  $\cL=\cbc{\sum_{i=1}^n\vd_i\vecone\{\vd_i>L\}<\eps\delta^2 n/(16\ell)}$ has probability at least $1-\delta/8$.
Thus, we may condition on the event $\cE=\cL\cap\{|\cC|\geq\eps n/2\}$.

On $\cE$ let $\vy\in\{x_1,\ldots,x_n\}$ be a variable node chosen from the distribution
$\pr[\vy=x_i\mid\vA']=|\cC\cap(\{x_i\}\times[\vd_i])|/|\cC|.$
Further, let $\vy_1,\ldots,\vy_\ell$ independent copies of $\vy$.
Then to prove \eqref{eqLemma_theta1} it suffices to show that
\begin{align}\label{eqLemma_theta2}
\Erw\brk{\dTV\bc{\mu_{\vA',\vy_1,\ldots,\vy_\ell},\bigotimes_{i=1}^\ell\mu_{\vA',\vy_i}}\,\bigg|\,\cE}&\leq\delta^2/2.
\end{align}
To see this, let $\vx_1,\ldots,\vx_\ell$ be a sequence of $\ell$ independently and uniformly chosen variables from $x_1,\ldots,x_n$.
Then for $\cW\subset\{x_1,\ldots,x_n\}^\ell$ we have {on the event }$\cE$,
\begin{align}\nonumber
\pr\brk{(\vy_1,\ldots,\vy_\ell)\in\cW\mid\vA'}&
	\leq (L/\eps)^\ell \pr\brk{(\vx_1,\ldots,\vx_\ell)\in\cW\mid\vA'}+\frac{\eps\delta^2 n}{16\ell}\cdot\frac{\ell}{\eps n/2}
	\\&\leq (L/\eps)^\ell \pr\brk{(\vx_1,\ldots,\vx_\ell)\in\cW\mid\vA'}+\delta^2/8\label{eqLemma_theta3}
\end{align}
Furthermore, since the distribution of $\G'-\{p_1,\ldots,p_{\THETA}\}$ is invariant under permutations of the variable nodes, \Lem~\ref{Lemma_0pinning} shows that
for sufficiently large $\Theta$,
\begin{align}\label{eqLemma_theta4}
\Erw\brk{\dTV\bc{\mu_{\vA',\vx_1,\ldots,\vx_\ell},\bigotimes_{i=1}^\ell\mu_{\vA',\vx_i}}\,\bigg|\,\cE}&\leq\frac{\delta^3}{64}\bcfr{\eps}{L}^\ell.
\end{align}
Thus, \eqref{eqLemma_theta2} follows from \eqref{eqLemma_theta3} and \eqref{eqLemma_theta4}.
\end{proof}

To prove \Lem s~\ref{Lemma_A'''} and~\ref{Lemma_A''} we need to study the impact on the nullity of attaching one new column and a few rows.
This requires explicit knowledge of their degrees.
Let $(\hat\vk_i)_{i\geq1}$ be a sequence of copies of $\hat\vk$, mutually independent and independent of everything else.
Moreover, let $\hat\GAMMA_j=\sum_{i=1}^{\vd_{n+2}}\vecone\{\hat\vk_i=j\}$ and $\hat\GAMMA=(\hat\GAMMA_j)_{j\geq1}$.
Additionally, let $\hat\DELTA=(\hat\DELTA_j)_{j\geq3}$ be a family random variables, mutually independent and independent of everything else, with distributions $\hat\DELTA_j\sim\Po((1-\eps)\pr\brk{\vk=j}d/k)$.
Further, let $\Sigma'$ be the $\sigma$-algebra generated by $\G'$, $\A'$, $\vM_-$ and $(\vd_i)_{i\in[n]}$.
We write $\GAMMA\mid\Sigma',\DELTA\mid\Sigma'$ for the conditional versions of $\GAMMA,\DELTA$ given $\Sigma'$.

\begin{lemma}\label{Cor_gamma}
With probability $1-\exp(-\Omega_\eps(1/\eps))$ we have
$\dTV(\GAMMA\mid\Sigma',\hat\GAMMA)+\dTV(\DELTA\mid\Sigma',\hat\DELTA)=O_\eps(\eps^{1/2})$.
\end{lemma}
\begin{proof}
We begin by studying the unconditional distributions of $\GAMMA$ and $\DELTA$.
Let $\zeta=(\sum_{i\geq3}i\vM_i^+)/(\sum_{i=1}^{n+1}\vd_i)$.
The choice (\ref{eqm}) of the $\vM_i^+$ and \Lem~\ref{Lemma_sums} ensure that
$\pr\brk{1-2\eps\leq\zeta\leq1-\eps/2}=1-o(1)$.
Further, given $1-2\eps\leq\zeta\leq1-\eps/2$ we can think of $\G_{n,\vM^+}$ as being generated by the following experiment.
\begin{enumerate}[(i)]
\item Choose a set $\vec C\subset\bigcup_{h=1}^{n+1}\cbc{x_h}\times[\vd_h]$ of size $(1-\zeta) \sum_{i=1}^{n+1}\vd_i$ uniformly at random.
\item Create a random perfect matching $\vec\Gamma^\star$ of the complete bipartite graph with vertex classes
\begin{align*}
\bc{\bigcup_{h=1}^{n+1}\cbc{x_h}\times[\vd_h]}\setminus\vec C\quad\mbox{and}\quad
\bigcup_{i\geq3}\bigcup_{j=1}^{\vM_i^+}\cbc{a_{i,j}}\times[i].
\end{align*}
\item Obtain $\G^\star$ with variable nodes $x_1,\ldots,x_{n+1}$ and check nodes $a_{i,j}$, $i\geq3$, $j\in[\vM_i^+]$ by inserting an edge between $x_h$ and $a_{i,j}$ for any edge of $\vec\Gamma^\star$ that links $\{x_h\}\times[\vd_h]$ to $\{a_{i,j}\}\times[i]$.
\end{enumerate}
In other words, in the first step we designate the set of $\vec C$ of cavities and in the next two steps we connect the non-cavities randomly.

By way of this alternative description we can easily get a grip on the degree of $x_{n+1}$.
Indeed, given that $\vd_{n+1}\leq\eps^{-1/2}$, the probability that one of the clones $\{n+1\}\times[\vd_{n+1}]$ ends up in $\vC$ is $O_\eps(\eps^{1/2})$.
Hence, the actual degree $\vd^\star_{n+1}$ of $x_{n+1}$ in $\G^\star$ satisfies
\begin{align}\label{eqnasty-1}
\dTV\bc{\vd^\star_{n+1}\mid\{\vd_{n+1}\leq\eps^{-1/2}\},\vd}&=O_\eps(\eps^{1/2}).
\end{align}
Regarding the degrees of the checks adjacent to $x_{n+1}$,
by the principle of deferred decisions we can construct $\vec\Gamma^\star$ by matching one variable clone at a time, starting with the clones $\{x_{n+1}\}\times[\vd_{n+1}]$.
Because $\vk$ has a finite mean, given $\vd_{n+1}\leq\eps^{-1/2}$
we find a fixed number $L$ such that with probability $1-O_\eps(\eps^{-1})$ all checks adjacent to $x_{n+1}$ have degree at most $L$.
Moreover, Chebyshev's inequality shows that $\vM_i^+=(1-\eps)\pr\brk{\vk=i}dn/k+o(n)$ for all $i\leq L$ 
and $\sum_{i\geq3}i\vM_i^+=(1-\eps)dn+o(n)$ \whp\
Since the probability that $\vec\Gamma^\star$ links a given clone of $x_{n+1}$ to a specific check is proportional to its degree, we conclude that
\begin{align}\label{eqnasty0}
\dTV(\GAMMA\mid\Sigma',\hat\GAMMA)=O_\eps(\eps^{1/2}).
\end{align}
Moreover, it is immediate from the construction that the unconditional $\DELTA$ is distributed as $\hat\DELTA$.

To complete the proof we are going to argue that $\vM^-,\vec d_1,\ldots,\vec d_n$ and $\GAMMA,\DELTA$ are asymptotically independent.
Arguing along the lines of the previous paragraph, we find that for large  $L>0$ the event $\cK=\{\sum_{i\geq3}i(\DELTA_i+\GAMMA_i)\leq L\}$ occurs with probability $\pr[\cK]\geq1-\exp(-1/\eps^2)$.
Consequently, the event $\cL=\{\pr[\cK\mid\vM^-,\vd_1,\ldots,\vd_n]\geq1-\exp(-1/\eps)\}$ satisfies $\pr[\cL]\geq1-\exp(-1/\eps)$.
Moreover, since $\vM$ comprises independent Poisson variables, the event
$$\cM=\{\forall i\leq L:|\vM_i^--\Erw[\vM_i]|\leq\sqrt n\ln n\}\cap\cbc{\sum_{i=1}^n\vd_i=(1-\eps)dn+o(n)}$$
satisfies $\pr[\cM\mid\cK]\sim1$.
In summary,
\begin{align}\label{eqnasty1}
\pr\brk{\cK}&\geq\exp(-1/\eps^2),&\pr\brk{\cL}&\geq1-\exp(-1/\eps),&\pr\brk{\cM\mid\cK}&=1-o(1).
\end{align}
We are going to show that for any outcomes $(M^-,d_1,\ldots,d_n)\in\cL\cap\cM$ and $(\gamma,\Delta)\in\cK$,
\begin{align}\label{eqnasty2}
\pr\brk{\GAMMA=\gamma,\DELTA=\Delta\mid\vM^-=M^-,\forall i\in[n]:\vd_i=d_i}
	&=(1+O_\eps(\eps))\pr\brk{\GAMMA=\gamma,\DELTA=\Delta}.
\end{align}
The assertion is immediate from \eqref{eqnasty0}--\eqref{eqnasty2}.

Thus, we are left to prove~\eqref{eqnasty2}.
Since on the event $\cM$ we have $\vM_i^-=\Erw[\vM_i]+O(\sqrt n\ln n)=\Omega(n)$ for any $i\leq L$ in the support of $\vk$, the local limit theorem for the Poisson distribution yields
\begin{align}
\pr&\brk{\vec M^-=M^-,\forall i\leq n:\vec d_i=d_i\mid \GAMMA=\gamma,\DELTA=\Delta}
	=\pr\brk{\vec M=M^-+\gamma,\forall i\leq n:\vec d_i=d_i\mid \GAMMA=\gamma,\DELTA=\Delta}\nonumber\\
	&=\frac{\pr\brk{\GAMMA=\gamma,\DELTA=\Delta\mid \vec M=M^-+\gamma,\forall i\leq n:\vec d_i=d_i			}}{\pr\brk{\GAMMA=\gamma,\DELTA=\Delta}}\cdot\pr\brk{\vec M=M^-+\gamma}\cdot\prod_{i=1}^n\pr\brk{\vec d_i=d_i}\nonumber\\
	&\sim \frac{\pr\brk{\GAMMA=\gamma\mid \vec M=M^-+\gamma,\forall i\leq n:\vec d_i=d_i,\DELTA=\Delta			}}{\pr\brk{\GAMMA=\gamma}}\cdot\pr\brk{\vec M=M^-}\cdot\prod_{i=1}^n\pr\brk{\vec d_i=d_i}.\label{eqnasty3}
\end{align}
Finally, we notice that the argument from which we derived~\eqref{eqnasty-1} implies that 
\begin{align*}
\pr\brk{\GAMMA=\gamma\mid \vec M=M^-+\gamma,\forall i\leq n:\vec d_i=d_i,\DELTA=\Delta			}\sim\pr\brk{\GAMMA=\gamma}.
\end{align*}
Hence, \eqref{eqnasty2} follows from \eqref{eqnasty1} and \eqref{eqnasty3}.
\end{proof}

\subsection{Proof of \Lem~\ref{Lemma_A'''}}\label{Sec_A'''}

The proof comprises several steps, each relatively simple individually.
Let $$\cE=\cbc{\mu_{\vA',\cC}\mbox{ is $(\exp(-1/\eps^4),\lceil\exp(1/\eps^4)\rceil)$-symmetric}}.$$

\begin{claim}\label{Claim_A'''2}
For sufficiently large $\Theta=\Theta(\eps)$ we have $\pr\brk{\cE}\geq1-\exp(-1/\eps^4).$
\end{claim}
\begin{proof}
This is an immediate consequence of \Lem~\ref{Lemma_theta}.
\end{proof}

\noindent
Let $\cE'=\cbc{|\cC|\geq\eps dn/2\wedge \max_{i\leq n}\vd_i\leq n^{1/2}}$.

\begin{claim}\label{Claim_A'''2a}
We have $\pr\brk{\cE'}=1-o(1).$
\end{claim}
\begin{proof}
This follows from the choice of the parameters in~\eqref{eqPoissons} and 
\Lem~\ref{Lemma_sums}.
\end{proof}

\noindent
Let
	$$X=\sum_{i\geq3}\DELTA_i,\qquad Y=\sum_{i\geq3}i\DELTA_i,\qquad Y'=\sum_{i\geq3}i\GAMMA_i.$$
Then the total number of new non-zero entries upon going from $\vA'$ to $\vA'''$ is bounded by $Y+Y'$.
Let $$\cE''=\cbc{X\vee Y\vee Y'\leq1/\eps}.$$

\begin{claim}\label{Claim_A'''1}
We have $\pr\brk{\cE''}=1-O_\eps(\eps)$.
\end{claim}
\begin{proof}
Since \eqref{eqPoissons} yields $\Erw[X],\Erw[Y]=O_\eps(1)$,  we obtain $\pr\brk{X>1/\eps}=O_\eps(\eps)$ and  $\pr\brk{Y>1/\eps}=O_\eps(\eps)$.
With respect to $Y'$ we observe that $\pr\brk{\vd_{n+1}>1/\eps}=O_\eps(\eps)$, because $\Erw[\vd_{n+1}]=O_\eps(1)$.
Further, 
we can bound the probability that a check of degree $i$ is adjacent to $\vd_{n+1}$ by $i\vd_{n+1}/n$, because one of the $i$ clones of the check has to be matched to one of the $\vd_{n+1}$ clones of $x_{n+1}$ and $\sum_{i=1}^n\vd_i\geq n$.
Hence,
$\Erw\brk{Y'}=\Erw\sum_{i\geq3}i\GAMMA_i\leq\Erw\sum_{i\in[\vm_\eps^+]}\vk_i^2\vd_{n+1}/n=O_\eps(1).$
Thus, the assertion follows from Markov's inequality.
\end{proof}

\noindent
Going from $\G'$ to $\G'''$ we add checks $a_{i,j}'''$, $i\geq3$, $j\in[\GAMMA_i]$ and $b_{i,j}'''$, $i\geq3$, $j\in[\vM_i^+-\vM_i^--\GAMMA_i]$.
Let 
$$\cX=\bc{\bigcup_{i\geq3}\bigcup_{j=1}^{\GAMMA_i}\partial a_{i,j}'''\setminus\{x_{n+1}\}}\cup 
		\bc{\bigcup_{i\geq3}\bigcup_{j\in[\vM_i^+-\vM_i^--\GAMMA_i]}\partial b_{i,j}'''}$$
comprise all the variable nodes adjacent to the new checks, except for $x_{n+1}$.
Further, let
$$\cE'''=\cbc{|\cX|=Y+\sum_{i\geq3}(i-1)\GAMMA_i}$$
be the event that the variables of $\G'$ where the new checks attach are all distinct.

\begin{claim}\label{Claim_A'''2b}
We have $\pr\brk{\cE'''\mid\cE'\cap\cE''}=1-o(1).$
\end{claim}
\begin{proof}
Given $\cE'$ there are $\Omega(n)$ cavities in total, while the maximum number belonging to any one variable is $O(\sqrt n)$.
Further, given $\cE''$ we merely pick a bounded number $Y+Y'=O_\eps(1/\eps)$ of these cavities randomly as neighbours of the new checks.
Thus, the probability of hitting the same variable twice is $o(1)$.
\end{proof}

\begin{claim}\label{Claim_A'''3}
We have
$\Erw\brk{\abs{\nul(\vA''')-\nul(\vA')}(1-\vecone\cE\cap\cE'\cap\cE''\cap\cE''')}=o_\eps(1)$.
\end{claim}
\begin{proof}
Clearly $\abs{\nul(\vA''')-\nul(\vA')}\leq X+\vd_{n+1} +1$ because going from $\vA'$ to $\vA'''$ we add one column and at most $X+\vd_{n+1}$ new rows.
Consequently, as $\Erw[X],\Erw[\vd_{n+1}]=O_\eps(1)$ and $X,\vd_{n+1}$ are independent,
\begin{align}
\Erw\brk{\abs{\nul(\vA''')-\nul(\vA')}(1-\vecone\cE'')}&\leq
		\Erw\brk{(X+\vd_{n+1}+1)\vecone\{X>1/\eps\}}+\Erw\brk{(\vd_{n+1}+1/\eps+1)\vecone\{\vd_{n+1}>\eps^{-1}\}}=
%	&\leq \Erw\brk{X\vecone\{X>1/\eps\}}+\Erw\brk{(\vd_{n+1}+1)\vecone\{X>1/\eps\}}
%			+o_\eps(1)=
o_\eps(1),
	\label{eqClaim_A'''3_1}
\end{align}
whence the assertion is immediate.
Furthermore, Claims~\ref{Claim_A'''2}, \ref{Claim_A'''2a} and~\ref{Claim_A'''2b} readily imply that
\begin{align}	\label{eqClaim_A'''3_2}
\Erw&\brk{\abs{\nul(\vA''')-\nul(\vA')}\vecone\cE''\setminus\cE}	
	\leq O_\eps(\eps^{-1})\exp(-1/\eps^4)=o_\eps(1);&\mbox{similarly},\\
	\label{eqClaim_A'''3_3}
\Erw&\brk{\abs{\nul(\vA''')-\nul(\vA')}\vecone\cE''\setminus\cE'},\Erw\brk{\abs{\nul(\vA''')-\nul(\vA')}\vecone\cE''\cap\cE'\setminus\cE'''}=o(1).
\end{align}
The assertion follows from \eqref{eqClaim_A'''3_1}--\eqref{eqClaim_A'''3_3}.
\end{proof}

%\begin{claim}\label{Claim_A'''3}
%We have
%$\Erw\brk{\abs{\nul(\vA''')-\nul(\vA')}(1-\vecone\cE'')}=o_\eps(1)$.
%\end{claim}
%\begin{proof}
%Clearly $\abs{\nul(\vA''')-\nul(\vA')}\leq X+\vd_{n+1} +1$ because going from $\vA'$ to $\vA'''$ we add one column and at most $X+\vd_{n+1}$ new rows.
%Consequently, as $\Erw[X],\Erw[\vd_{n+1}]=O_\eps(1)$ and $X,\vd_{n+1}$ are independent,
%\begin{align}\nonumber
%\Erw\brk{\abs{\nul(\vA''')-\nul(\vA')}(1-\vecone\cE'')}&\leq
%		\Erw\brk{(X+\vd_{n+1}+1)\vecone\{X>1/\eps\}}+\Erw\brk{(\vd_{n+1}+1/\eps+1)\vecone\{\vd_{n+1}>\eps^{-1}\}}=
%o_\eps(1),
%\end{align}
%as claimed.
%\end{proof}

We obtain $\G'''$ by adding checks $a_{i,j}'''$ adjacent to $x_{n+1}$ and $b_{i,j}'''$ not adjacent to $x_{n+1}$. Recall that
$\mu_{\vA',\cX}\in\cP(\FF_q^{\cX})$ denotes the joint distribution of the variables of $\G'$ where the new check attach.
Also remember that $\vec\alpha$ signifies the fraction of frozen cavities.
Depending on the value of $\ALPHA$, we consider three cases separately.
Let $\Sigma''\supset\Sigma'$ be the $\sigma$-algebra generated by  $\G'$, $\A'$, $\vM_-$, $(\vd_i)_{i\in[n]}$, $\GAMMA,\vM$ and $\DELTA$.

\begin{claim}\label{Claim_A'''4}
On the event $\{\vec\alpha>1-\exp(-1/\eps^2)\}\cap\cE\cap\cE'\cap\cE''$ we have
$\Erw\brk{(\nul(\vA''')-\nul(\vA'))\vecone\cE'''\mid\Sigma''}=o_\eps(1).$
\end{claim}
\begin{proof}
Since on $\cE''$ we have $|\cX|\leq O_\eps(1/\eps)$,  with probability $1-\exp(-\Omega_\eps(1/\eps^2))$ all of the variables in $\cX$ are frozen in $\vA'$.
In this case the new checks $b_{i,j}'''$ are trivially satisfied by any vector in the kernel of $\vA'$, because the adjacent variables always take the value zero.
For the same reason in all vectors in the kernel of $\vA'''$ variable $x_{n+1}$ must take the value zero.
Conversely, any vector in the kernel of $\vA'$ extends to a vector satisfying the new checks  by setting $x_{n+1}$ to $0$.
Thus, with probability $1-\exp(-\Omega_\eps(1/\eps^2))$ we have $\nul(\vA')=\nul(\vA''')$, while on $\cE''$ the difference 
$|\nul(\vA')-\nul(\vA''')|$ is bounded by $X+Y'=O_\eps(1/\eps)$ deterministically.
The claim follows.
\end{proof}

\begin{claim}\label{Claim_A'''5}
On $\{\vec\alpha<\exp(-1/\eps^2)\}\cap\cE\cap\cE'\cap\cE''$ we have
$\Erw\brk{(\nul(\vA''')-\nul(\vA'))\vecone\cE'''\mid\Sigma''}=o_\eps(1)+1-\sum_{i\geq3}\vM_i^+-\vM_i^-.$
\end{claim}
\begin{proof}
Since on $\cE''$ we have $|\cX|\leq O_\eps(1/\eps)$, with probability $1-\exp(-\Omega_\eps(1/\eps^2))$ none of the variables in $\cX$ is frozen.
Consequently, on $\cE'$ with probability at least $1-\exp(-\Omega_\eps(1/\eps^4))$ the joint distribution $\mu_{\vA',\cX}$ is within $\exp(-\Omega_\eps(1/\eps^4))$ of the uniform distribution on $\FF_q^{\cX}$ in total variation.
Therefore, the claim follows from Fact~\ref{Fact_nullity}.
\end{proof}

\begin{claim}\label{Claim_A'''6}
On the event $\{\exp(-1/\eps^2)\leq\vec\alpha\leq1-\exp(-1/\eps^2)\}\cap\cE\cap\cE'\cap\cE''$ we have
\begin{align*}
\Erw\brk{\bc{\nul(\vA''')-\nul(\vA')}\vecone\cE'''\mid\Sigma''}&=
	o_\eps(1)+\prod_{i\geq3}(1-\ALPHA^{i-1})^{\GAMMA_i}
		-\sum_{i\geq3}(1-\vec\alpha^{i-1})\GAMMA_i
	-\sum_{i\geq3}(1-\vec\alpha^{i})(\vM_i^+-\vM_i^--\GAMMA_i).
\end{align*}
\end{claim}
\begin{proof}
Fact~\ref{Fact_nullity} yields
\begin{align}\label{eqClaim_A'''6_1}
\nul(\vA''')-\nul(\vA')&=
	\log_q\sum_{\sigma\in\FF_q^{\cX\cup\{x_{n+1}\}}}\mu_{\vA',\cX}(\sigma_{\cX})
		\prod_{i\geq3}\bc{
		\prod_{j=1}^{\GAMMA_i}\vecone\cbc{\sigma\models_{\vA'''}a_{i,j}'''}
	\prod_{j\in[\vM_i^+-\vM_i^--\GAMMA_i]}\vecone\cbc{\sigma\models_{\vA'''}b_{i,j}'''}}.
\end{align}
To evaluate the mean of this expression we need to get a grip on the distribution $\mu_{\vA',\cX}$.
Because $|\cC|=\Omega(n)$ and $|\cX|\leq O_\eps(1/\eps)$ on $\cE'\cap\cE''$, 
the number $F$ of frozen variables in the random set $\cX$ has distribution $\Bin(|\cX|,\vec\alpha)$, up to an error of $o(1)$ in total variation (as cavities are drawn without replacement).
Further, given any outcome of $F$, on $\cE$ the joint distribution $\mu_{\vA',\cX}$ is within $O_\eps(\exp(-1/\eps^4))$ of a product measure on $\FF_q^{\cX}$ with probability $1-\exp(-\Omega_\eps(1/\eps^4))$.
Specifically, the factors of this product measure corresponding to frozen variables are point measures on zero, while the others are uniform on $\FF_q$.
Hence, define a random product measure $\MU$ on $\FF_q^{\cX}$ whose every factor is, independently of all others, the atom on zero with probability $\vec\alpha$ and the uniform distribution on $\FF_q$ with probability $1-\vec\alpha$.
Then \eqref{eqClaim_A'''6_1} shows that on $\cE\cap\cE'\cap\cE''$, % \cap\cE'''$,
\begin{align}\label{eqClaim_A'''6_2}
\Erw&\brk{\bc{\nul(\vA''')-\nul(\vA')}\vecone\cE'''\mid\Sigma''}
		=R+o_\eps(1),\qquad\mbox{where}\\
	R&=\Erw\brk{\vecone\cE'''\cdot\log_q\sum_{\sigma\in\FF_q^{\cX\cup\{x_{n+1}\}}}\MU(\sigma_{\cX})
		\prod_{i\geq3}\bc{
		\prod_{j=1}^{\GAMMA_i}\vecone\cbc{\sigma\models_{\vA'''}a_{i,j}'''}
	\prod_{j\in[\vM_i^+-\vM_i^--\GAMMA_i]}\vecone\cbc{\sigma\models_{\vA'''}b_{i,j}'''}}\,\bigg|\,
		\Sigma''}.\nonumber
\end{align}
Further, because the marginals  $(\MU_x)_{x\in\cX}$ are mutually independent, $R$ simplifies:
\begin{align}\label{eqClaim_A'''6_3}
	R=S+T,\qquad\mbox{where}\qquad
	S&=\Erw\brk{\vecone\cE'''\cdot\log_q\sum_{\tau\in\FF_q}\prod_{i\geq3}\prod_{j=1}^{\GAMMA_i}\sum_{\sigma\in\FF_q^{\partial a_{i,j}'''}}
				\vecone\{\sigma_{x_{n+1}}=\tau,\,\sigma\models_{\vA'''}a_{i,j}'''\}\MU(\sigma_{\partial a_{i,j}'''\setminus\{x_{n+1}\}})\mid\Sigma''},\\%\nonumber\\
	T&=\sum_{i\geq3}\sum_{j\in[\vM_i^+-\vM_i^--\GAMMA_i]}
		\Erw\brk{\vecone\cE'''\cdot\log_q\sum_{\sigma\in\FF_q^{\partial b_{i,j}'''}}\MU(\sigma)\vecone\cbc{\sigma\models_{\vA'''}a_{i,j}'''}\mid\Sigma''}
	.\nonumber
\end{align}

We evaluate $S,T$ on $\cE'''$.
Let $\cZ$ be the set of checks $a_{i,j}'''$ with
$\MU_y=\delta_0$ for all $y\in\partial a_{i,j}'''\setminus\{x_{n+1}\}$.
If $a_{i,j}'''\not\in\cZ$, then
$$\sum_{\sigma\in\FF_q^{\partial a_{i,j}'''}}
				\vecone\{\sigma_{x_{n+1}}=\tau,\,\sigma\models_{\vA'''}a_{i,j}'''\}\MU(\sigma_{\partial a_{i,j}'''\setminus\{x_{n+1}\}})=1/q
\qquad\mbox{for every }\tau\in\FF_q.$$
Hence, if $\cZ=\emptyset$, then
\begin{align}\label{eqClaim_A'''6_4}
	\log_q\sum_{\tau\in\FF_q}\prod_{i\geq3}\prod_{j=1}^{\GAMMA_i}\sum_{\sigma\in\FF_q^{\partial a_{i,j}'''}}
				\vecone\{\sigma_{x_{n+1}}=\tau,\,\sigma\models_{\vA'''}a_{i,j}'''\}\MU(\sigma_{\partial a_{i,j}'''\setminus\{x_{n+1}\}})
	&={1-\sum_{i\geq3}\GAMMA_i}.
\end{align}
By contrast, if $\cZ\neq\emptyset$, then only the summand $\tau=0$ contributes.
Indeed, a check $a_{i,j}\in\cZ$ can be satisfied iff $x_{n+1}$ is set to zero.
Hence,
\begin{align}\label{eqClaim_A'''6_5}
	\log_q\sum_{\tau\in\FF_q}\prod_{i\geq3}\prod_{j=1}^{\GAMMA_i}\sum_{\sigma\in\FF_q^{\partial a_{i,j}'''}}
				\vecone\{\sigma_{x_{n+1}}=\tau,\,\sigma\models_{\vA'''}a_{i,j}'''\}\MU(\sigma_{\partial a_{i,j}'''\setminus\{x_{n+1}\}})
	&={|\cZ|-\sum_{i\geq3}\GAMMA_i}.
\end{align}
The construction of $\MU$ ensures that $a_{i,j}\in\cZ$ with probability $\ALPHA^{i-1}$ independently.
Hence, \eqref{eqClaim_A'''6_4}, \eqref{eqClaim_A'''6_5} and Claim~\ref{Claim_A'''2b} yield
\begin{align}\nonumber
S&=\bc{1-\sum_{i\geq3}\GAMMA_i}\prod_{i\geq3}\bc{1-\ALPHA^{i-1}}^{\GAMMA_i}
			+\sum_{i\geq3}\GAMMA_i(\ALPHA^i-1)+\bc{\sum_{i\geq3}\GAMMA_i}\prod_{i\geq3}\bc{1-\ALPHA^{i-1}}^{\GAMMA_i}+o_\eps(1)\\&=\prod_{i\geq3}(1-\ALPHA^{i-1})^{\GAMMA_i}+\sum_{i\geq3}\GAMMA_i(\ALPHA^{i-1}-1)+o_\eps(1).
	\label{eqClaim_A'''6_6}
\end{align}

Moving on to $T$, we notice that the term inside the logarithm is one if $\MU_x=\delta_0$ for all $x\in\partial b_{i,j}'''$.
Otherwise the expression is equal to $1/q$.
Hence, by Claim~\ref{Claim_A'''2b} and the construction of $\MU$,
\begin{align}\label{eqClaim_A'''6_7}
T&=o_\eps(1)+\sum_{i\geq3}\sum_{j\in[\vM_i^+-\vM_i^--\GAMMA_i]}\ALPHA^i-1=\sum_{i\geq3}
	\bc{\vM_i^+-\vM_i^--\GAMMA_i}(\ALPHA^i-1)+o_\eps(1).
\end{align}
Finally, the assertion follows from \eqref{eqClaim_A'''6_2}, \eqref{eqClaim_A'''6_3}, \eqref{eqClaim_A'''6_6} and \eqref{eqClaim_A'''6_7}.
\end{proof}

\begin{proof}[Proof of \Lem~\ref{Lemma_A'''}]
Combining Claims~\ref{Claim_A'''2b}--\ref{Claim_A'''6}, we see that 
\begin{align*}%\label{eqLemma_A'''_1}
\Erw\abs{\Erw\brk{\nul(\vA''')-\nul(\vA')\mid\Sigma''}-
	\bc{\prod_{i\geq3}(1-\ALPHA^{i-1})^{\GAMMA_i}
		-\sum_{i\geq3}(1-\vec\alpha^{i-1})\GAMMA_i
	-\sum_{i\geq3}(1-\vec\alpha^{i})(\vM_i^+-\vM_i^--\GAMMA_i)}\vecone\cE''}&=o_\eps(1).
\end{align*}
Since on $\cE''$ all degrees $i$ with $\vM_i^+-\vM_i^--\GAMMA_i>0$ are bounded and Chebyshev's inequality shows that 
$\vM_i\sim\Erw[\vM_i]=\Omega(n)$ for any fixed $i$ \whp,
\eqref{eqminus} yields $\vM_i^-=\vM_i-\GAMMA_i$ \whp\
Hence,
\begin{align}\label{eqLemma_A'''_1}
\Erw\abs{\Erw\brk{\nul(\vA''')-\nul(\vA')\mid\Sigma''}-
	\bc{\prod_{i\geq3}(1-\ALPHA^{i-1})^{\GAMMA_i}
		-\sum_{i\geq3}(1-\vec\alpha^{i-1})\GAMMA_i
	-\sum_{i\geq3}(1-\vec\alpha^{i})\DELTA_i}\vecone\cE''}&=o_\eps(1).
\end{align}
Further, since $\sum_{i\geq3}\GAMMA_i\leq\vd_{n+1}$ and $\Erw[\vd_{n+1}]=O_\eps(1)$, we obtain
\begin{align}
\Erw&\brk{
	\bc{\prod_{i\geq3}(1-\ALPHA^{i-1})^{\GAMMA_i}
		-\sum_{i\geq3}(1-\vec\alpha^{i-1})\GAMMA_i}\vecone\cE''}
=\Erw\brk{
	\bc{\prod_{i\geq3}(1-\ALPHA^{i-1})^{\GAMMA_i}
		-\sum_{i\geq3}(1-\vec\alpha^{i-1})\GAMMA_i}\vecone\cE''\cap\cbc{\sum_{i\geq3}\GAMMA_i\leq\eps^{-1/4}}}+o_\eps(1)\nonumber\\
&=\Erw\brk{
	\bc{\prod_{i\geq3}(1-\ALPHA^{i-1})^{\GAMMA_i}
		-\sum_{i\geq3}(1-\vec\alpha^{i-1})\GAMMA_i}\vecone\cbc{\sum_{i\geq3}\GAMMA_i\leq\eps^{-1/4}}}+o_\eps(1)
\qquad\mbox{[by Claim~\ref{Claim_A'''1}]}\nonumber\\
&=\Erw\brk{
	\bc{\prod_{i\geq3}(1-\ALPHA^{i-1})^{\hat\GAMMA_i}
		-\sum_{i\geq3}(1-\vec\alpha^{i-1})\hat\GAMMA_i}\vecone\cbc{\sum_{i\geq3}\hat\GAMMA_i\leq\eps^{-1/4}}}+o_\eps(1)
\qquad\mbox{[by \Lem~\ref{Cor_gamma}]}\nonumber\\
&=o_\eps(1)+\Erw\brk{(1-\ALPHA^{\hat\vk-1})^{\vd}-\vd-\vd\ALPHA^{\hat\vk-1}}
		=o_\eps(1)+\Erw\brk{D(1-K'(\ALPHA)/k)-d-\frac dk K'(\ALPHA)}
	\qquad\mbox{[by the def.\ of $\hat\GAMMA$]}.\label{eqAddingVar100}
\end{align}
Similarly, Claim~\ref{Claim_A'''1}, \Lem~\ref{Cor_gamma} and the construction of $\hat\DELTA$ yield
\begin{align}
\Erw\brk{\bc{\sum_{i\geq3}(1-\vec\alpha^{i})\DELTA_i}\vecone\cE''}&=
\Erw\brk{\bc{\sum_{i\geq3}(1-\vec\alpha^{i})\DELTA_i}\vecone\cbc{\sum_{i\geq3}\DELTA_i\leq\eps^{-1/3}}}+o_\eps(1)
=\Erw\brk{\sum_{i\geq3}(1-\vec\alpha^{i})\hat\DELTA_i}+o_\eps(1)\nonumber\\
&=o_\eps(1)+(1-\eps)\frac dk\sum_{i\geq3}\pr\brk{\vk=i}\Erw[1-\vec\alpha^{i}]
	=o_\eps(1)+\frac dk-\frac dk\Erw[K(\ALPHA)].\label{eqAddingVar101}
\end{align}
Finally, the assertion follows from \eqref{eqLemma_A'''_1}, \eqref{eqAddingVar100} and \eqref{eqAddingVar101}.
\end{proof}

\subsection{Proof of \Lem~\ref{Lemma_A''}}\label{Sec_A''}
The argument resembles the one from the proof of \Lem~\ref{Lemma_A'''} but the details are considerably more straightforward as we merely add checks.
As before we consider the events
\begin{align*}
\cE&=\cbc{\mu_{\vA',\cC}\mbox{ is $(\exp(-1/\eps^4),\lceil\exp(1/\eps^4)\rceil)$-symmetric}},&
\cE'&=\cbc{|\cC|\geq\eps dn/2\wedge \max_{i\leq n}\vd_i\leq n^{1/2}}.
\end{align*}
Moreover, recalling that the total number of new non-zero entries when going from $\vA'$ to $\vA''$ is bounded by $\vd_{n+1}$,
we introduce $\cE''=\cbc{\vd_{n+1}\leq1/\eps}.$
Since $\Erw[\vd_{n+1}^2]=O_\eps(1)$, we have
\begin{align}\label{eqcEA''}
\pr\brk{\cE''}&=1-O_\eps(\eps^2).
\end{align}	
Further, let 
$$\cX=\bigcup_{i\geq3}\bigcup_{j\in[\vM_i-\vM_i^-]}\partial_{\G''}a_{i,j}''$$ contain the variable nodes adjacent to the new checks
added in the construction of $\A''$ and let $\cE'''$ be the event that none of the variable nodes in $\cX$ is connected with the set of new checks by more than one edge.
Then
\begin{align}\label{eqcE''A''}
\pr\brk{\cE'''\mid\cE'\cap\cE''}&=1-o(1).
\end{align}

\begin{claim}\label{Claim_A''_1}
We have $\Erw\brk{\abs{\nul(\vA'')-\nul(\vA')}(1-\vecone\cE\cap\cE'\cap\cE''\cap\cE''')}=o_\eps(1)$.
\end{claim}
\begin{proof}
We have $\abs{\nul(\vA'')-\nul(\vA')}\leq \vd_{n+1}$ as we add at most $\vd_{n+1}$ rows.
Because $\Erw[\vd_{n+1}]=O_\eps(1)$, we obtain
\begin{align}\label{eqClaim_A''_1_1}
\Erw\brk{\abs{\nul(\vA'')-\nul(\vA')}(1-\vecone\cE'')}&\leq
		\Erw\brk{\vd_{n+1}\vecone\{\vd_{n+1}>1/\eps\}}=o_\eps(1).
\end{align}
Moreover, 
Claims~\ref{Claim_A'''2} and~\ref{Claim_A'''2a} and~\eqref{eqcE''A''} show that
\begin{align}\label{eqClaim_A''_1_2}
\Erw\brk{\abs{\nul(\vA'')-\nul(\vA')}\vecone\cE''\setminus\cE},
\Erw\brk{\abs{\nul(\vA'')-\nul(\vA')}\vecone\cE''\setminus\cE'},
\Erw\brk{\abs{\nul(\vA'')-\nul(\vA')}\vecone\cE''\setminus\cE'''}=o_\eps(1).
\end{align}
The assertion follows from \eqref{eqClaim_A''_1_1} and\eqref{eqClaim_A''_1_2}.
\end{proof}

%\begin{claim}\label{Claim_A''_1}
%We have $\Erw\brk{\abs{\nul(\vA'')-\nul(\vA')}(1-\vecone\cE'')}=o_\eps(1)$.
%\end{claim}
%\begin{proof}
%We have $\abs{\nul(\vA'')-\nul(\vA')}\leq \vd_{n+1}$ as we add at most $\vd_{n+1}$ rows.
%Because $\Erw[\vd_{n+1}]=O_\eps(1)$, we obtain
%\begin{align*}%\label{eqClaim_A''_1_1}
%\Erw\brk{\abs{\nul(\vA'')-\nul(\vA')}(1-\vecone\cE'')}&\leq
%		\Erw\brk{\vd_{n+1}\vecone\{\vd_{n+1}>1/\eps\}}=o_\eps(1),
%\end{align*}
%as desired.
%\end{proof}

The matrix $\vA''$ results from $\vA'$ by adding checks $a_{i,j}''$, $i\geq3$, $j\in[\vM_i-\vM_i^-]$ that are connected to random cavities of $\A'$.
We recall that $\vec\alpha$ signifies the fraction of frozen cavities of $\A'$.
\aco{Let $\Sigma''_*\supset\Sigma'$ be the $\sigma$-algebra generated by $\GAMMA,\vM$ and $\DELTA$.
Moreover, let $\Sigma''\supset\Sigma''_*$ be the $\sigma$-algebra generated by $\G''$ and $\A''$.}
Once more we consider three scenarios separately, depending on the value of $\ALPHA$.

\begin{claim}\label{Claim_A''_2}
On the event $\{\ALPHA>1-\exp(-1/\eps^2)\}\cap\cE\cap\cE'\cap\cE''$ we have $\Erw\brk{(\nul(\vA'')-\nul(\vA'))\vecone\cE'''\mid\Sigma''}=o_\eps(1).$
\end{claim}
\begin{proof}
This follows from the argument that we used in the proof of Claim~\ref{Claim_A'''4}.
\end{proof}

\begin{claim}\label{Claim_A''_3}
On the event $\{\ALPHA<\exp(-1/\eps^2)\}\cap\cE\cap\cE'\cap\cE''$ we have $$\Erw\brk{(\nul(\vA'')-\nul(\vA'))\vecone\cE'''\mid\Sigma''}=
		o_\eps(1)-\sum_{i\geq3}\vM_i-\vM_i^-.$$
\end{claim}
\begin{proof}
This follows from the argument that we used in the proof of Claim~\ref{Claim_A'''5}.
\end{proof}

\begin{claim}\label{Claim_A''_6}
On the event $\{\exp(-1/\eps^2)\leq\ALPHA\leq1-\exp(-1/\eps^2)\}\cap\cE\cap\cE'\cap\cE''$ we have
\begin{align*}
\Erw\brk{(\nul(\vA'')-\nul(\vA'))\vecone\cE'''\mid\Sigma''}&=
	o_\eps(1)-\sum_{i\geq3}(1-\vec\alpha^{i})(\vM_i-\vM_i^-).
\end{align*}
\end{claim}
\begin{proof}
Fact~\ref{Fact_nullity} yields
\begin{align}\label{eqClaim_A''6_1}
\nul(\vA'')-\nul(\vA')&=
	\log_q\sum_{\sigma\in\FF_q^{\cX}}\mu_{\vA',\cX}(\sigma)
		\prod_{i\geq3}
		\prod_{j\in[\vM_i-\vM_i^-]}\vecone\cbc{\sigma\models_{\vA''}a_{i,j}''}.
\end{align}
Similarly as in the proof of Claim~\ref{Claim_A'''6},
we evaluate this expression by substituting a simpler random measure for $\mu_{\vA',\cX}$.
Indeed, let $\MU\in\cP(\FF_q^{\cX})$ be a random product measure whose every factor is $\delta_0$ with probability $\vec\alpha$ and the uniform distribution on $\FF_q$ with probability $1-\vec\alpha$, independently of all others.
Then on $\cE\cap\cE'\cap\cE''$ we have
\begin{align}
\Erw\brk{\bc{\nul(\vA'')-\nul(\vA')}\vecone\cE'''\mid\Sigma''}&=o_\eps(1)+
	\Erw\brk{\vecone\cE'''\cdot
	\log_q\sum_{\sigma\in\FF_q^{\cX}}\MU(\sigma)
		\prod_{i\geq3}
		\prod_{j\in[\vM_i-\vM_i^-]}\vecone\cbc{\sigma\models_{\vA''}a_{i,j}''}
		\,\big|\,
		\Sigma''}.\label{eqClaim_A''6_2}
\end{align}
Indeed, because the marginals  of $\MU$ are mutually independent, we can simplify
\begin{align}\label{eqClaim_A''6_3}
	\Erw\brk{\vecone\cE'''\cdot
	\log_q\sum_{\sigma\in\FF_q^{\cX}}\MU(\sigma)
		\prod_{i\geq3}
		\prod_{j\in[\vM_i-\vM_i^-]}\vecone\cbc{\sigma\models_{\vA''}a_{i,j}''}
		\,\big|\,
		\Sigma''}
		&=\sum_{\substack{i\geq3\\\j\in[\vM_i-\vM_i^-]}}
			\Erw\brk{\vecone\cE'''\cdot\log_q\sum_{\sigma\in\FF_q^{\partial a_{i,j}''}}\MU(\sigma)\vecone\cbc{\sigma\models_{\vA''}a_{i,j}''}\mid\Sigma''}.
\end{align}
This final expression is evaluated easily.
Indeed, if $\MU_x=\delta_0$ for all $x\in\partial a_{i,j}''$, then the term inside the logarithm is just one.
Otherwise, if $\MU_x$ is uniform on $\FF_q$ for at least one $x\in\partial a_{i,j}''$, the term equals $1/q$.
Hence, by \eqref{eqcE''A''},
\begin{align}\label{eqClaim_A''6_4}
\Erw\brk{\vecone\cE'''\cdot\log_q\sum_{\sigma\in\FF_q^{\partial a_{i,j}''}}\MU(\sigma)\vecone\cbc{\sigma\models_{\vA''}a_{i,j}''}\mid\Sigma''}
	&=\ALPHA^i-1+o_\eps(1).
\end{align}
Combining \eqref{eqClaim_A''6_2}--\eqref{eqClaim_A''6_4} completes the proof.
\end{proof}

\begin{proof}[Proof of \Lem~\ref{Lemma_A''}]
Combining Claims~\ref{Claim_A''_1}--\ref{Claim_A''_6}, we obtain
\begin{align*}
\Erw\abs{\Erw[\nul(\vA'')-\nul(\vA')\mid\Sigma'']+\bc{\sum_{i\geq3}(1-\vec\alpha^{i})(\vM_i-\vM_i^-)}\vecone\cE''}	&=o_\eps(1).
\end{align*}
Since \whp\ all degrees $i$ with $\vM_i^+-\vM_i^-$ are bounded, Chebyshev's inequality reveals that
$\vM_i-\vM_i^-=\GAMMA_i$ for all $i$ \whp\
Hence, 
\begin{align}\label{eqAddingVar011}
\Erw\abs{\Erw[\nul(\vA'')-\nul(\vA')\mid\Sigma'']+\bc{\sum_{i\geq3}(1-\vec\alpha^{i})\GAMMA_i}\vecone\cE''}	&=o_\eps(1).
\end{align}
Further, because $\sum_{i\geq3}\GAMMA_i\leq\vd_{n+1}$ and $\Erw[\vd_{n+1}]=O_\eps(1)$,
\begin{align}
\Erw\brk{\bc{\sum_{i\geq3}(1-\vec\alpha^{i})\GAMMA_i}\vecone\cE''}&=
\Erw\brk{\bc{\sum_{i\geq3}(1-\vec\alpha^{i})\GAMMA_i}\vecone\cbc{\sum_{i\geq3}\GAMMA_i\leq\eps^{-1/4}}}+o_\eps(1)\nonumber&&\mbox{[by \eqref{eqcEA''}]}\\
&=
\Erw\brk{\bc{\sum_{i\geq3}(1-\vec\alpha^{i})\hat\GAMMA_i}\vecone\cbc{\sum_{i\geq3}\hat\GAMMA_i\leq\eps^{-1/4}}}+o_\eps(1)
&&\mbox{[by \Lem~\ref{Cor_gamma}]}\nonumber\\
&=o_\eps(1)+d\Erw[1-\ALPHA^{\hat\vk}]=o_\eps(1)-  d\Erw[\vec\alpha K'(\vec\alpha)]/k+d.
\label{eqAddingVar012}
\end{align}
The assertion follows from \eqref{eqAddingVar011} and~\eqref{eqAddingVar012}.
\end{proof}

\subsection{Proof of \Lem~\ref{Lemma_valid}}\label{Sec_Lemma_valid}

The choice of the random variables in (\ref{eqPoissons}) and \Lem~\ref{Lemma_sums} ensure that the event $\cE=\{\sum_{i\geq3}i\vM_i\leq dn/k\}$ has probability $1-o(1/n)$.
Further, given $\cE$ the random variables $\nul(\vA'')$ and $\nul(\vA_{n,\vM})$ are identically distributed by the principle of deferred decisions.
Because the nullity of either matrix is bounded by $n$ deterministically, we thus obtain the first assertion.

Similarly, to prove the second assertion we may condition on the event 
		$$\cE^+=\cbc{\frac{dn}{2k}\leq\sum_{i\geq3}i\vM_i^+\leq \sum_{i=1}^n\vd_i,\forall i\geq  n/\ln^9n:\vM_i^+=0},$$
which occurs with probability $1-o(1/n)$ by \Lem~\ref{Lemma_sums}.
Further, since $\Erw[\vd^2],\Erw[\vk^2]=O(1)$, the event
	$$\cW=\cbc{\vd_{n+1}\leq\ln n
		,\ \sum_{i\geq3}i(\DELTA_i+\GAMMA_i)<\ln^4 n}$$
occurs \whp\
Additionally, let $\cU$ be the event that $x_{n+1}$ does not partake in any multi-edges of $\G_{n,\vM^+}$.
Then
\begin{align}\label{eqLemma_valid1}
\pr\brk{\cU\mid\cW\cap\cE^+}=1-o(\ln^{-5}n);
\end{align}
indeed, given $\cW\cap\cE^+$ variable node $x_{n+1}$ has target degree at most $\ln n$ and
all check degrees are bounded by $n/\ln^9n$.
Hence, the probability that $x_{n+1}$ joins the same check twice is $o(\ln^{-3}n)$.
Once more by the principle of deferred decisions, given $\cE^+\cap\cU\cap\cW$ the random variables
$\nul(\vA''')$ and $\nul(\vA_{n+1,\vM^+})$ are identically distributed;
they can therefore be coupled identically.
Thus,
\begin{align}\label{eqBigBound}
\Erw\brk{\bc{\nul(\vA''')-\nul(\vA_{n+1,\vM^+})}\vecone\cU\cap\cW\cap\cE^+}&=0.
\end{align}

Moreover, we can always couple $\vA_{n+1,\vM^+}$ and $\vA'''$ such that both differ in no more than $2(\sum_{i\geq3}i(\DELTA_i+\GAMMA_i))$ entries.
To see this, we estimate the number of edges of the Tanner graph incident with the checks $a_{i,j}$, $\vM_i^-<j\leq\vM_i^+$ or the new variable $x_{n+1}$ of $\G_{n+1,\vM^+}$.
By construction, there are at most $\sum_{i\geq3}i(\DELTA_i+\GAMMA_i)$ such edges.
Similarly, there are no more than $\sum_{i\geq3}i(\DELTA_i+\GAMMA_i)$ edges incident with the new checks $a_{i,j}'''$, $b_{i,j}'''$ added to $\vA'$ to obtain $\vA'''$.
By the principle of deferred decisions we can couple the Tanner graphs of $\vA'''$ and $\vA_{n+1,\vM^+}$ such that they coincide on all the edges that join variables $x_1,\ldots,x_n$ and checks $a_{i,j}$, $j\leq\vM_i^-$, and hence the matrices themselves so that they coincide on all the corresponding matrix entries.
Consequently, on $\cE^+$ we have
\begin{align}\label{eqQuickBound}
\abs{\nul(\vA_{n+1,\vM^+})-\nul(\vA''')}&\leq2\sum_{i\geq3}i(\DELTA_i+\GAMMA_i).
\end{align}
Combining \eqref{eqLemma_valid1} and \eqref{eqQuickBound}, we obtain
\begin{align}\label{eqBigBound2}
\Erw\brk{\abs{\nul(\vA_{n+1,\vM^+})-\nul(\vA''')}\vecone\cE^+\cap\cW\setminus\cU}=o(1).
\end{align}

To complete the proof we need to deal with the event $\cE^+\setminus\cW$, which is contained in the union of the events
\begin{align*}
\cQ_1&=\cE^+\cap\cbc{\exists i>\log n:\GAMMA_i>0},&
\cQ_2&=\cE^+\cap\cbc{\vd_{n+1}>\log n}\setminus\cQ_1,&
\cQ_3&=\cE^+\cap\cbc{\sum_{i\geq3}i\DELTA_i>\ln^3n}\setminus(\cQ_1\cup\cQ_2).
\end{align*}
To bound the contribution of $\cQ_1$, consider $\vm_\eps^+=\sum_{i\geq3}\vM_i^+\disteq\Po((1-\eps)d(n+1)/k)$.
We claim that, with $(\vk_i)_{i\geq3}$ independent of everything else,
\begin{align}\label{eqQuickBound10}
\Erw\brk{\sum_{i\geq3}i\GAMMA_i\vecone\cQ_1}&\leq O(1/n)\cdot\bc{1+\Erw\brk{\sum_{i=1}^{\vm_\eps^+}\vecone\{\vk_i\geq\log n\}\vk_i^2\vec d_{n+1}}}=O(1)\cdot \pr\brk{\vk\geq\log n}+O(1/n)=O(\log^{-2}n).
\end{align}
Indeed, the last equality sign follows from the first because $\Erw[\vk^2]=O(1)$ and the first equality sign follows because $\vm_\eps^+$ is independent of $\vd_{n+1}$ and the $\vk_i$.
Further, to obtain the first inequality we consider the $\vm_\eps^+$ checks one by one.
The degree of the $i$th check is distributed as $\vk_i$.
We discard it unless $\vk_i\geq\log n$.
But if $\vk_i\geq\log n$, then the probability that $\vk_i$ is adjacent to $x_{n+1}$ is bounded by $O(\vk_i\vd_{n+1}/\sum_{h=1}^{n+1}\vd_h)$ and $\sum_{h=1}^{n+1}\vd_h>n$.
Thus, we obtain (\ref{eqQuickBound10}).
Further, we observe that (\ref{eqQuickBound10}) yields
$\pr\brk{\cQ_1}\leq\Erw\sum_{i\geq3}i\GAMMA_i\vecone\cQ_1=O(\log^{-2}n).$
Hence,  as $\Erw\sum_{i\geq3}i\DELTA_i=O(1)$ we obtain
\begin{align}\label{eqQuickBound11}
\Erw\brk{\sum_{i\geq3}i\DELTA_i\vecone\cQ_1}&\leq\pr\brk{\cQ_1}\log n
		+\Erw\brk{\sum_{i\geq3}i\DELTA_i\vecone\cbc{\sum_{i\geq3}i\DELTA_i\geq\log n}}=o(1).
\end{align}
Combining \eqref{eqQuickBound}, \eqref{eqQuickBound10} and \eqref{eqQuickBound11}, we conclude that
\begin{align}\label{eqQuickBound12}
\Erw\brk{\abs{\nul(\vA_{n+1,\vM^+})-\nul(\vA''')}\vecone\cQ_1}&=o(1).
\end{align}

Regarding $\cQ_2$, we deduce from the bound $\Erw[\vd_{n+1}^r]=O(1)$ for an $r>2$ that
\begin{align*}
\Erw\brk{\sum_{i\geq3}i\GAMMA_i\vecone\cQ_2}&\leq O(\log n)\Erw\brk{\vd_{n+1}\vecone\{\vd_{n+1}>\log n\}}=o(1).
\end{align*}
Moreover, since the $\DELTA_i$ are independent of $\vd_{n+1}$ and $\Erw\sum_{i\geq3}i\DELTA_i=O(1)$, we obtain
$\Erw\brk{\sum_{i\geq3}i\DELTA_i\vecone\cQ_2}=o(1)$.
Hence, \eqref{eqQuickBound} yields
\begin{align}\label{eqQuickBound21}
\Erw\brk{\abs{\nul(\vA_{n+1,\vM^+})-\nul(\vA''')}\vecone\cQ_2}&=o(1).
\end{align}

Moving on to $\cQ_3$, we find
\begin{align*}
\pr\brk{\cQ_3}&\leq\Erw\brk{\sum_{i\geq3}i\DELTA_i}\ln^{-3}n=O(\Erw[\vk^2]\ln^{-3}n)=o(\log^{-2}n).
\end{align*}
Moreover, on $\cQ_3$ we have $\sum_{i\geq3}i\GAMMA_i\leq\log^2n$ because $\vd_{n+1}\leq\log n$ and
$\GAMMA_i=0$ for all $i\geq\log n$.
Consequently, since the $\DELTA_i$ are mutually independent and $\sum_{i\geq3}\Erw[i\DELTA_i]=O(1)$, \eqref{eqQuickBound} yields
\begin{align}\label{eqQuickBound30}
\Erw\brk{\abs{\nul(\vA_{n+1,\vM^+})-\nul(\vA''')}\vecone\cQ_3}
	&\leq o(1)+4\Erw\brk{\sum_{i\geq3}i\DELTA_i\vecone\cbc{\sum_{i\geq3}i\DELTA_i\geq\ln^3n}}
	=o(1).
\end{align}
Finally, the second assertion follows from \eqref{eqBigBound}, \eqref{eqBigBound2}, \eqref{eqQuickBound12}, \eqref{eqQuickBound21} and \eqref{eqQuickBound30}.

\section{The interpolation argument} 
\label{Sec_interpolation}

\subsection{Proof of \Lem~\ref{Lemma_interpol2}}
Each component of $\G_\eps(0)$ contains precisely one of the variable nodes $x_1,\ldots,x_n$.
In effect, $\vA_{\eps}(0)$ has a block diagonal structure, and the overall nullity is nothing but the sum of the nullities of the blocks.
It therefore suffices to study the contribution $\vec N_s$ of the block containing $x_s$ to $\cN_0$, i.e.,
\begin{align}\label{eqNs_0}
\vec N_s&=\bc{\log_q\sum_{\sigma\in\FF_q^{\{x_s\}\cup\partial^2x_s}}\vecone\{(s>\THETA\vee\sigma_{x_s}=0)
		\wedge\forall y\in\partial^2x_s\cap\cF_0:\sigma_y=0 \}
		\vecone\{\sigma\models a\}}-\abs{\partial^2x_s}+\abs{\partial^2x_s\cap\cF_0}.
\end{align}
Indeed, to see that $\cN_0=\sum_{s=1}^n\vec N_s$, we observe that 
$\sum_{s=1}^n\abs{\partial^2x_s}=\sum_{i\leq\vm'_\eps(0)}\vk_i'(\vk_i'-1)$ and that 
$\sum_{s=1}^n\abs{\partial^2x_s\cap\cF_0}=|\cF_0|$.
Moreover, the logarithmic term in~\eqref{eqNs_0} is just the nullity of the block containing $x_s$.
Dragging the last two terms from~\eqref{eqNs_0} into the logarithm, we obtain
\begin{align}\label{eqNs}
\vec N_s&=\log_q\sum_{\sigma\in\FF_q^{\{x_s\}\cup\partial^2x_s}}
		\vecone\{s>\THETA\vee\sigma_{x_s}=0\}
		\bc{\prod_{y\in\partial^2x_s\cap\cF_0}q\vecone\{\sigma_{y}=0\}}
		\bc{\prod_{a\in\partial x_s}\frac{\vecone\{\sigma\models a\}}{q^{|\partial a|-1}}},
\end{align}
and $\Erw[\cN_0]=\sum_{s=1}^n\Erw[\vec N_s]$.
Consequently, since $\THETA=O(1)$ it suffices to prove that
\begin{align}\label{eqLemma_interpol2_1}
\Erw[\vec N_s]=\begin{cases}dK'(\beta)/k+D(1-K'(\beta)/k)-d+o_\eps(1)&\mbox{ if }s>\THETA,\\	
O(1)&\mbox{ otherwise}.		
\end{cases}
\end{align}
In fact, the second case in \eqref{eqLemma_interpol2_1} simply follows from the bounds $|\vN_s|\leq\vd_s$ and
$\Erw[\vd_s]=O(1)$ for all $s$.

Hence, suppose that $s>\THETA$.
As $|\vN_s|\leq\vd_s$ and $\Erw[\vd_s^r]=O_\eps(1)$ for an $r>2$ we find $\xi>0$ such that
\begin{align}\label{eqLemma_interpol2_2}
\Erw[|\vN_s|\vecone\{\vd_s>\eps^{\xi-1/2}\}]&=o_\eps(1).
\end{align}
Moreover, let
$
\Xi=\sum_{i=1}^{\vm_\eps'(0)}\vk_i'\vecone\cbc{\vk_i'>\eps^{-8}}$, $
\vec M_j'=\sum_{i=1}^{\vm_\eps'(0)}\vecone\{\vk_i'=j\}.$
Because $\Erw[\vk^2]=O_\eps(1)$ we have
\begin{align}\label{eqLemma_interpol2_3}
\Erw\brk{\Xi}&\leq\frac{dn}k\Erw\brk{\vk\vecone\{\vk\geq\eps^{-8}\}}=nO_\eps(\eps^8),
\end{align}	
while $\vM_j'\sim(1-\eps)dn\pr\brk{\vk=j}/k$ for all $j\leq\eps^{-8}$ \whp\ by Chebyshev's inequality.
Hence, introducing the event
\begin{align*}
\cE_s&=\cbc{\vd_s\leq\eps^{\xi-1/2},\ \Xi\leq n\eps^6,\ 
\forall j\leq\eps^{-8}:\vM_j'\sim(1-\eps)dn\pr\brk{\vk=j}/k,\,
\sum_{i=1}^n\vd_i\sim dn,\,
\sum_{i\geq3}i\vM_i'\sim(1-\eps)dn},
\end{align*}
we obtain from \eqref{eqLemma_interpol2_2} and \eqref{eqLemma_interpol2_3} that
\begin{align}\label{eqLemma_interpol2_4}
\Erw[\vN_s]&=\Erw\brk{\vN_s\vecone\cE_s}+o_\eps(1).
\end{align}	

With $\vd_s^\star\leq\vec d_s$ the actual degree of $x_s$ in $\G_\eps(s)$,
let $\vec\kappa_1,\ldots,\vec\kappa_{\vd_s^\star}$ be the degrees of the checks adjacent to $x_s$.
We claim that given $\cE_s$ and $\vd_s$, 
\begin{align}\label{eqLemma_interpol2_5}
\dTV((\vec\kappa_1,\ldots,\vec\kappa_{\vd_s^\star}),(\hat\vk_1,\ldots,\hat\vk_{\vd_s}))&=o_\eps(\eps^{1/2}).
\end{align}
Indeed, on $\cE_s$ the probability that $x_s$ is adjacent to a check of degree greater than $\eps^{-8}$ is 
$O_\eps(\vd_s\Xi/\sum_{j\geq3}j\vM_j')=o_\eps(\eps)$.
Further, given $\cE_s$ we have $\sum_{j\geq3}j\vM_j'\geq(1-2\eps)dn$, and thus $\pr[\vd_s^\star<\vd_s\mid\cE_s]=o_\eps(\eps^{1/2})$.
Moreover, given $\vd_s^\star=\vd_s$, for each $i\in[\vd_s]$ the probability that the $i$th clone of $x_s$ gets matched
to a check of degree $j\leq\eps^{-8}$ is $$j\vM_j'/\sum_{h\geq3}h\vM_h'=j\pr\brk{\vk=j}/k+o(1)=\pr\brk{\hat\vk=j}+o(1).$$
These events are asymptotically independent for the different clones.
Thus, we obtain \eqref{eqLemma_interpol2_5}.

Finally, we can easily compute $\vN_s$ given the vector $(\vec\kappa_1,\ldots,\vec\kappa_\GAMMA)$.
There are two possible scenarios. % eqNs
\begin{description}
\item[Case 1: there is $a^*\in\partial x_s$ such that $\partial a^*\setminus x_s\subset\cF_0$]
then only the summand $\sigma_{x_s}=0$ contributes to the r.h.s.\ of~\eqref{eqNs}, because setting $x_s$ to zero is the only way to satisfy the check $a^*$.
Hence, we can rewrite $\vec N_s$ as
\begin{align*}
\vec N_s
		&=\sum_{a\in\partial x_s}\log_q\brk{\sum_{\sigma\in\FF_q^{\partial a}}
		\vecone\{\sigma\models a,\,\sigma_{x_s}=0,\,\forall y\in\partial a\cap\cF_0:\sigma_y=0\}q^{1-|\partial a|+|\cF_0\cap\partial a|}}.
\end{align*}
For $a\in\partial x_s$ the sum inside the logarithm evaluates to one if $\partial a\setminus\{x_s\}\subset\cF_0$
and to $1/q$ otherwise.
Hence,
\begin{align}\label{eqCase1}
\vec N_s&=\sum_{a\in\partial x_s}\bc{\vecone\{\partial a\setminus\{x_s\}\subset\cF_0\}-1}.
\end{align}
\item[Case 2: for all $a^*\in\partial x_s$ we  have $\partial a^*\setminus(\{x_s\}\cup\cF_0)\neq\emptyset$]
we rewrite $\vec N_s$ as
\begin{align*}
\vec N_s&=\log_q\sum_{\sigma\in\FF_q}\prod_{a\in\partial x_s}\sum_{\tau\in\FF_q^{\partial a}}\vecone\{\tau\models a,\,\tau_{x_s}=\sigma,\,\forall y\in\partial a\cap\cF_0:\tau_y=0\}q^{1-|\partial a|+|\partial a\cap\cF_0|}.
\end{align*}
Then for any $\sigma\in\FF_q$ the factor corresponding to any $a$ above evaluates to $1/q$.
Consequently,
\begin{align}\label{eqCase2}
\vec N_s&=1-\GAMMA.
\end{align}
\end{description}
Since any $a\in\partial^2x_s$ belongs to the set $\FF_q$ with probability $\beta$ independently, \eqref{eqCase1} and \eqref{eqCase2} imply that
\begin{align}%\nonumber
\Erw\brk{\vN_s\mid(\vec\kappa_1,\ldots,\vec\kappa_\GAMMA)}
	&=\underbrace{\GAMMA\prod_{i=1}^{\GAMMA}(1-\beta^{\vec\kappa_i-1})+\sum_{i=1}^{\GAMMA}(\beta^{\vec\kappa_i-1}-1)}_{\text{Case 1}}
		+\underbrace{(1-\GAMMA)\prod_{i=1}^{\GAMMA}(1-\beta^{\vec\kappa_i-1})}_{\text{Case 2}}
	=\prod_{i=1}^{\GAMMA}(1-\beta^{\vec\kappa_i-1})+
		\sum_{i=1}^{\GAMMA}(\beta^{\vec\kappa_i-1}-1).\label{eqLemma_interpol2_6}
\end{align}
Combining \eqref{eqLemma_interpol2_4}, \eqref{eqLemma_interpol2_5} and \eqref{eqLemma_interpol2_6} completes the proof.

\subsection{Proof of \Prop~\ref{Lemma_interpolation}}
The computation of the derivative $\frac{\partial}{\partial t}\Erw[\cN_t+\cY_t]$ is based on a coupling argument and a study of the Boltzmann distribution $\mu_{\vA_\eps(t)}$.
Specifically, we need to investigate the joint distribution of the cavities, i.e., the clones from $\bigcup_{i=1}^n\cbc{x_i}\times[\vd_i]$ that are not incident to an edge of $\vec\Gamma_\eps(t)$.
Denote this set by $\cC(t)$.
Then $\mu_{\vA_\eps(t)}$ induces a probability distribution $\mu_{\cC(t)}$ via
\begin{align*}
\mu_{\cC(t)}(\sigma)=\mu_{\vA_\eps(t),x_1,\ldots,x_n}\bc{
	\cbc{\tau\in\FF_q^{\{x_1,\ldots,x_n\}}:\forall (x_i,h)\in\cC(t):\tau_{x_i}=\sigma_{x_i,h}}
}
	\qquad(\sigma\in\FF_q^{\cC(t)}).
\end{align*}

\begin{lemma}\label{Lemma_pinint}
For any $\delta>0$ there is $\Theta=\Theta(\delta)>0$ such that $\pr\brk{\mu_{\cC(t)}\mbox{ is $\delta$-symmetric}}>1-\delta.$
\end{lemma}
\begin{proof}
The choice~\eqref{eqinterpol} of $\vm_\eps(t),\vm_\eps'(t)$ guarantees that $|\cC(t)|\geq\eps n/2$ \whp\
Moreover, since $\Erw[\vd]=O_\eps(1)$ we find $L=L(\eps,\delta)>0$ such that
the event  $\cL=\cbc{\sum_{i=1}^n\vd_i\vecone\{\vd_i>L\}<\eps\delta^2 n/16}$ has probability $\pr[\cL]\geq1-\delta/8$.
Therefore, we may condition on $\cE=\cL\cap\{|\cC(t)|\geq\eps n/2\}$.

On $\cE$ let $\vy\in\{x_1,\ldots,x_n\}$ be a random variable node chosen from the distribution
$$\pr[\vy=x_i\mid\vA_\eps(t)]=|\cC(t)\cap(\{x_i\}\times[\vd_i])|/|\cC(t)|.$$
Further, let $\vy'$ be an independent copy of $\vy$.
Since $\pr[\cE]\geq1-\delta/8-o(1)$, it suffices to prove that
\begin{align}\label{eqpinint_1}
\Erw\brk{\TV{\mu_{\vA_\eps(t),\vy,\vy'}-\mu_{\vA_\eps(t),\vy}\tensor\mu_{\vA_\eps(t),\vy'}}\,\bigg|\,\cE}&\leq\delta^2/2.
\end{align}
To verify \eqref{eqpinint_1}, let $\vx,\vx'$ be two variables chosen uniformly and independently among $x_1,\ldots,x_n$.
Then on $\cE$,
\begin{align}\label{eqpinint_2}
\pr\brk{(\vy,\vy')\in\cV\mid\vA_\eps(t)}&\leq L^2\cdot \pr\brk{(\vx,\vx')\in\cV\mid\vA_\eps(t)}+\delta^2/8
&\mbox{for any }\cV\subset\{x_1,\ldots,x_n\}^2.
\end{align}
Further, because the distribution of $\G_\eps(t)-\{p_1,\ldots,p_{\vec\theta}\}$ is invariant under permutations of $x_1,\ldots,x_n$, \Lem~\ref{Lemma_0pinning} shows that
for large enough $\Theta$,
\begin{align}\label{eqpinint_3}
\Erw\brk{\TV{\mu_{\vA_\eps(t),\vx,\vx'}-\mu_{\vA_\eps(t),\vx}\tensor\mu_{\vA_\eps(t),\vx'}}\,\bigg|\,\cE}&\leq\delta^3/(64L^2).
\end{align}
Thus, \eqref{eqpinint_1} follows from \eqref{eqpinint_2} and \eqref{eqpinint_3}.
\end{proof}

We proceed to derive \Prop~\ref{Lemma_interpolation} from \Lem~\ref{Lemma_pinint} and a coupling argument.
Let us write $\cN(m,m')$ for the conditional random variable $\cN_t$ given $\vm_\eps(t)=m,\vm_\eps(t)'=m'$.

\begin{lemma}\label{Claim_PoissonDeriv}
We have 
\begin{align*}
\frac\partial{\partial t}\Erw[\cN_t]=
	(1-\eps)\frac{dn}k\big(&
	(\Erw[\cN(\vm_\eps(t)+1,\vm_\eps'(t))]-\Erw[\cN(\vm_\eps(t),\vm_\eps'(t))])\\&-(\Erw[\cN(\vm_\eps(t),\vm_\eps'(t)+1)]-\Erw[\cN(\vm_\eps(t),\vm_\eps'(t))])\big).
\end{align*}
\end{lemma}
\begin{proof}
The parameter $t$ comes in via the Poisson variables $\vm_\eps(t),\vm_\eps'(t)$ from \eqref{eqinterpol}.
Hence, the product rule gives
\begin{align}
\frac\partial{\partial t}\Erw[\cN_t]&=
	\frac\partial{\partial t}\sum_{m,m'\geq0}\pr\brk{\vm_\eps(t)=m}\pr\brk{\vm_\eps'(t)=m'}\Erw[\cN(m,m')]
	=\sum_{m,m'\geq0}\Erw[\cN(m,m')]\frac\partial{\partial t}\pr\brk{\vm_\eps(t)=m}\pr\brk{\vm_\eps'(t)=m'}\nonumber\\
	&=\sum_{m,m'\geq0}\Erw[\cN(m,m')]\bc{\frac\partial{\partial t}\pr\brk{\vm_\eps(t)=m}}\pr\brk{\vm_\eps'(t)=m'}+\Erw[\cN(m,m')]\pr\brk{\vm_\eps(t)=m}{\frac\partial{\partial t}\pr\brk{\vm_\eps'(t)=m'}}\nonumber\\
	&=\sum_{m\geq0}\Erw[\cN(m,\vm_\eps'(t))]\frac\partial{\partial t}\pr\brk{\vm_\eps(t)=m}+
		\sum_{m'\geq0}\Erw[\cN(\vm_\eps(t),m')]\frac\partial{\partial t}\pr\brk{\vm_\eps'(t)=m'}.
		\label{eqClaim_PoissonDeriv1}
\end{align}
The Poisson derivatives work out to be
\begin{align}		\label{eqClaim_PoissonDeriv2}
\frac\partial{\partial t}\pr\brk{\vm_\eps(t)=m}&=\frac\partial{\partial t}\frac{((1-\eps)tdn/k)^m}{m!\exp((1-\eps)tdn/k)}
	=\frac{(1-\eps)dn}{k}\bc{\pr\brk{\vm_\eps(t)=m-1}-\pr\brk{\vm_\eps(t)=m}},\\
\frac\partial{\partial t}\pr\brk{\vm_\eps'(t)=m'}&=\frac\partial{\partial t}\frac{((1-\eps)(1-t)dn/k)^{m'}}{m'!\exp((1-\eps)(1-t)dn/k)}
	=\frac{(1-\eps)dn}{k}\bc{\pr\brk{\vm_\eps'(t)=m}-\pr\brk{\vm_\eps'(t)=m-1}}.
	\label{eqClaim_PoissonDeriv3}
\end{align}
Combining \eqref{eqClaim_PoissonDeriv1}--\eqref{eqClaim_PoissonDeriv3} completes the proof.
\end{proof}

We will couple
$\cN(\vm_\eps(t)+1,\vm_\eps'(t))$ and $\cN(\vm_\eps(t),\vm_\eps'(t))$ as well as $\cN(\vm_\eps(t),\vm_\eps'(t)+1)$ and $\cN(\vm_\eps(t),\vm_\eps'(t))$.
With respect to the first pair of random variables, obtain $\G'(t)$ from $\G_\eps(t)$ by adding a check $a_{\vm_\eps(t)+1}$ of target degree $\vk_{\vm_\eps(t)+1}$.
The new check is maximally matched with random cavities from $\cC(t)$.
Obtain $\vA'(t)$ by adding a new row to $\vA_\eps(t)$ corresponding to $a_{\vm_\eps(t)+1}$ and adding
a matrix entry drawn independently according to $\CHI$ for each edge of $\G'(t)$ incident with $a_{\vm_\eps(t)+1}$.

\begin{lemma}\label{Claim_Lemma_interpolation_1}
We have $\Erw[\cN(\vm_\eps(t)+1,\vm_\eps'(t))]-\Erw[\cN(\vm_\eps(t),\vm_\eps'(t))]=\Erw[\nul(\vA'(t))-\nul(\vA_\eps(t))]+o(1)$.
\end{lemma}
\begin{proof}
Due to \eqref{eqinterpol} and \Lem~\ref{Lemma_sums},  with probability $1-o(n^{-1})$ we have $|\cC(t)|\geq\vk_{\vm_\eps(t)+1}$.
Hence, the assertion follows from the principle of deferred decisions.
\end{proof}

\Lem~\ref{Lemma_Spartition} implies that each marginal $\mu_{\cC(t),c}$, $c\in\cC(t)$ is either $\delta_0$ or the uniform distribution on $\FF_q$.
Thus, let us call $c\in\cC(t)$ {\em frozen} if $\mu_{\cC(t),c}=\delta_0$.
Let $\ALPHA\in[0,1]$ be the fraction of frozen cavities in $\cC(t)$ (with the convention that $\ALPHA=0$ if $\cC(t)=\emptyset$).

\begin{lemma}\label{Claim_Lemma_interpolation_2}
We have $\Erw\abs{\Erw[\nul(\vA'(t))-\nul(\vA_\eps(t))\mid\vA_\eps(t)]-(K(\ALPHA)-1)}=o_\eps(1)$.
\end{lemma}
\begin{proof}
Pick $\zeta=\zeta(\eps)>0$ small enough and $\delta=\delta(\zeta)>0$ smaller still.
Since $|\nul(\vA'(t))-\nul(\vA_\eps(t))|\leq1$ and $\Erw[\vk^2]=O_\eps(1)$ we may condition on the event that $\vk_{\vm_\eps(t)+1}\leq\eps^{-1}$.
Similarly, because $|\nul(\vA'(t))-\nul(\vA_\eps(t))|\leq1$, \Lem~\ref{Lemma_pinint} shows that we may assume that $\mu_{\cC(t)}$ is $\delta$-symmetric for a small $\delta=\delta(\eps)>0$.
Let $\cX$ be the set of cavities adjacent to $a_{\vm_\eps(t)+1}$.
We consider three cases.
\begin{description}
\item[Case 1: $\ALPHA<\eps^2$] since $\vk_{\vm_\eps(t)+1}\leq\eps^{-1}$ the probability that $a_{\vm_\eps(t)+1}$ joins a frozen cavity is $o_\eps(1)$.
Moreover, \Lem~\ref{lem:l-wise} shows that with probability at least $1-\exp(-1/\eps)$ the joint distribution $\mu_{\cC(t),\cX}$ of the cavities that $a_{\vm_\eps(t)+1}$ joins is within $\exp(-1/\eps)$ of the uniform distribution on $\FF_q^{\vk_{\vm_\eps(t)+1}}$ in total variation, provided that $\Theta$ is chosen sufficiently large.
If so, then a random vector from the kernel of $\vA_\eps(t)$ satisfies the new check $\va_{\vm_\eps(t)+1}$ with probability $1/q+o_\eps(1)$.
Consequently, with probability $1-o_\eps(1)$ we have
\begin{align}\label{eqClaim_Lemma_interpolation_2_1}
\Erw[\nul(\vA'(t))-\nul(\vA_\eps(t))\mid\vA_\eps(t)]&=-1+o_\eps(1)=K(\ALPHA)-1+o_\eps(1).
\end{align}
\item[Case 2: $\ALPHA>1-\eps^2$]
since $\vk\leq\eps^{-1}$ the probability that $a_{\vm_\eps(t)+1}$ joins an unfrozen cavity is $o_\eps(1)$.
Hence, with probability $1-o_\eps(1)$ we have
\begin{align}\label{eqClaim_Lemma_interpolation_2_2}
\Erw[\nul(\vA'(t))-\nul(\vA_\eps(t))\mid\vA_\eps(t)]&=0=K(\ALPHA)-1+o_\eps(1).
\end{align}
\item[Case 3: $\eps^2\leq\ALPHA\leq 1-\eps^2$]
the distribution of the number $X$ of frozen cavities in $\cX$ is within $o(1)$ of a binomial distribution $\Bin(\vk_{\vm_\eps(t)+1},\ALPHA)$ in total variation.
Moreover, \Lem~\ref{lem:l-wise} shows that given any outcome of $X$ the joint distribution $\mu_{\cX}$ is within $o_\eps(1)$ of a product distribution with probability $1-o_\eps(1)$, provided that $\Theta$ is chosen big enough.
If $X=\vk_{\vm_\eps(t)+1}$, i.e., if all variables adjacent to the new check are frozen, then the new check will certainly be satisfied.
Otherwise, the probability of it being satisfied equals $1/q$.
Thus, with probability $1-o_\eps(1)$ we have
\begin{align}\label{eqClaim_Lemma_interpolation_2_3}
\Erw[\nul(\vA'(t))-\nul(\vA_\eps(t))\mid\vA_\eps(t)]&=
\Erw\brk{\ALPHA^{\vk_{\vm_\eps(t)+1}-1}\mid\vA_\eps(t)}+o_\eps(1)=K(\ALPHA)-1+o_\eps(1).
\end{align}
\end{description}
Finally, the assertion follows from \eqref{eqClaim_Lemma_interpolation_2_1}--\eqref{eqClaim_Lemma_interpolation_2_3}.
\end{proof}

To couple $\cN(\vm_\eps(t),\vm_\eps'(t)+1)$ and $\cN(\vm_\eps(t),\vm_\eps'(t))$ obtain $\G''(t)$ from $\G_\eps(t)$ by adding checks 
$$b_{\vm_\eps'(t)+1,1}, \ldots,b_{\vm_\eps'(t)+1,\vk'_{\vm_\eps'(t)+1}}$$
as well as variables
$$x_{\vm_\eps'(t)+1,i,j},\qquad i\in[\vk'_{\vm_\eps'(t)+1}],\qquad  j\in[\vk'_{\vm_\eps'(t)+1}-1].$$
Check $b_{\vm_\eps'(t)+1,i}$ is adjacent to $x_{\vm_\eps'(t)+1,i,j}'$, $j\in[\vk'_{\vm_\eps'(t)+1}-1]$.
Additionally, we pick a random maximal matching that matches each new check $b_{\vm_\eps'(t)+1,i}$ to (at most) one cavity from $\cC(t)$.
Finally, for each new variable $x_{\vm_\eps'(t)+1,i,j}$ we insert a check $f_{\vm_\eps'(t)+1,i,j}$ that pegs the variable to zero with probability $\beta$ independently.
Let $\cF_t''$ be the set of pairs $(i,j)$ for which such a check was inserted.
Obtain the random matrix $\vA''(t)$ from $\vA_\eps(t)$ by adding rows and columns corresponding to the additional checks and variables, representing each new edge of the Tanner graph by an independent matrix entry distributed as $\CHI$.

\begin{lemma}\label{Claim_Lemma_interpolation_3}
We have 
$$\Erw[\cN(\vm_\eps(t),\vm_\eps'(t)+1)]-\Erw[\cN(\vm_\eps(t),\vm_\eps'(t))]=
	\Erw\brk{\nul(\vA''(t))-\nul(\vA_\eps(t))-\vk'_{\vm_\eps'(t)+1}(\vk'_{\vm_\eps'(t)+1}-1)+|\cF''|}+o(1).$$
\end{lemma}
\begin{proof}
Due to \eqref{eqinterpol} and \Lem~\ref{Lemma_sums},  with probability $1-o(n^{-1})$ we have $|\cC(t)|\geq\vk_{\vm_\eps(t)+1}'$.
Hence, the assertion follows from the construction of $\vA(t)$ and the principle of deferred decisions.
\end{proof}

\begin{lemma}\label{Claim_Lemma_interpolation_4}
Let $Q(\alpha,\beta)=\Erw\brk{\vk(\alpha \beta^{\vk-1}-1)}$ for $\alpha\in[0,1]$.
Then
$$\Erw\abs{\Erw\brk{\nul(\vA''(t))-\nul(\vA(t))-\vk'_{\vm_\eps'(t)+1}(\vk'_{\vm_\eps'(t)+1}-1)+|\cF_t''|\mid\vA(t)}
-Q(\ALPHA,\beta)}=o_\eps(1).$$
\end{lemma}
\begin{proof}
Pick $\zeta=\zeta(\eps)>0$ small enough and $\delta=\delta(\zeta)>0$ smaller still.
As $$\nul(\vA(t))-\nul(\vA''(t))+\vk'_{\vm_\eps'(t)+1}(\vk'_{\vm_\eps'(t)+1}-1)-|\cF_t''|\leq\vk'_{\vm_\eps(t)'+1}$$
and $\Erw[\vk^2]=O_\eps(1)$ we may condition on the event that $\vk'_{\vm_\eps(t)'+1}\leq\eps^{-1}$.
Moreover, \Lem~\ref{Lemma_pinint} shows that we may assume that $\mu_{\cC(t)}$ is $\delta$-symmetric.
Let $\cX$ be the set of cavities adjacent to the new checks $b_{\vm_\eps'(t)+1,i}$.
Also let $$\cU=\cbc{x_{\vm_\eps'(t)+1,i,j}':i\in[\vk'_{\vm_\eps'+1}],j\in[\vk'_{\vm_\eps'+1}-1]}\setminus\cF''_t$$ be the set of new unfrozen variables.
We consider three cases.
\begin{description}
\item[Case 1: $\ALPHA<\eps^4$] since $\vk'_{\vm_\eps'(t)+1}\leq\eps^{-1}$ the probability that any $b_{\vm_\eps(t)'+1,i}$ joins a frozen cavity is $o_\eps(\eps)$.
Hence, \Lem\ \ref{lem:l-wise} shows that with probability $1-o_\eps(\eps)$ the joint distribution $\mu_{\cC(t),\cX}$ 
is within $o_\eps(\eps)$ of the uniform distribution on $\FF_q^{\cX}$.
Therefore, with probability $1-o_\eps(\eps)$ we have
\begin{align}\label{eqClaim_Lemma_interpolation_4_1}
\Erw[\nul(\vA''(t))-\nul(\vA_\eps(t))-\vk'_{\vm_\eps'(t)+1}(\vk'_{\vm_\eps'(t)+1}-1)+|\cF''|\mid\vA_\eps(t)]&=
	-\Erw[\vk'_{\vm_\eps'(t)+1}]+o_\eps(1)=Q(\ALPHA,\beta)+o_\eps(1).
\end{align}
\item[Case 2: $\ALPHA>1-\eps^4$]
since $\vk_{\vm_\eps'(t)+1}\leq\eps^{-1}$ the probability that any $b_{\vm_\eps'(t)+1,i}$ joins an unfrozen cavity is $o_\eps(\eps)$.
Hence, with probability $1-o_\eps(\eps)$ we have
\begin{align}\nonumber
\Erw[\nul(\vA''(t))-\nul(\vA_\eps(t))-\vk'_{\vm_\eps'(t)+1}(\vk'_{\vm_\eps'(t)+1}-1)+|\cF''|\mid\vA_\eps(t)]&=
		-\Erw[\vk'_{\vm_\eps'(t)+1}(\beta^{\vk'_{\vm_\eps'(t)+1}}-1)]+o_\eps(1)\\
		&=Q(\ALPHA,\beta)+o_\eps(1).\label{eqClaim_Lemma_interpolation_4_2}
\end{align}
\item[Case 3: $\eps^4\leq\ALPHA\leq 1-\eps^4$]
the distribution of the number $X$ of frozen cavities in $\cX$ is within $o(1)$ of a binomial distribution $\Bin(\vk_{\vm_\eps(t)+1}',\ALPHA)$ in total variation.
Moreover, \Lem~\ref{lem:l-wise} shows that given any outcome of $X$ the joint distribution $\mu_{\cX}$ is within $o_\eps(\exp(-1/\eps))$ of a product distribution with probability $1-o_\eps(\exp(-1/\eps))$, provided that $\Theta$ is chosen large enough.
Thus, with probability $1-o_\eps(\eps)$ we obtain
\begin{align}\label{eqClaim_Lemma_interpolation_4_3}
\Erw[\nul(\vA''(t))-\nul(\vA_\eps(t))-\vk'_{\vm_\eps'(t)+1}(\vk'_{\vm_\eps'(t)+1}-1)+|\cF'|\mid\vA_\eps(t)]&=
	Q(\ALPHA,\beta)+o_\eps(1).
\end{align}
\end{description}
The assertion follows from \eqref{eqClaim_Lemma_interpolation_4_1}--\eqref{eqClaim_Lemma_interpolation_4_3}.
\end{proof}

\begin{lemma}\label{Claim_Lemma_interpolation_5}
We have
$\frac{\partial}{\partial t}\Erw[\cY_t]=(1-\eps)dn\Erw[(\vk-1)(\beta^{\vk}-1)]/k$.
\end{lemma}
\begin{proof}
This is a straightforward calculation.
\end{proof}

\begin{proof}[Proof of \Prop~\ref{Lemma_interpolation}]
Combining Claims \ref{Claim_PoissonDeriv}--\ref{Claim_Lemma_interpolation_5}, we obtain
\begin{align}\label{eqLemma_interpolation_1}
\frac{\partial}{\partial t}\Erw[\cN_t]+\Erw[\cY_t]&=\frac{(1-\eps)dn}{k}
	\brk{\Erw\brk{\ALPHA^{\vk}-1-\vk(\ALPHA\beta^{\vk-1}-1)+(\vk-1)(\beta^{\vk}-1)}+o_\eps(1)}.
\end{align}
Since $x^k-kxy^{k-1}+(k-1)y^k\geq0$ for all $k\geq2$, $x,y\in[0,1]$, the assertion follows from~\eqref{eqLemma_interpolation_1}.
\end{proof}

\section{Concentration}\label{Sec_conc}

\noindent
At this point we have completed the proofs of \Prop s~\ref{Cor_lower} and~\ref{Lemma_interpolation}.
Thus, we know the expected rank of the random matrix $\vA_\eps$ with about $\eps dn$ cavities.
The aim of this section is to argue that the rank of the actual matrix $\vA$ that does not have any cavities and whose Tanner graph is simple is close to the expected rank of $\vA_\eps$ \whp\
In other words, we need to show that the rank of a random matrix is sufficiently concentrated that forbidding multi-edges is as inconsequential as conditioning on the event
	$$\cD=\cbc{\sum_{i=1}^n\vd_i=\sum_{i=1}^{\vm}\vk_i}.$$
Our main tool will be the local limit theorem for sums of independent random variables, which we use in \Sec~\ref{Sec_llt} to calculate the probability of $\cD$.
There we will also study the conditional distributions of $\vm$, $\sum_{i=1}^n\vd_i$ and $\sum_{i=1}^{\vm}\vk_i$ given $\cD$.
Further, in \Sec~\ref{Sec_welldef} we prove \Prop~\ref{Lemma_welldef}.
In \Sec~\ref{Sec_Cor_lower} we then derive \Prop~\ref{Cor_lower}.
Moreover, \Sec~\ref{Sec_Lemma_nulconc} contains the proof of \Prop~\ref{Lemma_nulconc}.
Finally, \Sec~\ref{Sec_LDPC} contains the proof of \Thm~\ref{Thm_LDPC}.

\subsection{The local limit theorem}\label{Sec_llt}

We recall the local limit theorem for sums of independent random variables.

\begin{theorem}[{\cite[p.~130]{Durrett}}]\label{Thm_llt}
Suppose that $(\vX_i)_{i\geq1}$ is a sequence of i.i.d.\ variables that take values in $\ZZ$ such that the greatest common divisor of the support of $\vX_1$ is one.
Also assume that $\Erw[\vX_1^2]=\sigma^2\in(0,\infty)$.
Then
\begin{align*}
\lim_{n\to\infty}\sup_{z\in\ZZ}\abs{\sqrt n\pr\brk{\sum_{i=1}^n\vX_i=z}-\frac{\exp(-(z-n\Erw[\vX_1])^2/(2n\sigma^2))}{\sqrt{2\pi}\sigma}}&=0.
\end{align*}
\end{theorem}

\noindent
As an application of \Thm~\ref{Thm_llt} we obtain the following estimate.
Because $\Erw[\vd^r]+\Erw[\vk^r]<\infty$ for an $r>2$, the event
\begin{align}\label{eqM}
\cM=\cbc{\max_{i\in[n]}\vd_i+\max_{i\in[\vm]}\vk_i\leq \sqrt n/\log^9 n}\qquad\mbox{satisfies}\qquad
\pr\brk{\cM}=1-o(1).
\end{align}

\begin{corollary}\label{Cor_llt}
If $\gcd(\vk)$ divides $n$, then $\pr\brk{\cD}=\Theta(n^{-1/2})$ and $\pr\brk{\cD\mid\cM}=\Theta(n^{-1/2})$.
\end{corollary}
\begin{proof} For $\pr\brk{\cD\mid\cM}$
there are several cases to consider.  
First, that $\Var(\vd)=\Var(\vk)=0$, i.e., $\vd,\vk$ are both atoms.
Since $\vm$ is a Poisson variable with mean $dn/k$ we find $\pr\brk{\cD\mid\cM}=\pr\brk{\vm=d n/k}=\Theta(n^{-1/2})$.

Second, suppose that $\Var(\vd)>0$ but $\Var(\vk)=0$.
Then \Thm~\ref{Thm_llt} 
and~\eqref{eqM} show that
\begin{align}\label{eqCor_llt1}
\pr\brk{\abs{dn-\sum_{i=1}^n\vd_i}\leq\sqrt n\wedge k\mbox{ divides }\sum_{i=1}^n\vd_i\mid\cM}&=\Omega(1).
\end{align}
Further, given $\abs{dn-\sum_{i=1}^n\vd_i}\leq\sqrt n$ and $k|\sum_{i=1}^n\vd_i$, the event $k\vm=\sum_{i=1}^n\vd_i$
has probability $\Theta(n^{-1/2}) $ by the local limit theorem for the Poisson distribution.

The case that $\Var(\vd)=0$ but $\Var(\vk)>0$ can be dealt with similarly.
Indeed, pick a large enough number $L>0$ and let $I=\cbc{i\in[\vm]:\vk_i>L}$, $\vm'=|I|$, $\vm''=\vm-|I|$, $S'=\sum_{i\in I}\vk_i$
and $S''=\sum_{i\in[\vm]\setminus I}\vk_i$.
Then $\vm',\vm''$ are stochastically independent, and so are $S',S''$.
Moreover, since $S'$ satisfies the central limit theorem we have
\begin{align}\label{eqCor_llt1a}
\pr\brk{|S'-\Erw[S'\mid\cM]|\leq\sqrt n\mid\cM}&=\Omega(1).
\end{align}
Further, \Thm~\ref{Thm_llt} applies to $S''$, which is distributed as $\sum_{i=1}^{\vm}\vk_i\vecone\{\vk_i\leq L\}$.
Hence, as $n$ is divisible by $\gcd(k)$,
\begin{align}\label{eqCor_llt1b}
\pr\brk{S'+S''=dn\mid|S'-\Erw[S'\mid\cM]|,\cM }&=\Omega(n^{-1/2}).
\end{align}
Thus, \eqref{eqCor_llt1a} and \eqref{eqCor_llt1b} show that $\pr\brk{\cD\mid\cM}=\Omega(n^{-1/2})$.
The upper bound $\pr\brk{\cD\mid\cM}=O(n^{-1/2})$ follows from the uniform upper bound from \Thm~\ref{Thm_llt}.

A similar argument applies in the final case $\Var(\vd)>0$, $\Var(\vk)>0$.
Indeed, \Thm~\ref{Thm_llt} and~\eqref{eqM} yield
\begin{align}\label{eqCor_llt3}
\pr\brk{\gcd(\vk)\mbox{ divides }\sum_{i=1}^{n}\vd_i\mbox{ and }\abs{dn-\sum_{i=1}^{n}\vd_i}\leq\sqrt n\mid\cM}&=\Omega(1).
\end{align}
Moreover, \eqref{eqCor_llt1a} remains valid regardless the variance of $\vd$.
Hence, applying \Thm~\ref{Thm_llt} to $S''$, we obtain
\begin{align}\label{eqCor_llt1c}
\pr\brk{S'+S''=\sum_{i=1}^n\vd_i\,\bigg|\,
\gcd(\vk)\mbox{ divides }\sum_{i=1}^{n}\vd_i,\,\abs{dn-\sum_{i=1}^{n}\vd_i}\leq\sqrt n,\,
|S'-\Erw[S'\mid\cM]|,\cM }&=\Omega(n^{-1/2}).
\end{align}
Combining~\eqref{eqCor_llt3} and~\eqref{eqCor_llt1c}, we see that $\pr\brk{\cD\mid\cM}=\Omega(n^{-1/2})$.
The matching upper bound $\pr\brk{\cD\mid\cM}=O(n^{-1/2})$ follows from the universal upper bound from \Thm~\ref{Thm_llt} once more. The treatment for $\pr\brk{\cD}$ is similar and simpler.
\end{proof}

\subsection{Proof of \Prop~\ref{Lemma_welldef}}\label{Sec_welldef}
\Prop~\ref{Lemma_welldef} is a consequence of \Lem~\ref{Cor_llt} and the following statement.

\begin{lemma}\label{Lemma_simple}
We have $\pr\brk{\vA_0\in\cS|\cD}=\Omega(1)$.
\end{lemma}

\noindent
Thus, we are left to prove \Lem~\ref{Lemma_simple}.

\begin{lemma}\label{Lemma_sm}
On the event $\cM\cap\cD$ we have
$\frac1n\sum_{i=1}^n\vd_i^2\to \Erw[\vd^2]$, $\frac1n\sum_{i=1}^{\vm}\vk_i^2\to d\Erw[\vk^2]/k$ in probability.
\end{lemma}
\begin{proof}
We will only prove the statement about the $\vk_i$; the same (actually slightly simplified) argument applies to the $\vd_i$.
By tail bounds on the Poisson distribution we may condition on $\{\vm=m\}$ for some integer $m$ with $|m-dn/k|\leq\sqrt n\ln n$.
Fix a small $\delta>0$ and a large enough $L=L(\delta)>0$.
Given $\vm=m$ the variables $Q_j=\sum_{i\in[\vm]}\vecone\{\vk_i=j\}$ have a bionomial distribution.
Therefore, the Chernoff bound yields
$\pr\brk{|Q_j-dn\pr\brk{\vk=j}/k|\leq\sqrt n\ln n\mid\vm=m}=1-o(1/n)$ for any $j\leq L$.
Hence, \eqref{eqM} and \Cor~\ref{Cor_llt} yield
\begin{align}\label{eqLemma_sm_1}
\pr\brk{\forall j\leq L:|Q_j-dn\pr[\vk=j]/k|\leq \sqrt n\ln n\mid\cD\cap\cM,\,\vm=m}=1-o(1).
\end{align}
Further, let 
\begin{align*}
R_h&=\sum_{j\geq3}\vecone\{(1+\delta)^{h-1}L< j\leq (1+\delta)^{h}L\wedge \sqrt n/\ln^9n\}Q_j,&
\bar R_h&=m\sum_{j\geq3}\vecone\{(1+\delta)^{h-1}L< j\leq (1+\delta)^{h}L\wedge \sqrt n/\ln^9n\}\pr\brk{\vk=j}
\end{align*}
and let $\cH$ be the set of all integers $h\geq1$ with $(1+\delta)^{h-1}L\leq \sqrt n/\ln^9n$.
Then the Chernoff bound implies that
\begin{align}\label{eqLemma_sm_2}
\pr\brk{\forall h\in\cH:\abs{R_h-\bar R_h}>\delta\bar R_h+\ln^2n\mid\cD\cap\cM,\vm=m}&=o(n^{-1}).
\end{align}
Finally, if $|Q_j-dn\pr\brk{\vk=j}/k|\leq\sqrt n\ln n$ for all $j\leq L$ and 
$\abs{R_h-\bar R_h}\leq\delta\bar R_h+\ln^2n$ for all $h\in\cH$, then
\begin{align*}
\frac1n\sum_{i=1}^m\vk_i^2&\leq o(1)+\frac dk\Erw\brk{\vk^2\vecone\{\vk\leq L\}}
	+\frac{d}{kn}\sum_{h\in\cH}(1+\delta)^{2h}L^2R_h\\
&=o(1)+\frac dk\Erw\brk{\vk^2\vecone\{\vk\leq L\}}
	+\frac{d}{kn}\sum_{h\in\cH}(1+\delta)^{2h+1}L^2\bar R_h\leq\frac{(1+\delta)d}k\Erw[\vk^2]+o(1),&\mbox{and analogously}\\
\frac1n\sum_{i=1}^m\vk_i^2&\geq \frac{(1-\delta)d}k\Erw[\vk^2]+o(1).
\end{align*}
Since this holds true for any fixed $\delta>0$, the assertion follows from \eqref{eqLemma_sm_1} and \eqref{eqLemma_sm_2}.
\end{proof}

\begin{lemma}\label{Lemma_doubleEdges}
Let $Y$ be the number of multi-edges of the Tanner graph and let $\ell\geq1$.
There is $\lambda>0$ such that on
	$$\cM\cap\cD\cap
\cbc{\sum_{i=1}^n\vd_i=dn+o(n),\,\sum_{i=1}^n\vd_i^2=n\Erw[\vd^2]+o(n)}\cap\cbc{
\sum_{i=1}^{\vm}\vk_i^2=dn\Erw[\vk^2]/k+o(n)}\cap\{\vm=dn/k+o(n)\}$$ we have
$
\Erw\brk{\prod_{i=1}^\ell Y-i+1\,\bigg|\,(\vd_i)_{i\in[n]},(\vk_i)_{i\in\vm}}=\lambda^\ell+o(1).
$
\end{lemma}
\begin{proof}
To estimate the $\ell$th factorial moments of $Y$ for $\ell\geq1$, we split the random variable into a sum of indicator variables.
Specifically, let $U_\ell$ be the set of all families $(a_i,y_i,w_i)_{i\in\ell}$ with $a_i\in[\vm]$, $y_i\in[\vn]$ and $2\leq w_i\leq \vk_{a_i}\wedge\vd_{x_i}$ such that the pairs $(a_i,y_i)$ are pairwise distinct.
Let $Y((a_i,x_i,w_i)_{i\in[\ell]})$ be the number of ordered $\ell$-tuples of multi-edges comprising precisely $w_i$ edges between check $a_i$ and variable $x_i$ for each $i$.
Then
\begin{align*}%\label{eqLemma_doubleEdges4}
\prod_{h=1}^\ell Y-h+1&=\sum_{(a_i,y_i,w_i)_{i\in[\ell]}\in U_\ell}Y((a_i,y_i,w_i)_{i\in[\ell]}).
\end{align*}
Moreover, letting $w=\sum_iw_i$, we have
\begin{align}\label{eqLemma_doubleEdges5}
\Erw[Y((a_i,y_i,w_i)_{i\in[h]})\mid(\vd_i)_{i\in[n]},(\vk_i)_{i\in\vm}]&\sim\frac{1}{(dn)^w}
		\prod_{i=1}^\ell
\bink{\vd_{y_i}}{w_i}\bink{\vk_{a_i}}{w_i}w_i!\enspace.
\end{align}
Now, for a sequence $\vw=(w_1,\ldots,w_\ell)$ let $Y_{\vw}=\sum_{(a_i,y_i,w_i)_{i\in[\ell]}\in U_\ell}Y((a_i,y_i,w_i)_{i\in[\ell]})$.
Then \eqref{eqLemma_doubleEdges5} yields
\begin{align*}\nonumber
\Erw[Y_{\vw}\mid(\vd_i)_{i\in[n]},(\vk_i)_{i\in\vm}]&\leq O(n^{-w})\prod_{i=1}^\ell\bc{\sum_{j=1}^n\vd_{j}^{w_i}}\bc{\sum_{j=1}^{\vm}\vk_{j}^{w_i}}
	\leq O(n^{-w})\max_{j\in[n]}\vd_j^{w-2\ell}\max_{j\in[\vm]}\vk_j^{w-2\ell}\bc{\sum_{j=1}^n\vd_{j}^{2}}^\ell\bc{\sum_{j=1}^{\vm}\vk_{j}^{2}}^\ell\\
&\leq O(n^{2\ell-w})\max_{j\in[n]}\vd_j^{w-2\ell}\max_{j\in[\vm]}\vk_j^{w-2\ell}=O(\ln^{2\ell-w} n).
%\label{eqLemma_doubleEdges6}
\end{align*}
As a consequence,
\begin{align}\label{eqLemma_doubleEdges7}
\sum_{\vw:w>2\ell}\Erw[Y_{\vw}\mid(\vd_i)_{i\in[n]},(\vk_i)_{i\in\vm}]&=o(1).
\end{align}
Finally, invoking \eqref{eqLemma_doubleEdges5}, we obtain
\begin{align}\label{eqLemma_doubleEdges8}
\Erw[Y_{(2,\ldots,2)}\mid(\vd_i)_{i\in[n]},(\vk_i)_{i\in\vm}]&\sim\lambda^\ell,\qquad\mbox{where}&
\lambda&\sim
\frac{\bc{\sum_{i=1}^n\vd_i(\vd_i-1)}\bc{\sum_{i=1}^n\vk_i(\vk_i-1)}}{2(dn)^2}.
\end{align}
Combining (\ref{eqLemma_doubleEdges5}), \eqref{eqLemma_doubleEdges7} and \eqref{eqLemma_doubleEdges8} completes the proof.
\end{proof}

\begin{corollary}\label{Cor_doubleEdges}
We have $\pr\brk{\cS\mid\cD\cap\cM}=\Omega(1)$.
\end{corollary}
\begin{proof}
\Lem s~\ref{Lemma_sm} and~\ref{Lemma_doubleEdges} show together with inclusion/exclusion that 
\whp\ on the event $\cM\cap\cD$,
$$\pr\brk{Y=0\mid(\vd_i)_{i\in[n]},(\vk_i)_{i\in\vm}}=\exp(-\lambda)=\Omega(1).$$
Since $\cS=\{Y=0\}$, the assertion follows.
\end{proof}

\begin{proof}[Proof of \Lem~\ref{Lemma_simple}]
The assertion follows immediately from \eqref{eqM}, \Cor~\ref{Cor_llt} and \Cor~\ref{Cor_doubleEdges}.
\end{proof}

\subsection{Proof of \Prop~\ref{Cor_lower}}\label{Sec_Cor_lower}

\noindent
The random matrix $\vA$ has $n$ columns and $\vm\disteq\Po(dn/k)$ rows, with the column and row degrees drawn from the distributions $\vd$ and $\vk$.
By comparison, $\vA_\eps$ has slightly fewer, namely $\vm_\eps\disteq\Po((1-\eps)dn/k)$ rows.
One might therefore think that the proof of \Lem~\ref{Cor_lower} is straightforward, as it appears that $\vA$ is obtained from $\vA_\eps$ by simply adding another random $\Po(\eps dn/k)$ rows.
Since adding $O(\eps n)$ rows cannot reduce the nullity by more than $O(\eps n)$, the bound on
 $\Erw[\nul(\vA)] -\Erw[\nul(\vA_{\eps})]$ appears to be immediate.
But there is a catch.
Namely, in constructing $\vA$  we condition on the event $\cD=\{\sum_{i=1}^{n}\vd_i=\sum_{i=1}^{\vm}\vk_i\}$.
Thus,  $\vA_\eps$ does not have the same distribution as the top $\Bin(\vm,1-\eps)$ rows of $\vA$ since the conditioning might distort the degree distribution.
We need to show that this distortion is insignificant.
To this end, recall that $\vm_\eps\disteq\Po((1-\eps)dn/k)$.

\begin{lemma}\label{Lemma_conc}
\Whp\ we have
\begin{align*}
\pr\brk{\abs{\nul(\vA_\eps)-\Erw[\nul(\vA_\eps)\mid\vm_\eps,(\vd_i)_{i\geq1},\,(\vk_i)_{i\geq1}]}>\sqrt n\ln n\mid\vm_\eps,(\vd_i)_{i\geq1},\,(\vk_i)_{i\geq1}}&=o(1),
\end{align*}
\end{lemma}
\begin{proof}
\Lem~\ref{Lemma_sums} shows that $\sum_{i=1}^n\vd_i,\sum_{i=1}^{\vm_\eps}\vk_i=O(n)$
and $\sum_{i=1}^{\vm_\eps}\vk_i\leq\sum_{i=1}^n\vd_i$  with probability $1-o(n^{-1})$.
Assuming that this is so, consider a filtration $(\fF_t)_{t\leq\sum_{i=1}^{\vm_\eps}\vk_i}$ that reveals the random matching $\vec\Gamma_\eps$ one clone at a time.
Then 
$$\abs{\Erw[\nul(\vA_\eps)\mid\fF_{t+1},\vm_\eps,(\vd_i)_{i\geq1},\,(\vk_i)_{i\geq1}-
\Erw[\nul(\vA_\eps)\mid\fF_{t},\vm_\eps,(\vd_i)_{i\geq1},\,(\vk_i)_{i\geq1}}]\leq O(1)$$
for all $t$.
Therefore, the assertion follows from Azuma's inequality.
\end{proof}

\begin{proof}[Proof of \Prop~\ref{Cor_lower}]
We will couple $\vA_{\eps}$ and $\vA$ so that they differ by $O(\eps n)$ rows with probability at least $1-\eps$. The proposition follows immediately, as altering one row can affect the nullity of a matrix by $O(1)$. 

To construct the coupling we first generate the following parameters for $\vA_{\eps}$. Parameter $\Theta=\Theta(\eps)$ is given. Generate $\bftheta\in [\Theta]$ uniformly at random. Generate $\bfm_{\eps}$ and then $\bfm_{\eps}$ check nodes. Each check node $a_i$ is associated with an integer $\bfk_i$ which is an independent copy of $\bfk$. To distinguish rows for $\vA_{\eps}$ from $\vA$, we colour these check nodes red. Add additional $\bftheta$ red check nodes $p_1,\ldots, p_{\bftheta}$.

Next, we generate $n$ variable nodes where variable node $x_i$ is associated with $\bfd_i$, which is an independent copy of $\bfd$. 

Next we generate parameters $\bfm$, and $\bfn_j$ for $\vA$, where $\bfn_j$ denotes the number of rows in $\vA$ that contains exactly $j$ non-zero entries. The following claim follows straightforwardly by the Chernoff bounds and Corollary~\ref{Cor_llt}.  

\begin{claim}\label{c:dominance} For every $\eps>0$ there is a sufficiently large $L=L(\eps)>0$ such that
\[
\pr\brk{\sum_{j\ge L} \bfn_j >\eps n} <\eps,\quad \pr\brk{\sum_{j\ge L}\sum_{i=1}^{\bfm_{\eps}} \vecone\{\bfk_i=j\} >\eps n} <\eps, \quad  \pr\brk{\bfn_j>\sum_{i=1}^{\bfm_{\eps}} \vecone\{\bfk_i=j\} \ \mbox{for all $j\le L$}} = 1-o(1).
\]
\end{claim}

Assume events $\{\sum_{j\ge L} \bfn_j \le \eps n\}$, $\{\sum_{j\ge L}\sum_{i=1}^{\bfm_{\eps}} \vecone\{\bfk_i=j\} \le \eps n\}$, and $\{\bfn_j>\sum_{i=1}^{\bfm_{\eps}} \vecone\{\bfk_i=j\} \ \mbox{for all $j\le L$}\}$ hold. Uncolour all check nodes $a_i$ where $\bfk_i\le L$. For each $j\le L$, generate $\bfn_j-\sum_{i=1}^{\bfm_{\eps}} \vecone\{\bfk_i=j\}$ check nodes,  colour them blue, and associate $\bfk_i=j$ to each such check node $a_i$. For each $j>L$, generate $\bfn_j$ blue check nodes and similarly each such node $a_i$ is associated with integer $j$. 

Now the Tanner graph of $\vA_{\eps}$ is generated by taking a random maximal matching from the clones of all {\em uncoloured} and {\em red} check nodes $\{a_i\}\times [\bfk_i]$ (excluding check nodes $p_1,\ldots,p_{\bftheta}$) to the set of variable clones $\cup_{j=1}^n\{x_j\}\times [\bfd_j]$, and then adding an edge between $p_i$ and $a_i$ for $1\le i\le \bftheta$. The Tanner graph of $\vA$ is generated by removing all matching edges from the clones of the {\em red} check nodes, and removing edges between $p_i$ and $a_i$ for $1\le i\le \bftheta$, and then matching all clones of the {\em blue} check nodes to the remaining clones of the variable nodes. 

By the construction, marginally, the induced random matrix of the first Tanner graph is  distributed as $\vA_{\eps}$, and that of the second Tanner graph is distributed as $\vA$. Moreover, Claim~\ref{c:dominance} and the fact that $\bftheta\le \Theta=O(1)$ ensures that with probability at least $1-\eps$, $|\nul{\vA_{\eps}}-\nul{\vA}|=O(\eps n)$. The proposition follows by taking $\eps\to 0$.
\end{proof}

%%%%%
%%%%%

\subsection{Proof of \Thm~\ref{Thm_LDPC}}\label{Sec_LDPC}

\noindent
Suppose that $D,K$ are polynomials that $\gcd(\vk)$ divides $n$ and that $m=dn/k$ is an integer divisible by $\gcd(\vd)$.

\begin{lemma}\label{Lemma_LDPC_degs} 
For any $\ell$ in the support of $\vd$ and any $\ell'$ in the support of $\vk$ we have
\begin{align*}
\pr\brk{\abs{n\pr\brk{\vd=\ell}-\sum_{i=1}^n\vecone\{\vd_i=\ell\}}\leq\sqrt n\ln n\mid\cD}&=1-o(1),&
\pr\brk{\abs{m\pr\brk{\vk=\ell}-\sum_{i=1}^m\vecone\{\vk_i=\ell'\}}\leq\sqrt n\ln n\mid\cD}&=1-o(1).
\end{align*}
\end{lemma}
\begin{proof}
Since $D,K$ are polynomials, 
$\sum_{i=1}^n\vecone\{\vd_i=\ell\}$ and $\sum_{i=1}^m\vecone\{\vk_i=\ell'\}$ are binomial variables with mean $\Omega(n)$.
Therefore, the Chernoff bound shows that the probability of a deviation of more than $\sqrt n\ln n$ is $O(1/n^2)$.
As \Cor~\ref{Cor_llt} shows that $\pr\brk{\cD}=\Omega(1/n)$, the assertion follows from Bayes' rule.
\end{proof}

\begin{proof}[Proof of \Thm~\ref{Thm_LDPC}]
Let $\cE$ be the event that for all $\ell,\ell'$ in the support of $\vd,\vk$,  respectively,
\begin{align*}
\abs{n\pr\brk{\vd=\ell}-\sum_{i=1}^n\vecone\{\vd_i=\ell\}}\leq\sqrt n\ln n,\quad
\abs{m\pr\brk{\vk=\ell'}-\sum_{i=1}^m\vecone\{\vk_i=\ell'\}}\leq\sqrt n\ln n.
\end{align*}
Then \Lem~\ref{Lemma_LDPC_degs} and \Thm~\ref{thm:rank} yield $\pr[\cE]\sim1$.
Moreover, let $\cE'$ be the event that $n\pr\brk{\vd=\ell}=\sum_{i=1}^n\vecone\{\vd_i=\ell\}$
and $m\pr\brk{\vk=\ell'}=\sum_{i=1}^m\vecone\{\vk_i=\ell'\}$ for all $\ell,\ell'$.
We are going to couple $\vA_0$ given $\cE$ with $\vA_0$ given $\cE'$ such that
\begin{align}\label{eqThm_LDPC1}
\Erw\abs{\rank(\vA_0\mid\cE)-\rank(\vA_0\mid\cE')}&=o(n).
\end{align}
Since $\rate_n(\vd,\vk)=\frac1n\Erw\brk{\nul(\A)\mid\fD}=\frac1n\Erw\brk{\nul(\A_0)\mid\cE'\cap\cS}$
and $\pr[\cS\mid\cE']=\Omega(1)$ by \Lem~\ref{Lemma_doubleEdges}, the assertion follows from~\eqref{eqThm_LDPC1} and \Thm~\ref{thm:rank}.

To construct the coupling first choose a random matrix $\vA'$ from the distribution of $\vA_0$ given $\cE$.
Let $N_\ell'=\sum_{i=1}^n\vecone\{\vd_i=\ell\}$ and $M_\ell'=\sum_{i=1}^m\vecone\{\vk_i=\ell\}$.
Further, let $\cV_\ell'$ comprise the first $n\pr\brk{\vd=\ell}\wedge N_\ell'$ indices $i$ with $\vd_i=\ell$.
Analogously, let $\cF_\ell'$ contain the first $m\pr\brk{\vk=\ell}\wedge M_\ell'$ indices $i\in[m]$ with $\vk_i=\ell$.
Let $X$ be the number of edges of the Tanner graph of $\vA'$ that run between $\bigcup_\ell\cV_\ell'$ and $\bigcup_\ell\cF_\ell'$.

Moreover, let $Y$ be a random variable distributed as follows.
Draw $\vA''$ from the distribution of $\vA_0$ given $\cE'$.
Let $\cV_\ell''$ contain the first $n\pr\brk{\vd=\ell}\wedge N_\ell'$ variables of degree $\ell$, and similarly let $\cF_\ell''$ contain the first $m\pr\brk{\vk=\ell}\wedge M_\ell'$ constraints of degree $\ell$.
Then $Y$ is the number of edges of the Tanner graph of $\vA''$ linking $\bigcup_\ell\cV_\ell''$ and $\bigcup_\ell\cF_\ell''$.

Because $\vA'\in\cE$ and $\vd,\vk$ are supported on finite sets of integers, we clearly have
\begin{align}\label{eqThm_LDPC2}
\abs{X-Y}\leq O(\sqrt n\ln n).
\end{align}
Letting $W=X\wedge Y$, we couple $\vA'$ and $\vA''$ as follows.
Obtain the Tanner graph of $\vA'''$ by creating a random matching of size $W$ between the clones of the variables $\bigcup_\ell\cV_\ell'$ and the checks $\bigcup_\ell\cF_\ell'$.
Then obtain $\vA'$, $\vA''$ from $\vA'''$ by randomly matching the remaining clones as determined by the respective degree sequences.
The matrix entries corresponding to the edges of the Tanner graph are drawn independently from $\CHI$ in such a way that those associated with the Tanner graph of $\vA'''$ agree.
By the principle of deferred decisions the resulting matrices $\vA'$, $\vA''$ are distributed as $\vA$ given $\cE$ and $\vA$ given $\fD$.
Furthermore, because \eqref{eqThm_LDPC2} shows that $\vA'$ and $\vA''$ disagree in no more than $O(\sqrt n\ln n)$ entries, we obtain (\ref{eqThm_LDPC1}).
\end{proof}

%%%
%%%

 \section{Proof of \Thm~\ref{Thm_tight}}\label{Sec_tight} 

Recall that
\begin{equation}\label{def:f}
\phi(\alpha)=1-\alpha-\frac{1}{d}D'\left(1-\frac{K'(\alpha)}{k}\right).
\end{equation}

It is sufficient to prove that if conditions (i) and (ii) are satisfied then 
(a)
$\max_{\alpha\in[0,1]}\Phi(\alpha)=\max\{\Phi(0),\Phi(\rho)\}$;
and (b) $\phi'(\rho)<0$ unless
  \begin{equation}
  \pr(\bfd=1)=0\quad \mbox{and}\quad 2(\ex\bfk-1)\pr(\bfd=2)>\ex \bfd. \label{exception}
  \end{equation} 
 
 Since $\Phi(\alpha)$ is continuous on $[0,1]$, the maximum occurs at either 0 or 1 or at a stable point.

  \smallskip
 
 \noindent {\em Case A: $\Var(\bfk)=0$.} In this case, $\bfk=k$ always and thus $K(\alpha)=\alpha^k$. 
 Then,
 \begin{align*}
 \phi(\alpha)&=1-\alpha-\frac{1}{d}D'(1-\alpha^{k-1})\\
 \phi'(\alpha)&=-1+\frac{(k-1)\alpha^{k-2}}{d}D''(1-\alpha^{k-1})\\
 \phi''(\alpha)&=\frac{k-1}{d}\alpha^{k-3}\Big((k-2)D''(1-\alpha^{k-1})-(k-1)D'''(1-\alpha^{k-1})\alpha^{k-1}\Big)\\
 &=\frac{k-1}{d}\alpha^{k-3}\Big((k-2)D''(t)-(k-1)D'''(t)(1-t)\Big) \quad \mbox{where $t=1-\alpha^{k-1}$}.
 \end{align*}
 Hence,
 \begin{align}
&& \phi(0)=0 && \phi(1)=-\frac{1}{d}D'(0)\le 0 && \label{f}\\
 && \phi'(0)=-1 && \phi'(1)=-1+\frac{k-1}{d}D''(0).&& \label{f'}
 \end{align}
Recall that $\Phi'(\alpha)=\frac{d}{k}K''(\alpha) \phi(\alpha)$. We have $K''(\alpha)> 0$ for all $\alpha\in(0,1]$.  By~\eqn{f} we have $\Phi'(1)\le 0$ and thus the supremum of $\Phi(\alpha)$ can only occur at 0 or a stable point. In all of the following sub-cases, we will prove that $\phi''(\alpha)$ has at most 1 root in $[0,1]$ (except for some trivial cases that we discuss separately). It follows immediately that $\phi$ can have at most three roots on $[0,1]$ including the trivial one at $\alpha=0$. Now we prove that this implies claims (a) and (b).
 
 If $\phi$ has only a trivial root, then so is $\Phi'(\alpha)$. Thus, $\alpha=0$ is the unique maxima of $\Phi(\alpha)$ and $\rho=0$. This verifies (a). As $\phi'(0)=-1$ we immediately have $\phi'(\rho)<0$.
 
If $\phi$ has two roots, then the larger root is $\rho$. Since $\phi'(0)<0$, in this case, $\phi$ is negative in $(0,\rho)$ and positive in $(\rho,1)$. This is only possible when $\phi(1)=0$, which requires $\pr(\bfd=1)=0$. In this case, $\rho=1$. %, but $f'(\rho_{\phi,\tau})>0$. This is an example where $(\D,\K)$ is not nice. \jcom{This is also the case that the original factor graph is an unstable 2-core.}
Next we consider two further cases: (i) $2(k-1)\pr(\bfd=2)>d$ corresponding to $\phi'(1)>0$; (ii) $2(k-1)\pr(\bfd=2)<d$ corresponding to $\phi'(1)<0$. As $\phi$ has only two roots, case (ii) obviously cannot happen. Thus, it means that the only situation that $\phi$ has two roots would be $\pr(\bfd=1)=0$ and $2(k-1)\pr(\bfd=2)>d$, as in~\eqn{exception}. In this situation we are only required to verify (a). Note that $\phi$ is negative in $(0,1)$ as $\rho=1$. It follows then that $\Phi(\alpha)$ is a decreasing function in $(0,1)$. Hence, $\alpha=0$ is the unique maxima, as desired.

If $\phi$ has three roots, then there is a root $\rho^*$ between $0$ and $\rho$. Then $\phi$ is negative in $(0,\rho^*)$ and positive in $(\rho^*,\rho)$. As $K''(\alpha)>0$ for all $\alpha\in (0,1]$, the sign of $\phi$  implies that $\rho^*$ is a local minima and $\rho$ is a local maxima. This verifies (a).  Moreover, as $\phi$ is positive in $(\rho^*,\rho)$ and $\phi(\rho)=0$, $\phi'(\rho)<0$ follows immediately. 
 
 \medskip
 
 \noindent {\em Case A1: $\Var(\bfk)=0$ and $\Var(\bfd)=0$.} In this case $\bfd=d$. Then $D(\alpha)=\alpha^d$. If $d\ge 3$ then
 \begin{align*}
 \phi''(\alpha)&=\frac{k-1}{d}\alpha^{k-3}\Big((k-2)d(d-1)t^{d-2}-(k-1)d(d-1)(d-2)t^{d-3}(1-t)\Big)\\
 &=(k-1)(d-1)t^{d-3} \alpha^{k-3}\Big((k-2)t-(k-1)(d-2)(1-t)\Big) \quad \mbox{where $t=1-\alpha^{k-1}$}.
 \end{align*}
 Obviously, $\phi''(\alpha)$ has a unique root in $[0,1]$. 
 
 If $d=1$ then $\phi'(\alpha)=-1$ and so $\phi$ has only a trivial root at $\alpha=0$; 
 If $d=2$ then $\phi''(\alpha)>0$ in $(0,1)$ and so $\phi$ is convex and thus has only a trivial root at $\alpha=0$ by~\eqn{f}. Hence for $d\le 2$, $\rho=0$ and is the unique maxima. Claims (a) and (b) hold trivially.
 \smallskip
  
 \noindent {\em Case A2: $\Var(\bfk)=0$ and $\bfd\sim \po_{\ge r}(\lambda)$.} In this case $D(\alpha)=h_r(\lambda\alpha)/h_r(\lambda)$, where
 \begin{align}
 h_r(x)&=\sum_{j\ge r}\frac{x^j}{j!} \quad \mbox{for all nonnegative integers $x$}; \label{h1}\\
 h_r(x)&=e^x \quad \mbox{for all negative integers $x$}.\label{h2}
 \end{align}
 Then, for all integers $t$,
  \[
  D'(\alpha)=\frac{\lambda h_{r-1}(\lambda \alpha)}{h_r(\lambda )},\  D''(\alpha)=\frac{\lambda ^2h_{r-2}(\lambda \alpha)}{h_r(\lambda )}, \ D'''(\alpha)=\frac{\lambda ^3h_{r-3}(\lambda \alpha)}{h_r(\lambda )}.
  \]
   Since $\ex \bfd=d$, it requires that $\lambda$ satisfies
   \begin{equation}
   D'(1)=\frac{\lambda h_{r-1}(\lambda)}{h_r(\lambda)}=d.\label{lambda}
   \end{equation}
Thus,
 \[
 \phi''(\alpha)=\frac{(k-1)d\alpha^{k-3}}{h_r(\lambda )}\Big((k-2)h_{r-2}(\lambda  t)-(k-1)(1-t)h_{r-3}(\lambda  t)\Big).
 \]
 Solving $\phi''(\alpha)=0$ yields
 \begin{equation}
\frac{k-1}{k-2}(1-t)= \frac{h_{r-2}(\lambda  t)}{h_{r-3}(\lambda  t)}=1-\frac{h_{r-3}(\lambda  t)-h_{r-2}(\lambda  t)}{h_{r-3}(\lambda t)}. \label{eq:truncate}
 \end{equation}
 The right hand side above is obviously a constant function if $r\le 2$. If $r\ge 3$, then $h_{r-3}(\lambda t)-h_{r-2}(\lambda t)=(\lambda  t)^{r-3}/(r-3)!$, and  $h_{r-3}(\lambda t)$ is a power series of $\lambda t$ with minimum degree $r-3$. Hence, by dividing $(\lambda t)^{r-3}/(r-3)!$ from both the numerator and the denominator, we immediately get that the right hand side of~\eqn{eq:truncate} is an increasing function. However the left hand side of~\eqn{eq:truncate} is a decreasing function. Hence~\eqn{eq:truncate} has at most one solution, implying that $\phi''(\alpha)$ has at most one root. \smallskip

  \noindent {\em Case B: $\bfk\sim \po_{\ge s}(\gamma)$ where $s\ge 3$.} We must have $\gamma$ satisfy
  \[
  \frac{\gamma h_{s-1}(\gamma)}{h_s(\gamma)}=k,
  \]
 so that $\ex\bfk=k$. Now we have $K(\alpha)=h_s(\gamma\alpha)/h_s(\gamma)$, where $h_s$ is defined as in~\eqn{h1} and~\eqn{h2}.
 Thus, 
 \begin{align*}
 \phi(\alpha)&=1-\alpha-\frac{1}{d}D'\left(1-\frac{h_{s-1}(\gamma \alpha)}{h_{s-1}(\gamma)}\right)\\
 \phi'(\alpha)&=-1+\frac{\gamma h_{s-2}(\gamma\alpha)}{d h_{s-1}(\gamma)}D''\left(1-\frac{h_{s-1}(\gamma \alpha)}{h_{s-1}(\gamma)}\right)\\
 \phi''(\alpha)&=\frac{\gamma^2}{d h_{s-1}(\gamma)}\left(h_{s-3}(\gamma \alpha)D''\left(1-\frac{h_{s-1}(\gamma \alpha)}{h_{s-1}(\gamma)}\right) - \frac{h_{s-2}(\gamma \alpha)^2}{h_{s-1}(\gamma)} D'''\left(1-\frac{h_{s-1}(\gamma \alpha)}{h_{s-1}(\gamma)}\right) \right).
 \end{align*}
 Hence,
  \begin{align}
&& \phi(0)=0 && \phi(1)=-\frac{1}{d}D'(0)\le 0 && \label{f2}\\
 && \phi'(0)=-1 && \phi'(1)=-1+\frac{\gamma h_{s-2}(\gamma)}{d h_{s-1}(\gamma)}D''(0).&& \label{f'2}
 \end{align}
 
 As before, we will prove that $\phi''(\alpha)$ has at most 1 root in $[0,1]$ (except for some trivial cases that will be discussed separately), which is sufficient to ensure (a) and (b).
 
 \medskip

 \noindent {\em Case B1: $\bfk\sim \po_{\ge s}(\gamma)$ and $\Var(\bfd)=0$.} In this case $\bfd=d$. Then $D(\alpha)=\alpha^d$. If $d\ge 3$ then
 solving $\phi''(\alpha)=0$ yields
 \begin{equation}
 \frac{d-2}{h_{s-1}(\gamma)} \cdot h_{s-2}(\gamma\alpha) = \left(1-\frac{h_{s-1}(\gamma \alpha)}{h_{s-1}(\gamma)}\right) \frac{h_{s-3}(\gamma\alpha)}{h_{s-2}(\gamma\alpha)}. \label{B1}
 \end{equation}
 On the right hand side above,
 $1-h_{s-1}(\gamma \alpha)/h_{s-1}(\gamma)\ge 0$ and is a decreasing function of $\alpha$. 
 We also have
 \[
 \frac{h_{s-3}(\gamma\alpha)}{h_{s-2}(\gamma\alpha)}=\frac{h_{s-3}(\gamma\alpha)}{h_{s-3}(\gamma\alpha)-(\gamma\alpha)^{s-3}/(s-3)!}=\left(1-\frac{(\gamma\alpha)^{s-3}/(s-3)!}{h_{s-3}(\gamma\alpha)}\right)^{-1},
 \]
 which is positive and a decreasing function of $\alpha$. Hence, the left hand side of~\eqn{B1} is an increasing function whereas the right hand side is a decreasing function. Hence $\phi''(\alpha)$ has at most one root.
 
 If $d\le 2$ the same argument as in Case A1 shows that claims (a) and (b) hold.
\smallskip

 \noindent {\em Case B2: $\bfk\sim \po_{\ge s}(\gamma)$ and $\bfd\sim \po_{\ge r}(\lambda)$.} In this case $D(\alpha)=h_r(\lambda\alpha)/h_r(\lambda)$, and $\lambda$ necessarily satisfies~\eqn{lambda}.
  Then solving $\phi''(\alpha)=0$ yields
  \[
  \frac{\lambda}{h_{s-1}(\gamma)} h_{s-2}(\gamma\alpha) = \frac{h_{s-3}(\gamma\alpha)}{h_{s-2}(\gamma\alpha)} \cdot  \frac{h_{r-2}(\lambda(1-h_{s-1}(\gamma \alpha)/h_{s-1}(\gamma)))}{h_{r-3}(\lambda(1-h_{s-1}(\gamma \alpha)/h_{s-1}(\gamma)))}
  \]
 The left hand side is an increasing function whereas the right hand side is the product of two positive decreasing functions. Thus, $\phi''(\alpha)$ has at most one root.

\section{Proof of \Thm~\ref{thm:2core}}\label{Sec_thm:2core}\label{sec:2core}

 \noindent {\em Proof of Theorem~\ref{thm:2core}.\ } We describe how to extend the proof of~\cite{molloy2005cores} to $\bfG$. Using the terminology in~\cite{molloy2005cores}, variable nodes in $\bfG$ are called vertices, and each check node corresponds to a hyperedge in the following sense: if $f_a$ is a check node adjacent with variable nodes $\{v_{a_1},\ldots, v_{a_h}\}$ for some $h\ge 3$, then the set of vertices $\{v_{a_1},\ldots, v_{a_h}\}$ is called a hyperedge. Consider the parallel stripping process where all vertices of degree less than 2 are deleted in each step, together with the hyperedge (if any) incident with it. Take a random vertex $v\in[n]$.  Let $\lambda_t$ be the probability that $v$ survives after $t$ iterations of the stripping process. It is easy to see that $\lambda_t$ is monotonely non-increasing and thus $\lambda=\lim_{t\to\infty} \lambda_t$ exists. For any vertex $u\in [n]$, let $\N_j(u)$ denote the set of vertices of distance $j$ from $u$. We claim that
 \begin{claim} \label{c:neighbourhood}
With high probability, the maximum degree and the maximum size of hyperedges in $\bfG$ is at most \newline
$(n\log n)^{1/(2+r)}$, and for every $u\in[n]$ and for all fixed $r$, $|\cup_{j\le r}\N_j(u)|=O(n^{1/(2+r)}\log^2 n)$. 
 \end{claim}
Let $H_t$ be the subgraph of $G_n$ obtained after $t$ iterations of the parallel stripping process. Consider Doob's martingale $(\ex(H_t\mid e_1,\ldots, e_j))_{0\le j\le m}$ where random hyperedges are added in order $e_1,\ldots, e_{m}$ using the configuration model. By Claim~\ref{c:neighbourhood}, swapping two clones in the configuration model would affect $H_t$ by $O(n^{1/(2+r)}\log^2 n)$, as each altered hyperedge can only affect the vertices (if surviving the first $t$-th iteration or not) within its $t$-neighbourhood.  Standard concentration arguments of Azuma's inequality~\cite[Theorem 2.19]{Regular} (with Lipschitz constant $Cn^{1/(2+r)}\log^2 n$ for some fixed $C>0$) produce that $||H_t|-\lambda_t n|=O(n^{(4+r)/(4+2r)}\log^3 n)=o(n)$. Next we deduce an expression for $\lambda_t$.  Consider a random hypertree $T$ iteratively built as follows. The root of $T$ is $v$, which is incident to $d_v$ hyperedges of size $\bfk_1,\ldots, \bfk_{d_v}$ where the $\bfk_i$s are i.i.d.\ copies of $\hat\bfk$ where
 \begin{equation}\label{hatk}
 \pr(\hat\bfk=j) = \frac{j \pr(\bfk=j)}{k}.
 \end{equation}
Then the $i$-th hyperedge is incident to other $\bfk_i-1$ vertices (other than $v$) whose degrees are i.i.d.\ copies of $\hat\bfd$, where 
 \begin{equation}\label{hatd}
 \pr(\hat\bfd=j) = \frac{j \pr(\bfd=j)}{d}.
 \end{equation}
 This builds the first neighbourhood of $v$ in $T$. Iteratively we can build the $r$-neighbourhood of $v$ in $T$ for any fixed $r$. It follows from the following claim that the $r$-neighbourhood of $v$ in $\bfG$ converges in distribution to the $r$-neighbourhood of $T$, as $n\to\infty$, for any fixed $r\ge 1$. This is because when uniformly picking a random variable clone (or check clone), the degree of the corresponding variable node (or check node) has the distribution in~\eqn{hatd} (or~\eqn{hatk}).
  
 \begin{claim}~\label{c:cycle}
 With high probability, for all fixed $r\ge 1$, $\cup_{j\le r}\N_j(v)$ induces no cycles.
 \end{claim}
 
   If $v$ survives $t$ iterations of the stripping process then at least 2 hyperedges incident with $v$ survives after $t$ iterations of the stripping process.  Let $\rho_t$ denote the probability that $v$ is incident with at least 1 hyperedge surviving after $t$ iterations of the stripping process. Then, ignoring an $o(1)$ error accounting for the probability of the complement of the evens in Claims~\ref{c:neighbourhood} and~\ref{c:cycle}:
 \[
 \rho_0=1
 \]
and
 \begin{align*}
 \rho_{t+1}&=\sum_{j\ge 1} \frac{j\pr(\bfd=j)}{d} \sum_{S\subseteq[j-1], |S|\ge 1} \sum_{k_1,\ldots,k_{j-1}} \prod_{i=1}^{j-1} \pr(\hat\bfk_i=k_1) \prod_{i\in S} \rho_t^{k_i-1} \prod_{i\notin S} (1-\rho_t)^{k_i-1}\\  
 &=\sum_{j\ge 1} \frac{j\pr(\bfd=j)}{d} \sum_{h\ge 1} \binom{j-1}{h} \left(\sum_{k_1}\pr(\hat\bfk_i=k_1) \rho_t^{k_1-1} \right)^h \left(\sum_{k_1}\pr(\hat\bfk_i=k_1) (1-\rho_t)^{k_1-1}\right)^{j-1-h}\\
 &=\sum_{j\ge 2} \frac{j\pr(\bfd=j)}{d} \sum_{h\ge 1}\binom{j-1}{h} \left(\frac{K'(\rho_t)}{k}\right)^h\left(1-\frac{K'(\rho_t)}{k}\right)^{j-1-h}\\
 &=\sum_{j\ge 2} \frac{j\pr(\bfd=j)}{d}\left(1-\left(1-\sum_{j\ge 2} \frac{j\pr(\bfd=j)}{d}\right)^{j-1}\right)=1-\frac{D'(1-\frac{K'(\rho_t)}{k})}{d},
 \end{align*}
 noting that 
 \[
 \ex \rho^{\hat\bfk-1}=\sum_j \frac{j\pr(\bfk=j)}{k}\rho^{j-1}=\frac{K'(\rho)}{k}.
 \]
 Similarly,
  \begin{align*}
 \lambda_t &= \sum_{j\ge 2} \pr(\bfd=j) \sum_{h\ge 2}\binom{j}{h} \left(\sum_{k_1}\pr(\hat\bfk_i=k_1) \rho_t^{k_1-1} \right)^h \left(\sum_{k_1}\pr(\hat\bfk_i=k_1) (1-\rho_t)^{k_1-1}\right)^{j-h}\\
 &= \sum_{j\ge 2} \pr(\bfd=j) \sum_{h\ge 2}\binom{j}{h} \left(\ex \rho_t^{\hat\bfk-1}\right)^h \left(1-\ex \rho_t^{\hat\bfk-1}\right)^{j-h}\\
 &= \sum_{j\ge 2} \pr(\D=j) \left(1-\left(\frac{K'(\rho_t)}{k}\right)^j-j\frac{K'(\rho_t)}{k}\left(1-\frac{K'(\rho_t)}{k}\right)^{j-1}\right)\\
 &=1-D\left(1-\frac{K'(\rho_t)}{k}\right)-\frac{K'(\rho_t)}{k}D'\left(\frac{K'(\rho_t)}{k}\right).
 \end{align*}
 Let $g(x)=1-\frac{1}{d}D'(1-\frac{K'(x)}{k})$. Then $g'(x)=\frac{1}{dk}D''(1-\frac{K'(x)}{k})K''(x)$ which is non-negative over $[0,1]$. We also have $\phi(x)=g(x)-x$, where $\phi$ is given in~\eqn{def:f}.
Since $\phi(1)=-D'(0)/d\le 0$,  $\phi'(\rho)<0$ by the hypothesis, and $g(x)$ is nondecreasing in $[0,1]$, it follows that $\rho$ is an attractive fix point of $x=g(x)$. As $\rho_0=1$. It follows that $\rho_t\to\rho$ as $t\to \infty$. Consequently, for every $\hat\eps>0$ there is sufficiently large $I$ such that $|\rho_t-\rho|<\hat\eps$. Hence, after $I$ iterations of the parallel stripping process, the number of vertices remaining is $(\lambda+o(1)) n+O(\hat\eps n)$ where 
\begin{equation}
\lambda=1-D\left(1-\frac{K'(\rho)}{k}\right)-\frac{K'(\rho)}{k}D'\left(\frac{K'(\rho)}{k}\right).\label{lambda}
\end{equation}
 If $\rho=0$ then $\lambda=0$. The case $\rho=0$ of our theorem follows by letting $I\to\infty$.

Suppose $\rho>0$. It is sufficient to show that the 2-core is obtained after removing further $O(\hat\eps n)$ vertices, following the same approach as~\cite[Lemma 4]{molloy2005cores}. We briefly sketch it. Following the same argument as before, the probability that a random vertex has degree $j\ge 2$ after $I$ iterations of the stripping process is
\[
\sum_{i\ge 2} \pr(\bfd=i) \binom{i}{j} (\ex \rho_I^{\hat\bfk-1})^j (1-\ex \rho_I^{\hat\bfk-1})^{i-j}= \sum_{i\ge 2} \pr(\bfd=i) \binom{i}{j}\left(\frac{K'(\rho_I)}{k}\right)^j\left(1-\frac{K'(\rho_I)}{k}\right)^{i-j}.
\] 
Similarly, the probability of a uniformly random hyperedge in $G_n$ having size $j\ge 3$ and surviving the first $I$ iterations of the stripping process is
\[
\pr(\bfk=j) \rho_I^j.
\]
The number of vertices with degree less than 2 after $I$ iterations is bounded by $(\lambda_I-\lambda_{I+1}) n +o(n)$. Hence, by choosing $I$ sufficiently large, we can make these quantities arbitrarily close to those with $\rho_I$ replaced by $\rho$. Now standard concentration arguments apply to show that the number of degree $j\ge 2$ vertices is 
$\gamma_j n+O(\hat\eps n)$, where
\[
\gamma_j=\sum_{i=0}^{\infty} \pr(\bfd=i) \binom{i}{j}\left(\frac{K'( \rho)}{k}\right)^j\left(1-\frac{K'( \rho)}{k}\right)^{i-j},
\] 
the number of vertices of degree less than 2 is $O(\hat\eps n)$, and the proportion of remaining hyperedges of size $j$ is \[
\frac{\pr(\bfk=j) \rho^j }{\sum_{i\ge 3} \pr(\bfk=i) \rho^i}+O(\hat\eps)=\frac{\pr(\bfk=j) \rho^j}{K(\rho)}+O(\hat\eps),
\] 
as the probability that a hyperedge survives is proportional to the probability that all of the vertices it contains survive.
Note that $\hat\eps$ can be made arbitrarily small by choosing sufficiently large $I$.

 Now we remove one hyperedge incident with a vertex with degree 1 at a time. Call this process SLOWSTRIP. Let $G_t$ denote the hypergraph obtained after $t$ steps of SLOWSTRIP and let $X_t$ denote the total degree of the vertices of degree 1 in $G_t$. Then,
 \begin{align*}
 &\ex(X_{t+1}-X_t\mid G_t)\\
  &= -1 + \sum_{j\ge 3} \frac{j\pr(\bfk=j) \rho^j/K(\rho)}{\rho K'(\rho)/K(\rho)} \cdot (j-1) \cdot \frac{2\gamma_2}{\sum_{j\ge 2} j \gamma_j} +O(\hat\eps)\\
 &=-1+\frac{1}{\rho K'(\rho)}\left( \sum_{j\ge 3} j(j-1) \pr(\bfk=j) \rho^j\right) \frac{2\cdot \frac12 (K'(\rho)/k)^2 D''(1-K'(\rho)/k)}{K'(\rho)d\rho/k}+O(\hat\eps)\\
 &=-1+\frac{D''(1-K'(\rho)/k)K''(\rho)}{kd}+O(\hat\eps).
 \end{align*}
 Note that in the first equation above, $-1$ accounts for the removal of one variable clone $x$ from the set of vertices of degree less than 2. The term $j\pr(\bfk=j) \rho^j/\rho K'(\rho)$ approximates the probability that $x$ is contained in a hyperedge of size $j$, up to an $O(\hat\eps)$ error. In that case, $j-1$ variable clones that lie in the same hyperedge as $x$ will be removed. For each of these $j-1$ deleted variable clones, if it lies in a variable of degree 2, then it results in one new variable node of degree 1. The probability for that to happen is approximated by $2\gamma_2/\sum_{j\ge 2}j\gamma_j$, up to an $O(\hat\eps)$ error.
 By the assumption that $f'(\rho)<0$ we have 
 \[
 -1+\frac{D''(1-K'(\rho)/k)K''(\rho)}{kd}<0.
 \]
 Hence, $ \ex(X_{t+1}-X_t\mid G_t)<-\delta$ for some $\delta>0$, by making $\hat\eps$ sufficiently small (i.e.\ by choosing sufficiently large $I$). Then standard Azuma inequality~\cite[Lemma 29]{GaoMolloy} (with Lipschitz constant $n^{1/5}\log n$ by Claim~\ref{c:neighbourhood}) will be sufficient to show that $X_t$ decreases to 0 after $O(\hat\eps n)$ steps (See details in~\cite[Lemma 4]{molloy2005cores}). The case $\rho>0$ of the theorem follows by~\eqn{lambda} and 
 \[
 \lim_{n\to\infty} \frac{\bfm^*}{n} =  \lim_{n\to\infty} \frac{m}{n} \cdot \sum_{i\ge 3} \pr(\bfk=i) \rho^i = \frac{d}{k}K(\rho).\qed
 \]  
 
\noindent {\em Proof of Claim~\ref{c:neighbourhood}. } Since both $\ex\bfd^{2+r}=O(1)$ and $\ex\bfk^{2+r}=O(1)$, the probability that $\bfd>(n\log n)^{1/(2+r)}$ or $\bfk>(n\log n)^{1/(2+r)}$ is $O(1/n\log n)$. The bound on the maximum degree and maximum size of the hyperedges in $\bfG$ follows by taking the union bound. 

For any $u\in[n]$, let $N_i(u)=|\N_i(u)|$. We will prove that with high probability for every $u$ and for every fixed $i$, $N_i(u)=O(n^{1/(2+r)}\log^2 n)$, which then completes the proof for Claim~\ref{c:neighbourhood}. We prove by induction. Let $d_1,\ldots, d_{N_i(u)}$ denote the degrees of the vertices in $\N_i(u)$.  Then the number of hyperedges incident with these vertices is bounded by $M:=\sum_{j=1}^{N_i(u)} d_j$. By the construction of $\bfG$, each $M$ is stochastically dominated by $\sum_{j=1}^{N_i(u)} (1+o(1)) \hat\bfd_j$ where $\hat\bfd_j$ are i.i.d.\ copies of $\hat\bfd$ whose distribution is given in~\eqn{hatd}. The $o(1)$ error is caused by the exposure of $\cup_{j\le i}\N_j$ which contains $o(n)$ vertices by induction. Since $\ex \bfd^{2+r}=O(1)$, we have $\hat d :=\ex \hat\bfd =O(1)$. Note that $\ex M=\hat d N_i(u)$. Applying the Chernoff bound to the sum of independent $[0,1]$-valued random variables we have
\[
\pr\left(M\ge 2\hat d N_i(u)+ n^{1/(2+r)}\log^2 n \right) = \pr\left(\sum_{j=1}^{N_i(u)} \frac{\hat \bfd_i}{(n\log n)^{1/(2+r)} } \ge \frac{2\hat d N_i(u)}{(n\log n)^{1/(2+r)}} + (\log n)^{(3+r)/(2+r)}\right) <n^{-2}.
\]
Similarly, $N_{i+1}(u)$ is bounded by $\sum_{j=1}^M k_i$, where $k_i$ are the sizes of the hyperedges incident with the vertices in $\N_i$. Similarly, $\sum_{j=1}^M k_i$ is stochastically dominated by $(1+o(1))\sum_{j=1}^M \hat\bfk_j $ where $\hat\bfk_j$ are i.i.d.\ copies of $\hat\bfk$ whose distribution is defined in~\eqn{hatk}. Let $\hat k=\ex\hat\bfk$. Applying the Chernoff bound again we obtain that with probability at least $1-n^{-2}$, $N_{i+1}(u)< 2\hat k M +n^{1/(2+r)}\log^2 n < 4\hat d\hat k N_i(u) +(1+2\hat k)n^{1/(2+r)}\log^2 n$. Apply this recursion inductively and the union bound on the failure probability, we obtain $N_i(u)=O(n^{1/(2+r)}\log^2 n)$, as desired. \qed\smallskip

\noindent {\em Proof of Claim~\ref{c:cycle}. } Fix $\eps>0$. Choose $L=L(\eps,r)$  sufficiently large so that the probability that $d_v>L$ is smaller than $\eps$ (note that $v$ is a uniformly random vertex). Given $d_v\le L$. Let $k_1,\ldots, k_{d_v}$ be the sizes of the hyperedges incident to $v$. Similarly to the proof of Claim~\ref{c:neighbourhood},  $k_j$s are approximated by i.i.d.\ copies of $\hat\bfk$ defined in~\eqn{hatk}, up to an $1+o(1)$ multiplicative error. We can assume $L$ is sufficiently large so that with probability at least $1-\eps$, $\sum_{i=1}^{d_v} k_i\le L$. Inductively, we can make $L$ sufficiently large so that $|\N_i(v)|\le L$ for all $i\le r$.    Let $\E_i$ denote the set of hyperedges incident with vertices in $\N_i(v)$, but not incident with any in $\N_{i-1}$. Cycles in $\N_i(v)$ can appear in two ways: (a) two vertices in $\N_i(v)$ are incident with the same hyperedge in $\E_i$; (b) two hyperedges in $\E_{i-1}$ are incident with the same vertex in $\N_i(v)$. We will prove that with high probability, none of the two cases occurs for any fixed $i$. For (a), let $(d_j)_{j\in \N_i(v)}$ denote the degrees of the vertices in $\N_i(v)$. The expected number of occurrences of pairs of vertices in (a) is 
\begin{equation}\label{pairs}
\ex\left(\sum_{j, k\in \N_i(v)} \binom{d_j}{2} \binom{d_k}{2} \sum_{h\in [\bfm]} \binom{k_h}{2} O(n^{-2}) \right)= O(n^{-1}) \ex\left(\sum_{j,k\in \N_i(v)} d_j^2 d_k^2\right).
\end{equation}
Note that $|\N_j(v)|\le L$ for each $j\le r$. This immediately implies that $d_j\le L$ for all $j\in \N_i(v)$. Hence, the above probability is $O(n^{-1})$.  The probability that $|\N_i(v)|\le L$ fails is at most $r\eps$ by our choice of $L$. Hence, the probability that (a) fails is at most $r\eps+o(1)$.
The treatment of (b) is analogous. Our claim now follows by letting $\eps\to 0$. \qed

\subsection*{Acknowledgment}
We thank David Saad for a helpful conversation and Guilhem Semerjian for bringing~\cite{Lelarge} to our attention.

\begin{appendix}

\section{Proof of \Lem~\ref{Lemma_sums}}\label{Sec_sums}

\noindent
Since $\Erw[\vec\lambda^r]<\infty$, the event $\cM=\cbc{\max_{i\in[s]}\vec\lambda_i\leq n/\ln^9n}$ has probability
\begin{align}\label{eqLemma_sums0}
\pr\brk{\cM}=1-o(1/n).
\end{align}
Moreover, fixing a small enough $\eta=\eta(\delta)>0$ and a large enough $L=L(\eta)>0$ and setting $Q_j=\sum_{i\in[s]}\vecone\{\lambda_i=j\}$, we obtain from the Chernoff bound that 
$\pr\brk{\forall j\leq L:|Q_j-s\pr\brk{\vec\lambda=j}|>\sqrt n\ln n}=o(1/n)$.
Hence, by Bayes' rule,
\begin{align}\label{eqLemma_sums1}
\pr\brk{\exists j\leq L:|Q_j-s\pr\brk{\vec\lambda=j}|>\sqrt n\ln n\mid\cM}&=o(1/n).
\end{align}
In addition, let $\cH=\cbc{h\in\NN:(1+\eta)^{h-1}L\leq n/\ln^9n}$ and for $h\in\cH$ let
\begin{align*}
R_h&=\sum_{j\geq1}Q_j\vecone\{L(1+\eta)^{h-1}<j\leq L(1+\eta)^{h}\wedge n/\ln^9n\},&
\bar R_h&=s\sum_{j\geq1}\pr\brk{\vec\lambda=j}\vecone\{L(1+\eta)^{h-1}<j\leq L(1+\eta)^{h}\wedge n/\ln^9n\}.
\end{align*}
Then the Chernoff bound and Bayes' rule yield
\begin{align}\label{eqLemma_sums2}
\pr\brk{\exists h\in\cH:\abs{R_h-\bar R_h}>\eta\bar R_h+\ln^2n\mid\cM}&=o(1/n).
\end{align}
Finally, given $\cM$ and $|Q_j-s\pr\brk{\vec\lambda=j}|\le\sqrt n\ln n$ for all $j\leq L$ and 
$\abs{R_h-\bar R_h}\leq\eta\bar R_h+\ln^2n$ for all $h\in\cH$, we obtain
\begin{align*}
\frac1s\sum_{i=1}^s\vec\lambda_i&\leq \sum_{j=1}jQ_j/s+\sum_{h\in\cH}(1+\eta)^hLR_h/s=o(1)+\Erw\brk{\vec\lambda\vecone\{\vec\lambda\leq L\}}
	+\sum_{h\in\cH}(1+\eta)^{h+1}(\bar R_h+(\ln^2n))/s\leq \Erw\brk{\vec\lambda\vecone}+\delta/2+o(1).
\end{align*}
Similarly,  $\frac1s\sum_{i=1}^s\vec\lambda_i\geq \Erw\brk{\vec\lambda\vecone}-\delta/2+o(1)$.
Thus, the assertion follows from \eqref{eqLemma_sums0}--\eqref{eqLemma_sums2}

\section{Proof of \Lem~\ref{Lemma_0pinning}}\label{Apx_0pinning}

\noindent
For $t=1,\ldots,\THETA$ let $S_0(t),\ldots,S_{\ell(t)}(t)$ be the decomposition of $[n]$ obtained by applying \Lem~\ref{Lemma_Spartition}
to the matrix $A[\vec i_1,\ldots,\vec i_t]$.
Further, let $R_h(t)=S_h(t)\cap U$ for all $t$.
We proceed in two steps.
First we establish a deterministic statement:
\begin{align}\label{eqLemma_0pinning1}
\mbox{if $\sum_{h\geq1}|R_h(t)|^2<\eps|U|^2$, then $\mu_{\hat A,U}\mbox{ is $\eps$-symmetric}$.}
\end{align}
Indeed, using (i) and (iii) from \Lem~\ref{Lemma_Spartition}, we obtain
\begin{align*}
\sum_{u,v\in U}\dTV(\mu_{\hat A,u,v},\mu_{\hat A,u}\tensor \mu_{\hat A,v})&=
	\sum_{h=1}^{\ell(t)}\sum_{u,v\in R_h(t)}\dTV(\mu_{\hat A,u,v},\mu_{\hat A,u}\tensor \mu_{\hat A,v})
	\leq\sum_{h=1}^{\ell(t)}|R_h(t)|^2, %\leq|U|\max_{h\geq1}|R_h(t)|,
\end{align*}
whence \eqref{eqLemma_0pinning1} follows.

Second, we claim that for large enough $\Theta=\Theta(\eps)$,
\begin{align}\label{eqLemma_0pinning2}
\Erw\sum_{h\geq1}|R_h(\THETA)|^2<\eps^3|U|^2.
\end{align}
To see this, consider the random variables $\Delta_t=|R_0(t)|-|R_0(t-1)|$ for $t\in[\Theta]$.
Since $S_0(t-1)\subset S_0(t)$, these random variables are non-negative.
Moreover, we claim that
\begin{align}\label{eqLemma_0pinning3}
\sum_{t\in[\Theta]}\Delta_t&\leq|U|,&\Erw[\Delta_t\mid \vec i_1,\ldots,\vec i_{t-1}]&\geq\sum_{h\geq1}|R_h(t-1)|^2/|U|.
\end{align}
The first inequality is self-evident.
To obtain the second, we deduce from \Lem~\ref{Lemma_Spartition} (ii) that freezing any variable in class $R_h(t-1)$ freezes the entire class.
In other words, if $\vec i_t\in R_h(t-1)$, then $R_h(t-1)\subset R_0(t)$.
Since $\vec i_t$ is chosen uniformly at random, we obtain the second inequality.
Combining the two inequalities from \eqref{eqLemma_0pinning3}, we find
\begin{align*}
\Erw\brk{\sum_{h\geq1}|R_h(\THETA)|^2}&=
	\frac1{\Theta}\sum_{t=1}^{\Theta}\Erw\brk{\sum_{h\geq1}|R_h(t)|^2}
	\leq\frac{|U|^2}{\Theta}+\frac1{\Theta}\sum_{t=1}^{\Theta}\Erw\brk{\sum_{h\geq2}|R_h(t-1)|^2}\leq
		\frac{|U|^2}{\Theta}+\frac{|U|}{\Theta}\Erw\brk{\sum_{t=1}^{\Theta}\Delta_t}\leq\frac{2|U|^2}{\Theta},
\end{align*}
whence \eqref{eqLemma_0pinning2} follows.
Finally, the assertion follows from~\eqref{eqLemma_0pinning1} and \eqref{eqLemma_0pinning2}.

\end{appendix}

\end{document}